\documentclass[10pt,letterpaper]{amsart}

\usepackage{amssymb}
\usepackage{cite}
\usepackage{etoolbox}
\usepackage[hidelinks]{hyperref}
\hypersetup{
  pdftitle={Berry--Esseen Bounds for the Number of Real Zeros of Gaussian Weyl Polynomials},
  pdfauthor={Yuchen Wang, Dawei Lu and Song-Hao Liu},
  pdfsubject={},
  pdfkeywords={random Weyl polynomial, stationary Gaussian process, real zeros, Kac--Rice formula, local dependence, Berry--Esseen bound}
}

\makeatletter
\patchcmd{\andify}{\unskip, \@@and~}{\unskip{} \@@and~}{}
  {\PackageError{manuscript}{Unable to adjust running-head author punctuation}{}}
\patchcmd{\author@andify}{\unskip ,\penalty-2}{\unskip{}\penalty-2}{}
  {\PackageError{manuscript}{Unable to adjust title author punctuation}{}}
\makeatother

\newtheorem{theorem}{Theorem}[section]
\newtheorem{lemma}[theorem]{Lemma}
\newtheorem{proposition}[theorem]{Proposition}
\newtheorem{corollary}[theorem]{Corollary}

\theoremstyle{remark}
\newtheorem{remark}[theorem]{Remark}

\numberwithin{equation}{section}

\newcommand{\E}{\mathbb E}
\newcommand{\Pp}{\mathbb P}
\newcommand{\R}{\mathbb R}
\newcommand{\Var}{\operatorname{Var}}
\newcommand{\Cov}{\operatorname{Cov}}
\newcommand{\dd}{\,\mathrm d}

\begin{document}

\title[Berry--Esseen bounds for zeros of Weyl polynomials]
{Berry--Esseen Bounds for the Number of Real Zeros of
Gaussian Weyl Polynomials}

\author{Yuchen Wang}
\address{\textnormal{Yuchen Wang} School of Mathematical Sciences, Dalian University of Technology}
\email{\href{mailto:wycwyc@mail.dlut.edu.cn}{\nolinkurl{wycwyc@mail.dlut.edu.cn}}}

\author{Dawei Lu}
\address{\textnormal{Dawei Lu} School of Mathematical Sciences, Dalian University of Technology}
\email{\href{mailto:ludawei_dlut@163.com}{\nolinkurl{ludawei_dlut@163.com}}}

\author{Song-Hao Liu}
\address{\textnormal{Song-Hao Liu} School of Mathematical Sciences, Dalian University of Technology}
\email{\href{mailto:liusonghao@dlut.edu.cn}{\nolinkurl{liusonghao@dlut.edu.cn}}}


\date{}

\begin{abstract}
We establish Berry--Esseen bounds for the number of real roots of
Gaussian Weyl polynomials \(P_n\), where \(n\) denotes the degree and is
assumed to be sufficiently large.  For each fixed \(B\) above an absolute
threshold, let
\(I_n=[-\sqrt n+B\sqrt{\log n},\,\sqrt n-B\sqrt{\log n}]\).
Uniformly over deterministic compact intervals \(I\subseteq I_n\) whose length
\(\ell\) is sufficiently large and depends on \(n\), the
distribution of the standardized number of real roots in \(I\) has
 Kolmogorov distance at most \(C\log\ell/\sqrt\ell\) from the standard Gaussian
distribution.  Consequently, every such interval sequence with
\(\ell\to\infty\) satisfies a central limit theorem.  In particular, taking
\(I=I_n\) gives the bound \(C\log n/n^{1/4}\).  The same
bound holds for the distribution of the standardized number of real roots
on \(\mathbb R\).  The key idea is to
approximate the polynomial zero count by a sum of locally dependent random
variables.  We first couple the polynomial to a stationary Gaussian process
and then truncate a moving-average representation of that process to obtain
finite-range dependence.  This strategy provides a route to Berry--Esseen
bounds for other random polynomials with Gaussian coefficients whenever such
a stationary approximation and quantitative truncation are available.
\end{abstract}

\maketitle

\section{Introduction}
\label{sec:introduction}

\subsection{Background}

Many classical real random polynomials are finite random sums
\(\sum_{k=0}^{n}c_{n,k}\xi_kx^k\), where the \(\xi_k\) are random
variables, usually taken to be independent and identically distributed.  The
choices \(c_{n,k}=1\),
\(c_{n,k}=\binom nk^{1/2}\) and \(c_{n,k}=1/\sqrt{k!}\) correspond,
respectively, to the Kac, elliptic (Kostlan--Shub--Smale) and Weyl
(flat) models.  Different choices of the weights \(c_{n,k}\) lead to different
properties of the real zeros.  When the \(\xi_k\) are Gaussian, we use the
Kac--Rice formula to compute expected zero counts, whereas the analysis of
variances and quantitative limit laws requires control of two-point
correlations and dependence.

For a real polynomial \(P\) and a Borel set \(A\subset\R\), let
\(N_P(A)\) be the number of zeros of \(P\) in \(A\), counted with
multiplicity.  By convention, \(N_P(A)=0\) if \(P\equiv0\).  We write
\(N_n(A)=N_{P_n}(A)\), so \(N_n(\R)\) is the total real-zero count.

The probabilistic study of real roots has a long history.  Bloch and
P\'olya obtained a bound for the expected number of real zeros of random
polynomials with discrete coefficients \cite{BlochPolya1932} and the
foundational work
of Littlewood and Offord \cite{LittlewoodOfford1938} and Kac
\cite{Kac1943} established the subject as a systematic part of
probability theory.  Kac's calculation of expected real-zero counts and
Rice's level-crossing theory form the foundations of the modern Kac--Rice
formula \cite{Kac1943,Rice1944,Rice1945}.  Readers interested
in the Kac--Rice formula may consult \cite{AzaisWschebor2009}.
Cram\'er and Leadbetter give a classical stationary-process treatment
in \cite{CramerLeadbetter1967}.  For the Gaussian Kac model, the mean
total real-zero count is \(\frac2\pi\log n+O(1)\).  Maslova later proved
that its variance is asymptotic to
\(\frac4\pi(1-\frac2\pi)\log n\) and that the standardized count
 converges in distribution to a standard Gaussian random variable
\cite{Maslova1974Variance,Maslova1975Distribution}.  The progression
from expectation to variance and then to a limit law became a basic
template for the subject.  Recent work has added concentration and
almost-sure information.  Can and Nguyen proved concentration bounds for
Kac real-zero counts and Do established a local strong law under a
bounded \(2+\varepsilon\) moment assumption
\cite{CanNguyen2025,Do2024StrongLaw}.

A geometric viewpoint relates the classical flat, elliptic and
hyperbolic ensembles to Gaussian analytic functions associated with
the Euclidean plane, the Riemann sphere and the hyperbolic disk
\cite{HoughEtAl2009}.  Their real-coefficient polynomial counterparts
or truncations display sharply different scales.  Kac polynomials are
the finite truncations of the parameter-one hyperbolic series and have
logarithmically many real roots.  For the elliptic model, the expected
number of real zeros equals \(\sqrt n\), its variance is asymptotic to
\(c\sqrt n\) for an explicit constant \(c>0\) and the standardized zero
count satisfies a central limit theorem
\cite{AnconaLetendreKostlan2021,Dalmao2015}.  Armentano et al.\
established variance asymptotics and a central limit theorem for
higher-dimensional Kostlan--Shub--Smale systems
\cite{ArmentanoEtAl2018,ArmentanoEtAl2021}.  Do and Nguyen obtained
precise variance asymptotics for broad classes of generalized Kac and
hyperbolic polynomials, including their derivatives and models with
non-Gaussian coefficients \cite{DoNguyen2025}.  By contrast, the
natural spatial scale for the real zeros of the Weyl model is
\(O(\sqrt n)\).  In the Gaussian case, both \(\E N_n(\R)\) and
\(\Var(N_n(\R))\) have order \(\sqrt n\).  Recent comparison
principles also give sharp variance asymptotics and central limit
theorems for non-centered Kac and hyperbolic ensembles
\cite{DoNguyenORourke2026}.

Related models reinforce the role of the deterministic basis.
For Gaussian random trigonometric polynomials, Dunnage gave a classical
analysis of the real-zero count \cite{Dunnage1966}, Granville and Wigman
studied the zero distribution \cite{GranvilleWigman2011} and
Aza\"{\i}s, Dalmao and Le\'on established a central limit theorem
\cite{AzaisDalmaoLeon2016}.
Coutin and Peralta obtained a rate of convergence by comparing the
distribution of the number of zeros of a random trigonometric polynomial with
the corresponding distribution for a stationary Gaussian process with sinc
covariance \cite{CoutinPeralta2023}.  Their result is therefore different from
a Berry--Esseen bound, which concerns convergence to the standard Gaussian
distribution in Kolmogorov distance.
Do, Nguyen and Nguyen distinguish universal and
coefficient-dependent parts of the variance
\cite{DoNguyenNguyen2022}, whereas Nguyen and Zeitouni prove
exponential concentration \cite{NguyenZeitouni2024}.  A substantial
literature also studies random orthogonal polynomials.  Readers
interested in this topic may consult
\cite{DoEtAl2024,DoNguyenVuOrthogonal2023,
LubinskyPritsker2021,LubinskyPritskerXie2018}.
These developments connect Kac--Rice theory, stationary limits,
 orthogonal expansions and Gaussian approximation.

Universality with respect to the coefficient law provides another
organizing principle.  Tao and Vu proved through a replacement
principle that, under suitable moment, tail and non-concentration
assumptions, replacing the coefficient law by one that matches its
first two moments leaves local real and complex zero correlations
asymptotically unchanged \cite{TaoVu2015}.  Do, Nguyen and Vu
proved expectation and root-repulsion estimates for Kac polynomials
\cite{DoNguyenVu2015} and optimal local universality for ensembles
with polynomially growing weights \cite{DoNguyenVu2018}.  Nguyen and
Vu developed a unified framework that treats the Kac, Weyl, elliptic
and random trigonometric models on the same footing
\cite{NguyenVu2022} and they also proved
central limit theorems for broad polynomial-growth ensembles
\cite{NguyenVu2021}.  At the expectation level, recent work identifies
the first distribution-dependent constant term for Kac real-zero
counts under broad moment assumptions \cite{LamNguyen2025}.  These
results demonstrate the robustness of limiting zero statistics across
coefficient distributions.

Real-zero counts are nonlinear functionals of Gaussian processes and
geometric invariants of random equations.  They serve as test cases
for mixing, local dependence and quantitative central limit theory.
They also arise in random real algebraic geometry and describe the
level-crossing behavior of random signals
\cite{ArmentanoEtAl2018,ArmentanoEtAl2021,AzaisWschebor2009,
EdelmanKostlan1995}.

For the Gaussian Weyl ensemble, Do and Vu proved exact
square-root-order asymptotics for the mean and variance of the total
real-zero count, a central limit theorem for its standardized form
and corresponding variance and limit results for growing-window
linear statistics \cite{DoVu2020}.  Related work develops moment,
variance, concentration and limit theory for zeros of smooth
stationary Gaussian processes
\cite{AnconaLetendre2021,AssafBuckleyFeldheim2023,
BasuEtAl2020,Gass2023}, as well as cumulant and central-moment methods
for Gaussian analytic functions, including Weyl functions
\cite{Nguyen2024}.  Recent results also provide concentration
estimates for several Gaussian and non-Gaussian ensembles and
leading-order expectation and variance universality for Weyl
polynomials with general coefficients
\cite{AguirreNguyenWang2025,AguirreNguyenWangWeyl2025}.
Together, these results provide a substantial qualitative
understanding of Weyl zero counts, including the \(n^{1/4}\)
fluctuation scale and a Gaussian limit.

A natural next question is how rapidly the standardized zero count
 approaches the standard Gaussian law.  A classical way to measure this
rate is through a Berry--Esseen bound.  The pioneering works of Berry
and Esseen established uniform bounds for the Gaussian approximation
of sums of independent random variables
\cite{Berry1941,Esseen1942}.  In the i.i.d.\ setting with a finite
third absolute moment, the classical Berry--Esseen theorem gives an
\(O(n^{-1/2})\) error bound, uniform over all thresholds.

Drawing on the local-dependence Gaussian approximation of Chen and Shao
\cite{ChenShao2004}, we seek a quantitative Gaussian approximation for
Weyl zero counts, first on growing intervals contained in
\([-\sqrt n,\sqrt n]\) and separated from its endpoints by distances
proportional to \(\sqrt{\log n}\) and then on the entire real line.  The
separation is imposed because the real-zero density begins to change near
\(\pm\sqrt n\).

\subsection{Main results}

This paper studies quantitative Gaussian approximation for zero counts in the
Gaussian Weyl ensemble.  We work on one coefficient space carrying independent
standard Gaussian random variables \((\xi_k)_{k\geq0}\).  For every
integer \(n\geq0\), we define the Gaussian Weyl polynomial
\begin{equation}
 P_n(x)=\sum_{k=0}^{n}\frac{\xi_k}{\sqrt{k!}}x^k,
 \qquad x\in\R.
 \label{eq:weyl-polynomial}
\end{equation}
The corresponding real-zero counts are Borel-measurable functions of the
coefficient vector.
Section~\ref{sec:preliminaries} introduces the full Weyl series and its
stationary normalization.

For the total real-zero count \(N_n(\R)\), the precise asymptotics of Do and
Vu are
\begin{equation}
 \lim_{n\to\infty}
 \frac{\E N_n(\R)}{\sqrt n}
 =
 \frac2\pi,
 \qquad
 \lim_{n\to\infty}
 \frac{\Var(N_n(\R))}{\sqrt n}
 =
 2V_\infty,
 \label{eq:intro-known-Weyl-results}
\end{equation}
where \(V_\infty=0.18198\ldots>0\) is the asymptotic variance per unit
length of the stationary Weyl zero count.  We use these limits only to
normalize the zero counts.  We obtain Berry--Esseen bounds, first for
zero counts on growing subintervals of \(I_n\) and then for
\(N_n(\R)\).

For \(B>0\) and all sufficiently large \(n\), depending on \(B\), we
set
\begin{equation}
 L_n=\sqrt n-B\sqrt{\log n},
 \qquad
 I_n=[-L_n,L_n].
 \label{eq:logarithmic-edge-interval}
\end{equation}
Unless stated otherwise, \(I\) denotes a deterministic compact interval and
\(\ell=|I|\) denotes its length.  When both endpoints are named explicitly,
we write \(I=[\varphi,\psi]\).  In
Theorem~\ref{thm:polynomial-logarithmic-edge-BE}, we assume
\(I\subseteq I_n\).
We apply Lemma~\ref{lem:appendix-do-vu-thm5} with
\(h=\mathbf 1_{[-1,1]}\) and window sequence \(R_n=L_n\) and obtain,
for every fixed
\(B>0\),
\begin{equation}
 \lim_{n\to\infty}
 \frac{\Var(N_n(I_n))}{\sqrt n}
 =
 2V_\infty.
 \label{eq:intro-growing-window-variance}
\end{equation}
We let \(Z\) be a standard Gaussian random variable.  Its
distribution function is
\begin{equation}
 \Phi(z)=
 \frac{1}{\sqrt{2\pi}}
 \int_{-\infty}^{z}e^{-u^{2}/2}\dd u,
 \qquad z\in\R.
 \label{eq:standard-Gaussian-distribution-function}
\end{equation}
For a real random variable \(U\), we define
\begin{equation}
 d_{\mathrm K}(U,Z)
 =
 \sup_{z\in\R}|\Pp(U\leq z)-\Phi(z)|.
 \label{eq:Kolmogorov-distance-definition}
\end{equation}
Our two principal conclusions are as follows.  The first gives a quantitative
Gaussian approximation for \(N_n(I)\), uniformly over deterministic intervals
\(I\subseteq I_n\) and identifies the corresponding variance density.

\begin{theorem}
\label{thm:polynomial-logarithmic-edge-BE}
There are finite absolute constants \(B_0,C>0\) such that the following holds.
For every fixed \(B\geq B_0\) and every deterministic sequence
\(r_n\to\infty\) satisfying \(r_n\leq|I_n|\) eventually, for all sufficiently
large \(n\), we have
\begin{equation}
 \sup_{\substack{I\subseteq I_n\\
                  |I|\geq r_n}}
 \frac{\sqrt{|I|}}{\log |I|}\,
 d_{\mathrm K}\!\left(
 \frac{N_n(I)-\E N_n(I)}
  {\sqrt{\Var(N_n(I))}},
  Z
 \right)
 \leq
 C
 \label{eq:polynomial-logarithmic-edge-BE}
\end{equation}
and
\begin{equation}
 \sup_{\substack{I\subseteq I_n\\
                  |I|\geq r_n}}
 \left|
  \frac{\Var(N_n(I))}{|I|}-V_\infty
 \right|
 \longrightarrow0.
 \label{eq:polynomial-bulk-subinterval-variance-asymptotic}
\end{equation}
\end{theorem}

Since \(|I_n|/\sqrt n\to2\), we set \(r_n=|I_n|\) in
\eqref{eq:polynomial-logarithmic-edge-BE} and obtain the bound
\(C(\log n)n^{-1/4}\) for the Kolmogorov distance between the distribution of
the standardized version of \(N_n(I_n)\) and that of \(Z\).  We use this
estimate in the proof of
Theorem~\ref{thm:full-line-polynomial-BE}.

The second conclusion extends the quantitative Gaussian approximation to the
full real line.

\begin{theorem}
\label{thm:full-line-polynomial-BE}
There is a finite absolute constant \(C\) such that, for all
sufficiently large \(n\),
\begin{equation}
 d_{\mathrm K}\!\left(
  \frac{N_n(\R)-\E N_n(\R)}
  {\sqrt{\Var(N_n(\R))}},
 Z
 \right)
 \leq
 C\frac{\log n}{n^{1/4}}.
 \label{eq:full-line-polynomial-BE}
\end{equation}
\end{theorem}

\begin{remark}
For each fixed \(B\geq B_0\), the same conclusion holds with \(\R\) replaced
by any deterministic interval \(J=J(n)\supseteq I_n\).
Indeed, the same edge and exterior estimates apply to \(J\setminus I_n\),
with any finite endpoints outside \([-\sqrt n,\sqrt n]\) handled by the same
Gaussian small-ball argument.  Thus the zeros in
\(J\setminus I_n\) are negligible on the \(n^{1/4}\) standard-deviation
scale.  We state the
theorem on \(\R\) because the total number of real zeros of a polynomial is
usually the quantity of primary interest.
\end{remark}

The proof of Theorem~\ref{thm:full-line-polynomial-BE} uses
Theorem~\ref{thm:polynomial-logarithmic-edge-BE} with \(r_n=|I_n|\).  The
logarithmic factor in \eqref{eq:polynomial-logarithmic-edge-BE} comes from
the square of the finite dependence range \(O(\sqrt{\log\ell})\).  With high
probability, the zero counts on \(I_n\)
and on \(\R\setminus[-\sqrt n,\sqrt n]\) agree with two zero counts that
depend on disjoint sets of coefficients and are therefore independent.  The
second of these two counts has variance \(O((\log n)^2)\), whereas
\(\Var(N_n(I_n))\) has order \(\sqrt n\).  Hence adding the zero count outside
\([-\sqrt n,\sqrt n]\) contributes \(O((\log n)^2/\sqrt n)\) to the
Kolmogorov error.  Adding
\(N_n([-\sqrt n,\sqrt n]\setminus I_n)\) contributes
\(O((\log n)/n^{1/4})\), which has the same order as the bound for
\(N_n(I_n)\).  Hence the final rate remains
\(O((\log n)/n^{1/4})\).
We expect that the Kolmogorov bound in
Theorem~\ref{thm:polynomial-logarithmic-edge-BE} can be improved to order
\(\ell^{-1/2}\).  For \(I=I_n\), this corresponds to order \(n^{-1/4}\).
Thus, a natural direction for further work is to reduce the logarithmic
loss in the present bounds.

The theoretical contribution is an explicit Kolmogorov rate uniformly
over growing subintervals of \(I_n\) and on the full real line.  The main
new technical ingredient is a quantitative coupling among \(P_n\), the
stationary Weyl process and an \(m\)-dependent Gaussian process.  This coupling
preserves the relevant zero counts with high probability and controls the
displacement of ordered zeros.  We combine it with local-dependence Gaussian
approximation and use Jensen estimates for reciprocal polynomials to control
the real zeros outside \([-\sqrt n,\sqrt n]\).  Previous proofs
of qualitative central limit theorems for stationary Gaussian zero
crossings and random-polynomial zero counts also truncate moving-average
kernels and partition intervals into blocks
\cite{Cuzick1976,GranvilleWigman2011,NguyenVu2021}.  Here we combine
these steps quantitatively with a high-probability matching of the zeros
for the Weyl model.

The proof has three main components.

First, we compare \(P_n\) with the stationary Weyl process on growing
subintervals of \(I_n\) and match their zeros with high probability.
This comparison
transfers the required zero-count estimates from the stationary
process to the polynomial.

Second, we approximate the stationary process by a smooth
finite-range moving average.  We also show that its zero count can be
written as a sum of locally dependent random variables, which allows
us to obtain a quantitative Gaussian approximation.  Combining this
result with the first comparison, we prove
Theorem~\ref{thm:polynomial-logarithmic-edge-BE}.

Third, we decompose \(N_n(\R)\) into the numbers of zeros in \(I_n\), in
\([-\sqrt n,\sqrt n]\setminus I_n\) and in
\(\R\setminus[-\sqrt n,\sqrt n]\).  We use Jensen's formula to control the
numbers of zeros in \([-\sqrt n,\sqrt n]\setminus I_n\) and
\(\R\setminus[-\sqrt n,\sqrt n]\) and split the coefficient sequence to
construct independent random variables that, with high probability, agree
with \(N_n(I_n)\) and with the number of zeros outside
\([-\sqrt n,\sqrt n]\), respectively.  These estimates allow us to pass from
Theorem~\ref{thm:polynomial-logarithmic-edge-BE} to
Theorem~\ref{thm:full-line-polynomial-BE}.

\subsection*{Organization of the paper}

We organize the remainder of the paper as follows.  In
Section~\ref{sec:preliminaries}, we introduce the notation and state the
results used in the proofs of the main theorems, together with several useful
tools.  In
Sections~\ref{sec:proof-bulk-main-theorem} and
\ref{sec:proof-full-line-main-theorem}, we prove
Theorems~\ref{thm:polynomial-logarithmic-edge-BE} and
\ref{thm:full-line-polynomial-BE}, respectively.  In
Sections~\ref{sec:bulk-result-proofs} and \ref{sec:full-line-BE}, we establish,
in logical order, the results stated in Sections~\ref{subsec:bulk-results} and
\ref{subsec:full-line-results}, respectively.  In
Appendix~\ref{app:external-results}, we collect the external results used in
the proofs and group them by source, ordering the sources by their first use in
the paper.

\section{Preliminaries}
\label{sec:preliminaries}

We first introduce the notation used throughout the paper.  In
Sections~\ref{subsec:bulk-results} and \ref{subsec:full-line-results}, we state
the results needed to prove Theorems~\ref{thm:polynomial-logarithmic-edge-BE}
and \ref{thm:full-line-polynomial-BE}, respectively.  In
Section~\ref{subsec:useful-tools}, we collect several useful tools.

\subsection{Notational conventions}

We use the following conventions throughout the paper.  All
unqualified limits concern \(n\) tending to infinity and \(\log\) is
the natural logarithm.  We allow the implicit lower bound on \(n\) to
depend on fixed parameters when necessary, even when the displayed
constant is absolute.  For positive \(v_n\), the notation
\(u_n=O_\alpha(v_n)\) means that
\(|u_n|\leq C_\alpha v_n\) for all sufficiently large \(n\), where
the constant is independent of the asymptotic variables and depends
only on the displayed parameter \(\alpha\).  We write \(u_n=o(v_n)\)
when \(\lim_{n\to\infty}u_n/v_n=0\).  We use \(C>0\) to denote a
finite absolute constant whose value may change from line to line
and we add subscripts when necessary to indicate its dependence on or
relation to other quantities.  We use \(c>0\) analogously in
exponential estimates.  We reserve \(P_n\) for the finite Gaussian
Weyl polynomials and \(P_\infty\) for the full Weyl series.  We reserve Roman
\(N\) for real-zero counts.  We write \(\mathbb N_0\) for the
set of natural numbers including zero.

For \(f\in L^2(\R)\), we write
\[
 \|f\|_2=\|f\|_{L^2(\R)}
 =\left(\int_{\R}|f(u)|^2\dd u\right)^{1/2}.
\]
For a random variable \(X\) with \(\E|X|^p<\infty\), where
\(p\geq1\), we write
\[
 \|X\|_{L^p}=\left(\E|X|^p\right)^{1/p}.
\]
Thus \(\|\cdot\|_2\) denotes the deterministic \(L^2(\R)\) norm,
whereas \(\|\cdot\|_{L^p}\) denotes the norm on the underlying
probability space.

We let \(\ell^2(\mathbb N_0)\) be the real square-summable sequence
space.  For \(a=(a_k)_{k\geq0},b=(b_k)_{k\geq0}\in
\ell^2(\mathbb N_0)\), we write
\[
 \|a\|_{\ell^2}=\left(\sum_{k=0}^{\infty}|a_k|^2\right)^{1/2},
 \qquad
 \langle a,b\rangle_{\ell^2}=\sum_{k=0}^{\infty}a_kb_k.
\]
For an integer \(d\geq1\), \(x\in\R^d\) and a real
\(d\)-by-\(d\) matrix \(A\), we write
\[
 |x|=\left(\sum_{i=1}^{d}x_i^2\right)^{1/2},
 \qquad
 \|A\|_{\mathrm{op}}=\sup_{|x|=1}|Ax|.
\]
For a bounded function \(h\) on a set \(D\), we write
\(\|h\|_\infty=\sup_{u\in D}|h(u)|\).  For nonnegative integers
\(k,r\),
\((k)_r=k(k-1)\cdots(k-r+1)\) denotes the falling factorial, with
\((k)_0=1\).

For use in the polynomial--series comparisons and the proof of
Theorem~\ref{thm:full-line-polynomial-BE}, we fix the deterministic scales
\begin{equation}
 \delta_n=n^{-19},
 \qquad
 \varepsilon_n=n^{-38}.
 \label{eq:logarithmic-edge-parameter-choice}
\end{equation}

\subsection{Results used in the proof of Theorem~\ref{thm:polynomial-logarithmic-edge-BE}}
\label{subsec:bulk-results}

We first state exactly the results used in the short proof of
Theorem~\ref{thm:polynomial-logarithmic-edge-BE}.  On the coefficient
space from Section~\ref{sec:introduction}, set
\[
 P_\infty(x)=\sum_{k=0}^{\infty}\frac{\xi_k}{\sqrt{k!}}x^k,
 \qquad
 q_\infty(x)=e^{-x^2/2}P_\infty(x),
 \qquad x\in\R.
\]
For every bounded real interval \(A\), let \(N_\infty(A)\) be the
number of real zeros of \(P_\infty\) in \(A\), counted with
multiplicity.  Section~\ref{subsec:useful-tools} constructs an
everywhere-defined real-analytic version of \(P_\infty\), fixes its null-set
convention and proves the basic properties of these counts.

We next fix notation for the Gaussian processes \(q_{\infty,m}\) used below.
Their joint construction and selected pathwise versions are given in
Section~\ref{sec:finite-range-approximation}.  On the probability space
constructed in Lemma~\ref{lem:common-white-noise-realization}, let \(W\) be the real
Gaussian white noise on \(\R\).  Thus \(W\) is an isonormal Gaussian
process over \(L^2(\R)\) and
\[
 \E[W(h)W(k)]=\langle h,k\rangle_{L^2(\R)},
 \qquad h,k\in L^2(\R).
\]
Set \(g(u)=(2/\pi)^{1/4}e^{-u^2}\).  Fix
\(\chi\in C_c^\infty(\R)\) with \(0\leq\chi\leq1\), equal to one on
\([-1/4,1/4]\) and supported in \([-1/2,1/2]\).  For \(m\geq1\), set
\[
 g_m=
 \frac{g\,\chi(\mathord\cdot/m)}
 {\|g\,\chi(\mathord\cdot/m)\|_2},
 \qquad
 q_{\infty,m}(t)=W(g_m(t-\mathord\cdot)),
 \qquad t\in\R.
\]
We write \(N_{\infty,m}(I)\) for the zero count of this process on
\(I\).  The pathwise and measurability details are established in
Section~\ref{sec:bulk-result-proofs}.

The stationary Weyl zero count serves as the reference model for the bulk
analysis.  The following proposition collects the translation invariance,
moment estimate and variance asymptotics required in the subsequent
comparison arguments.

\begin{proposition}
\label{prop:stationary-Weyl-zero-count-results}
For every compact interval \(I\subset\R\),
\begin{equation}
 N_\infty(I)
 \stackrel{\mathrm d}{=}
 N_\infty([0,\ell]).
 \label{eq:stationary-count-translation-in-law}
\end{equation}
There is a finite absolute constant \(C\) such that, for \(\ell\geq1\),
\begin{equation}
 \E[N_\infty([0,\ell])^4]\leq C\ell^4.
 \label{eq:stationary-zero-count-fourth-moment}
\end{equation}
Moreover, for the constant \(V_\infty>0\) in
\eqref{eq:intro-known-Weyl-results},
\begin{equation}
 \lim_{\ell\to\infty}
 \frac{\Var(N_\infty([0,\ell]))}{\ell}
 =
 V_\infty.
 \label{eq:stationary-linear-variance-all-L}
\end{equation}
\end{proposition}

To compare \(N_n(I)\) with \(N_\infty(I)\), we use the following
high-probability coupling, which pairs the zeros of \(P_n\) and \(P_\infty\)
on intervals \(I\subseteq I_n\).

\begin{proposition}
\label{prop:logarithmic-edge-root-pairing}
There is an absolute constant \(B_0>0\) such that, for every fixed
\(B\geq B_0\), with \(L_n\) and \(I_n\) as in
\eqref{eq:logarithmic-edge-interval}, there is a finite absolute
constant \(C\) such that the following holds for all sufficiently large
\(n\).  For every \(I=[\varphi,\psi]\subseteq I_n\) with \(\ell\geq1\), there exists an event
\(\mathcal G_{n,I}\) satisfying
\begin{equation}
 \Pp(\mathcal G_{n,I}^c)
 \leq C\ell n^{-18}
 \leq Cn^{-17}.
 \label{eq:logarithmic-edge-good-event-probability}
\end{equation}
On \(\mathcal G_{n,I}\), we have
\(N_n(I)=N_\infty(I)\).  If this common cardinality is positive, let
\[
 x_1<\cdots<x_k,
 \qquad
 y_1<\cdots<y_k
\]
denote the zeros of \(P_\infty\) and \(P_n\), respectively, in
increasing order on \(I\).  Then
\begin{equation}
 \max_{1\leq j\leq k}|x_j-y_j|
 \leq
 n^{-1}.
 \label{eq:logarithmic-edge-root-displacement}
\end{equation}
\end{proposition}

We next derive the \(L^2\), mean and standard-deviation estimates needed to
compare the normalizations of \(N_n(I)\) and \(N_\infty(I)\).

\begin{corollary}
\label{cor:logarithmic-edge-L2-comparison}
Under the assumptions of
Proposition~\ref{prop:logarithmic-edge-root-pairing}, there is a finite
absolute constant \(C\) such that, for every \(I\subseteq I_n\) with
\(\ell\geq1\),
\begin{equation}
 \|N_n(I)-N_\infty(I)\|_{L^{2}}
 \leq
 C\ell^{1/4}n^{-7/2}
 \leq Cn^{-3}
 \label{eq:logarithmic-edge-L2-comparison}
\end{equation}
and
\begin{equation}
 |\E N_n(I)-\E N_\infty(I)|
 +
 \left|
  \sqrt{\Var(N_n(I))}
  -
  \sqrt{\Var(N_\infty(I))}
 \right|
 \leq
 C\ell^{1/4}n^{-7/2}
 \leq Cn^{-3}.
 \label{eq:logarithmic-edge-normalization-comparison}
\end{equation}
\end{corollary}

We now record the quantitative estimates for \(N_{\infty,\mu_\ell}(I)\) that
are used in the Gaussian approximation on \(I\).  The following corollary compares
\(N_{\infty,\mu_\ell}(I)\) with \(N_\infty(I)\), establishes the required
variance bounds and provides the Berry--Esseen estimate used below.

\begin{corollary}
\label{cor:finite-range-BE-package}
There are constants \(b,C>0\) and \(\ell_1\geq e\) with the following
property.  For every \(I\subset\R\) with \(\ell\geq\ell_1\), set
\[
 \mu_\ell=b\sqrt{\log\ell}.
\]
Then
 \begin{equation}
   \Pp\!\left(
    N_{\infty,\mu_\ell}(I)\neq N_\infty(I)
   \right)
   \leq C\ell^{-16},
   \label{eq:finite-range-package-count-comparison}
 \end{equation}
 \begin{equation}
   \|N_{\infty,\mu_\ell}(I)-N_\infty(I)\|_{L^{2}}
   \leq C\ell^{-3}
   \label{eq:finite-range-package-L2-comparison}
 \end{equation}
 and
 \begin{equation}
   |\E N_{\infty,\mu_\ell}(I)-\E N_\infty(I)|
   +
   \left|
    \sqrt{\Var(N_{\infty,\mu_\ell}(I))}
    -
    \sqrt{\Var(N_\infty(I))}
   \right|
   \leq C\ell^{-3}.
  \label{eq:finite-range-package-normalization-comparison}
 \end{equation}
Moreover,
 \begin{equation}
  \frac{V_\infty}{8}\ell
  \leq
  \Var(N_{\infty,\mu_\ell}(I))
  \leq
  C\ell
 \label{eq:finite-range-package-variance}
\end{equation}
and
 \begin{equation}
  d_{\mathrm K}\!\left(
   \frac{N_{\infty,\mu_\ell}(I)-\E N_{\infty,\mu_\ell}(I)}
   {\sqrt{\Var(N_{\infty,\mu_\ell}(I))}},
   Z
  \right)
 \leq
 C\frac{\log\ell}{\sqrt\ell}.
 \label{eq:finite-range-package-BE}
\end{equation}
\end{corollary}

To complete the comparison with the polynomial zero count, we use the following
elementary lemma, which bounds the Kolmogorov distance in terms of the
probability that the two random variables differ and the differences between
their means and standard deviations.

\begin{lemma}
\label{lem:finite-range-exact-coupling-transfer}
Let \(X\) and \(Y\) be square-integrable real random variables.  Suppose that
\(\Var(X)>0\) and
\[
 |\E X-\E Y|
 +
 \left|\sqrt{\Var(X)}-\sqrt{\Var(Y)}\right|
 \leq
 \frac12\sqrt{\Var(X)}.
\]
Then \(\Var(Y)>0\) and
\[
 \begin{aligned}
 d_{\mathrm K}\!\left(
  \frac{Y-\E Y}{\sqrt{\Var(Y)}},Z
 \right)
 &\leq
 d_{\mathrm K}\!\left(
  \frac{X-\E X}{\sqrt{\Var(X)}},Z
 \right)
 +
 \Pp(X\neq Y)\\
 &\quad+
 2
 \frac{
  |\E X-\E Y|
  +
  \left|\sqrt{\Var(X)}-\sqrt{\Var(Y)}\right|
 }{
  \sqrt{\Var(X)}
 }.
 \end{aligned}
\]
\end{lemma}

\subsection{Results for the proof of Theorem~\ref{thm:full-line-polynomial-BE}}
\label{subsec:full-line-results}

Throughout the proof of Theorem~\ref{thm:full-line-polynomial-BE}, we take
\(B=20\).  In addition to the interval \(I_n\) from
\eqref{eq:logarithmic-edge-interval}, define
\begin{equation}
 I_n^{\mathrm{edge}}
 =
 [-\sqrt n,-L_n)\cup(L_n,\sqrt n],
 \qquad
 I_n^{\mathrm{out}}
 =
 (-\infty,-\sqrt n)\cup(\sqrt n,\infty).
 \label{eq:full-line-region-definitions}
\end{equation}
 With the endpoint conventions in \eqref{eq:full-line-region-definitions},
 the sets \(I_n\), \(I_n^{\mathrm{edge}}\) and \(I_n^{\mathrm{out}}\) form a
 partition of \(\R\).  Consequently, we obtain the disjoint decomposition
\begin{equation}
 N_n(\R)
 =
 N_n(I_n)+N_n(I_n^{\mathrm{edge}})+N_n(I_n^{\mathrm{out}}).
 \label{eq:full-line-count-identity}
\end{equation}
We split the coefficients at
\begin{equation}
 m_n=\left\lfloor n-B\sqrt{n\log n}\right\rfloor,
 \qquad
 d_n=n-m_n=\left\lceil B\sqrt{n\log n}\right\rceil.
 \label{eq:full-line-coefficient-band-cutoff}
\end{equation}
The count \(N_{P_{m_n}}(I_n)\) depends only on
\(\xi_0,\ldots,\xi_{m_n}\), whereas
\(N_{P_n-P_{m_n}}(I_n^{\mathrm{out}})\) depends only on
\(\xi_{m_n+1},\ldots,\xi_n\).

The following five results, together with
Theorem~\ref{thm:polynomial-logarithmic-edge-BE},
Lemma~\ref{lem:finite-range-exact-coupling-transfer} and the auxiliary
estimates in Section~\ref{subsec:useful-tools}, are used in the proof of
Theorem~\ref{thm:full-line-polynomial-BE}.

To incorporate the zeros in \(I_n^{\mathrm{edge}}\) into the proof of
Theorem~\ref{thm:full-line-polynomial-BE}, we establish the following
high-probability bound, which shows that their number is \(O(\log n)\).

\begin{proposition}
\label{prop:full-line-edge-strip-count}
For every \(A>0\), there is a finite constant \(C_{A,B}\) such that,
for all sufficiently large \(n\),
\begin{equation}
 \Pp\!\left(
  N_n(I_n^{\mathrm{edge}})>C_{A,B}\log n
 \right)
 \leq n^{-A}.
 \label{eq:full-line-edge-strip-tail}
\end{equation}
\end{proposition}

The next proposition compares \(N_n(I_n)\) with the zero count
\(N_{P_{m_n}}(I_n)\), which is determined entirely by the coefficients
\(\xi_0,\ldots,\xi_{m_n}\).

\begin{proposition}
\label{prop:full-line-low-band-localization}
For all sufficiently large \(n\),
\begin{equation}
 \Pp\!\left(
  N_n(I_n)\neq N_{P_{m_n}}(I_n)
 \right)
 \leq
 Cn^{-16}.
 \label{eq:full-line-low-band-localization}
\end{equation}
\end{proposition}

The following proposition compares \(N_n(I_n^{\mathrm{out}})\) with the zero
count \(N_{P_n-P_{m_n}}(I_n^{\mathrm{out}})\), which is determined entirely by
the coefficients \(\xi_{m_n+1},\ldots,\xi_n\).  It also provides the
probability and \(L^2\) bounds used below.

\begin{proposition}
\label{prop:full-line-upper-band-localization}
For all sufficiently large \(n\),
\begin{equation}
 \Pp\!\left(
  N_n(I_n^{\mathrm{out}})
  \neq
  N_{P_n-P_{m_n}}(I_n^{\mathrm{out}})
 \right)
 \leq Cn^{-17}.
 \label{eq:full-line-upper-band-localization}
\end{equation}
Moreover,
\begin{equation}
 \|N_n(I_n^{\mathrm{out}})
   -N_{P_n-P_{m_n}}(I_n^{\mathrm{out}})\|_{L^2}
 \leq
 Cn^{-15/2}
 \label{eq:full-line-upper-band-coupling-L2}
\end{equation}
and
\begin{equation}
 \|N_{P_n-P_{m_n}}(I_n^{\mathrm{out}})
   -\E N_{P_n-P_{m_n}}(I_n^{\mathrm{out}})\|_{L^2}
 \leq
 C\log n.
 \label{eq:full-line-truncated-exterior-L2}
\end{equation}
\end{proposition}

The next lemma bounds the Kolmogorov distance for the normalized sum \(U+V\)
in terms of the corresponding distance for \(U\) and the variance ratio
\(\Var(V)/\Var(U)\), assuming that \(U\) and \(V\) are independent.

\begin{lemma}
\label{lem:full-line-independent-transfer}
Let \(U,V\) be independent square-integrable real random variables
with \(\Var(U)>0\).  Then
\begin{equation}
 d_{\mathrm K}\!\left(
  \frac{U+V-\E(U+V)}{\sqrt{\Var(U)+\Var(V)}},Z
 \right)
 \leq
 d_{\mathrm K}\!\left(
  \frac{U-\E U}{\sqrt{\Var(U)}},Z
 \right)
 +
 \frac{1}{\sqrt{2\pi e}}\frac{\Var(V)}{\Var(U)}.
 \label{eq:full-line-independent-transfer}
\end{equation}
\end{lemma}

The zero count \(N_n(I_n^{\mathrm{edge}})\) need not be independent of
the sum of \(N_{P_{m_n}}(I_n)\) and
\(N_{P_n-P_{m_n}}(I_n^{\mathrm{out}})\).  The following lemma addresses this
situation without assuming independence.  It bounds the resulting Kolmogorov
distance in terms of the original distance, the exceptional-event probability
and both \(L^2\) and high-probability bounds for the centered random variable
being added.

\begin{lemma}
\label{lem:full-line-additive-transfer}
Let \(X,T,Y\) be square-integrable, with \(Y=X+T\) and
\(\Var(X)>0\).
Suppose that an event \(G\) and numbers \(p,r\geq0\) satisfy
\[
 \Pp(G^c)\leq p,
 \qquad
 |T-\E T|\leq r
 \quad\text{on }G.
\]
If
\[
 \|T-\E T\|_{L^2}\leq\frac12\sqrt{\Var(X)},
\]
then \(\Var(Y)>0\) and
\begin{equation}
 d_{\mathrm K}\!\left(
  \frac{Y-\E Y}{\sqrt{\Var(Y)}},Z
 \right)
 \leq
 d_{\mathrm K}\!\left(
  \frac{X-\E X}{\sqrt{\Var(X)}},Z
 \right)
 +
 p
 +
 \frac{\|T-\E T\|_{L^2}+r}{\sqrt{\Var(X)}}.
 \label{eq:full-line-additive-transfer}
\end{equation}
\end{lemma}

\subsection{Auxiliary constructions and estimates}
\label{subsec:useful-tools}

We next present several auxiliary analytic and probabilistic results used in
the detailed proofs.

\subsubsection{The full Weyl series and its stationary normalization}

We work on a complete probability space
\((\Omega,\mathcal F,\Pp)\), and we complete every enlargement of this
probability space used below.  Recall that the sequence
\((\xi_k)_{k\geq0}\) and the polynomials \(P_n\) are defined in
\eqref{eq:weyl-polynomial}.  The full Weyl series introduced in
Section~\ref{subsec:bulk-results} is
\begin{equation}
 P_\infty(x)
 =
 \sum_{k=0}^{\infty}\frac{\xi_k}{\sqrt{k!}}x^k,
 \qquad x\in\R.
 \label{eq:weyl-full-series}
\end{equation}

We first verify that \eqref{eq:weyl-full-series} almost surely defines a
real-analytic function on \(\R\) that can be differentiated term by term.  For
\(R\in\mathbb N\) and \(j\in\mathbb N_0\), set
\[
 S_{R,j}
 =
 \sum_{k=j}^{\infty}
 |\xi_k|\frac{k!}{(k-j)!\sqrt{k!}}R^{k-j}.
\]
Since the random variables \(\xi_k\) are identically distributed, the ratio
of two consecutive terms in the corresponding series of expectations is
\[
 R\frac{\sqrt{k+1}}{k+1-j}
 \longrightarrow 0
 \qquad (k\to\infty).
\]
We apply the ratio test and Tonelli's theorem to obtain
\[
 \E S_{R,j}
 =
 \sum_{k=j}^{\infty}
 \E|\xi_k|\frac{k!}{(k-j)!\sqrt{k!}}R^{k-j}
 <\infty,
\]
Thus, for every \(R\in\mathbb N\) and \(j\in\mathbb N_0\), we have
\[
 \Pp(S_{R,j}<\infty)=1.
\]
\begin{samepage}
Define the event
\[
 \Omega_{\mathrm W}
 =
 \{\xi_0\neq0\}
 \cap
 \bigcap_{R=1}^{\infty}
 \bigcap_{j=0}^{\infty}
 \{S_{R,j}<\infty\}.
\]
\end{samepage}
 Then \(\Pp(\Omega_{\mathrm W})=1\).  On this event, for every
 \(R\in\mathbb N\) and \(j\in\mathbb N_0\), we apply the Weierstrass
 \(M\)-test to the series obtained by formally differentiating
 \eqref{eq:weyl-full-series} \(j\) times and obtain absolute and uniform
 convergence on \([-R,R]\).  We may therefore differentiate term by term and
 obtain
\[
 P_\infty^{(j)}(x)
 =
 \sum_{k=j}^{\infty}
 \frac{\xi_k}{\sqrt{k!}}
 \frac{k!}{(k-j)!}x^{k-j},
 \qquad x\in\R,\quad j\in\mathbb N_0,
\]
with locally uniform convergence in \(x\).  Thus \(P_\infty\) is real
analytic on \(\R\).  Moreover, \(P_\infty(0)=\xi_0\neq0\) on
\(\Omega_{\mathrm W}\), so \(P_\infty\) is not identically zero.  The zeros
of a nonzero real-analytic function are isolated.  Hence \(P_\infty\) has only
finitely many zeros in every bounded real interval.  On
\(\Omega\setminus\Omega_{\mathrm W}\), we set \(P_\infty\equiv1\).  All
subsequent occurrences refer to this everywhere-defined version.

For \(x\in\R\), define \(q_n\) and recall the stationary normalization
\(q_\infty\) from Section~\ref{subsec:bulk-results}:
\begin{equation}
 q_n(x)=e^{-x^{2}/2}P_n(x),
 \qquad
 q_\infty(x)=e^{-x^{2}/2}P_\infty(x).
 \label{eq:normalized-weyl-functions}
\end{equation}
Since \(e^{-x^2/2}>0\), \(q_n\) and \(q_\infty\) have
the same real zeros, with the same multiplicities, as \(P_n\) and
 \(P_\infty\), respectively.  For \(x,y\in\R\), we sum the exponential series
 and obtain
\begin{equation}
 \E[q_\infty(x)q_\infty(y)]
 =e^{-(x^2+y^2)/2}
 \sum_{k=0}^{\infty}\frac{(xy)^k}{k!}
 =e^{-(x-y)^2/2}.
 \label{eq:stationary-weyl-covariance}
\end{equation}
For each fixed \(x\in\R\),
\[
 \E|P_\infty(x)-P_n(x)|^2
 =
 \sum_{k=n+1}^{\infty}\frac{x^{2k}}{k!}
 \longrightarrow0.
\]
Consequently, for every \(d\geq1\) and
\(x_1,\ldots,x_d\in\R\),
\[
 \bigl(q_n(x_1),\ldots,q_n(x_d)\bigr)
 \longrightarrow
 \bigl(q_\infty(x_1),\ldots,q_\infty(x_d)\bigr)
 \quad\text{in }\bigl(L^2(\Omega)\bigr)^d.
\]
 The limit is a centered Gaussian vector.  Combining this fact with
 \eqref{eq:stationary-weyl-covariance}, we obtain
\[
 \E q_\infty(x)=0,
 \qquad
 \Var(q_\infty(x))=1,
 \qquad
 \Cov(q_\infty(x),q_\infty(y))=e^{-(x-y)^2/2}.
\]
Thus \(q_\infty\) is a centered, unit-variance stationary Gaussian process.

\subsubsection{Basic zero-count properties and the Kac--Rice formula}

For every fixed \(B>0\) and all sufficiently large \(n\), the quantities in
\eqref{eq:logarithmic-edge-interval} satisfy
\begin{equation}
 1\leq B\sqrt{\log n}\leq\frac{\sqrt n}{4},
 \qquad
 \frac34\sqrt n\leq L_n\leq\sqrt n,
 \qquad
 |I_n|=2L_n.
 \label{eq:logarithmic-edge-length-bounds}
\end{equation}
Recall that \(N_\infty(A)\) is the zero count introduced in
Section~\ref{subsec:bulk-results}.  Only bounded real intervals occur below.

We first verify the nondegeneracy conditions used in the zero-count formulas.
For \(n\geq1\) and \(x\in\R\), the coefficient vectors of \(P_n(x)\) and
\(P_n'(x)\) are linearly independent: their zeroth coordinates are one and
zero, respectively, while the first coordinate of the second vector is one.
 Hence \((P_n(x),P_n'(x))\) is nondegenerate.  We differentiate
 \eqref{eq:stationary-weyl-covariance} and obtain
\[
 \partial_y e^{-(x-y)^{2}/2}
 =(x-y)e^{-(x-y)^{2}/2},
 \qquad
 \partial_x\partial_y e^{-(x-y)^{2}/2}
 =\bigl(1-(x-y)^{2}\bigr)e^{-(x-y)^{2}/2}.
\]
 We then set \(y=x\) and obtain
\begin{equation}
 \Var(q_\infty(x))=\Var(q_\infty'(x))=1,
 \qquad
 \E[q_\infty(x)q_\infty'(x)]=0.
 \label{eq:stationary-value-derivative-covariance}
\end{equation}
Thus \((q_\infty(x),q_\infty'(x))\) is also nondegenerate.

Let \(p_Y\) denote the density of a real random variable \(Y\).  Since
 \(\Var(P_n(x))\geq1\) and \(\Var(q_\infty(x))=1\), we apply the Gaussian
 density formula and obtain, uniformly in \(n\geq1\) and \(x\in\R\),
\[
 \sup_{u\in\R}p_{P_n(x)}(u)
 \leq\frac1{\sqrt{2\pi}},
 \qquad
 \sup_{u\in\R}p_{q_\infty(x)}(u)
 =\frac1{\sqrt{2\pi}}.
\]
 For every compact interval \(K\subset\R\), we apply
 Lemma~\ref{lem:appendix-azais-wschebor-prop1-20} and obtain
\[
 \Pp\!\left(
  \exists x\in K:\ P_n(x)=P_n'(x)=0
 \right)
 =
 \Pp\!\left(
  \exists x\in K:\ q_\infty(x)=q_\infty'(x)=0
 \right)
 =0.
\]
 Intersecting \(\Omega_{\mathrm W}\) with these probability-one events over
 \(n\in\mathbb N\) and \(K=[-r,r]\), \(r\in\mathbb N\), we obtain a single
probability-one event on which every \(P_n\), \(n\in\mathbb N_0\), as well as
\(q_\infty\) and \(P_\infty\), has only simple real zeros.  The zeros of
\(P_\infty\) are also locally finite there.

We next record the measurability consequence used below.  Let \(K\) be a
compact interval with nonempty interior and let \(k\geq1\).  Up to a null
event,
\[
 \begin{aligned}
  \{N_\infty(K)\geq k\}
  ={}&
  \bigcup_{s=1}^{\infty}
  \bigcap_{M=1}^{\infty}
  \bigcup_{\substack{
   t_1,\ldots,t_k\in K\cap\mathbb Q\\
   |t_i-t_j|\geq1/s,\ i\neq j
  }}
  \bigcap_{i=1}^{k}
  \{|q_\infty(t_i)|<1/M\}.
 \end{aligned}
\]
 Using this countable representation, we conclude that \(N_\infty(K)\) is a finite-valued
random variable measurable with respect to the \(\Pp\)-augmentation of
\(\sigma\{q_\infty(t):t\in K\}\).  For a singleton \(K=\{a\}\), the count
agrees almost surely with \(\mathbf1_{\{q_\infty(a)=0\}}\).  The same
conclusion for every bounded interval follows by monotone approximation with
compact intervals and the endpoint indicators.

We now record the Kac--Rice formula for the expected number of zeros in the
form used below.  Let
\(X\) be a centered Gaussian process with almost surely \(C^1\) paths on a
compact interval \(I\) and suppose that \((X(t),X'(t))\) is nondegenerate
for every \(t\in I\).  Denote its number of zeros in \(I\) by
 \(\mathcal N_X(I)\).  We apply
 Lemma~\ref{lem:appendix-azais-wschebor-thm3-2} and obtain
\begin{equation}
 \begin{aligned}
  \E\mathcal N_X(I)
  &=
  \int_I
  \E\!\left[|X'(t)|\mid X(t)=0\right]p_{X(t)}(0)\dd t\\
  &=
  \frac1\pi\int_I
  \frac{
   \sqrt{\Var(X(t))\Var(X'(t))-\Cov(X(t),X'(t))^{2}}
  }{\Var(X(t))}\dd t.
 \end{aligned}
 \label{eq:expected-zero-count-Kac-Rice}
\end{equation}
Formula~\eqref{eq:expected-zero-count-Kac-Rice} applies to \(P_n\) for every
\(n\geq1\) and to \(q_\infty\) on every compact interval.  By
\eqref{eq:stationary-value-derivative-covariance},
\[
 \Var(q_\infty(t))=\Var(q_\infty'(t))=1,
 \qquad
 \Cov(q_\infty(t),q_\infty'(t))=0.
\]
Since \(q_\infty\) and \(P_\infty\) have the same real zeros, it follows that
\[
 \E N_\infty(I)
 =
 \frac1\pi\int_I\dd t
 =
 \frac{\ell}{\pi}.
\]
Thus the expected stationary zero count per unit length is \(1/\pi\).

For every fixed \(x\in\R\), the random variables \(P_n(x)\) and
\(P_\infty(x)\) are nondegenerate Gaussian.  Therefore,
\[
 \Pp(P_n(x)=0)=\Pp(P_\infty(x)=0)=0.
\]
Consequently, the inclusion or exclusion of a prescribed endpoint does not
affect either zero count almost surely.  More generally, if \(Q\) is a
nonzero real polynomial and \(A\subseteq\R\), then
\(N_Q(A)\leq\deg Q\).  In particular, for every Borel set
\(A\subseteq\R\),
\begin{equation}
 0\leq N_n(A)\leq N_n(\R)\leq n
 \qquad\text{almost surely}.
 \label{eq:polynomial-zero-count-degree-bound}
\end{equation}

\subsubsection{A moment comparison lemma}

The following lemma bounds the \(L^2\) distance, the difference of the means
and the difference of the standard deviations of two random variables in
terms of their fourth moments and their probability of disagreement.  No
independence assumption is required.

\begin{lemma}
\label{lem:rare-disagreement-moment-transfer}
Let \(X,Y\in L^4\) be real-valued random variables.  Then
\begin{equation}
 \|X-Y\|_{L^{2}}
 \leq
 \left(
  8\bigl(\E|X|^4+\E|Y|^4\bigr)
 \right)^{1/4}
 \Pp(X\neq Y)^{1/4},
 \label{eq:rare-disagreement-L2-transfer}
\end{equation}
and
\begin{equation}
 |\E X-\E Y|
 +
 \left|
  \sqrt{\Var(X)}-\sqrt{\Var(Y)}
 \right|
 \leq
 2\|X-Y\|_{L^{2}}.
 \label{eq:rare-disagreement-mean-sd-transfer}
\end{equation}
\end{lemma}

\begin{proof}
Set
\[
 D=X-Y,
 \qquad
 \mathcal E=\{X\neq Y\}.
\]
 Since \(D=0\) on \(\mathcal E^c\), we apply the Cauchy--Schwarz inequality
 and use \(|x-y|^4\leq8(|x|^4+|y|^4)\) to obtain
\[
 \begin{aligned}
  \|D\|_{L^{2}}^{2}
  &=
  \E\!\left[|D|^2\mathbf1_{\mathcal E}\right]\\
  &\leq
  \bigl(\E|D|^4\bigr)^{1/2}\Pp(\mathcal E)^{1/2}\\
  &\leq
  \left(
   8\bigl(\E|X|^4+\E|Y|^4\bigr)
  \right)^{1/2}
  \Pp(\mathcal E)^{1/2}.
 \end{aligned}
\]
 Taking square roots proves \eqref{eq:rare-disagreement-L2-transfer}.  Moreover, we have
\[
 |\E X-\E Y|
 =
 |\E D|
 \leq
 \|D\|_{L^2},
\]
 and we apply the reverse triangle inequality in \(L^2\) to obtain
\[
 \begin{aligned}
  \left|
   \sqrt{\Var(X)}-\sqrt{\Var(Y)}
  \right|
  &=
  \left|
   \|X-\E X\|_{L^2}-\|Y-\E Y\|_{L^2}
  \right|\\
  &\leq
  \|D-\E D\|_{L^2}
  \leq
  \|D\|_{L^2}.
 \end{aligned}
\]
Adding these two estimates proves
\eqref{eq:rare-disagreement-mean-sd-transfer}.
\end{proof}

\subsubsection{An \texorpdfstring{\(L^p\)}{Lp} Sobolev estimate on intervals}

For a nonnegative integer \(r\) and \(h\in C^r(I)\), define
\begin{equation}
 \|h\|_{C^r(I)}
 =
 \sum_{j=0}^{r}\sup_{t\in I}|h^{(j)}(t)|.
 \label{eq:comparison-Cr-norm}
\end{equation}
We use the norm in \eqref{eq:comparison-Cr-norm} throughout the paper.

The following lemma bounds the supremum norm of a \(C^1\) function in terms
of the \(L^p\) norms of the function and its derivative.

\begin{lemma}
\label{lem:comparison-interval-Sobolev}
For every \(p\geq1\), every \(I\) with \(\ell>0\) and every
\(h\in C^1(I)\),
\begin{equation}
 \sup_{t\in I}|h(t)|^p
 \leq
 2^{p-1}
 \left(
  \ell^{-1}\int_I|h(t)|^p\dd t
  +
  \ell^{p-1}\int_I|h'(t)|^p\dd t
 \right).
 \label{eq:comparison-interval-Sobolev}
\end{equation}
In particular, if \(p=2\) and \(1/2\leq\ell\leq1\), then
\[
 \sup_{t\in I}|h(t)|^2
 \leq
 4\int_I\bigl(|h(t)|^2+|h'(t)|^2\bigr)\dd t.
\]
If \(\ell=1\), then, for every \(p\geq1\),
\[
 \sup_{t\in I}|h(t)|^p
 \leq
 2^{p-1}
 \int_I\bigl(|h(t)|^p+|h'(t)|^p\bigr)\dd t.
\]
\end{lemma}

\begin{proof}
Since \(h\) is continuous on \(I\), there exists \(t_0\in I\) such that
\[
 |h(t_0)|^p
 \leq
 \ell^{-1}\int_I|h(u)|^p\dd u.
\]
 For every \(t\in I\), we apply the fundamental theorem of calculus and
 H\"older's inequality to obtain
\[
 |h(t)|
 \leq |h(t_0)|+\int_I|h'(u)|\dd u
 \leq |h(t_0)|
 +\ell^{1-1/p}
 \left(\int_I|h'(u)|^p\dd u\right)^{1/p}.
\]
Raising both sides to the \(p\)th power and using
\((u+v)^p\leq2^{p-1}(u^p+v^p)\) proves
\eqref{eq:comparison-interval-Sobolev}.  The first special case follows from
\(\ell^{-1}\leq2\) and \(\ell\leq1\).  The second follows by setting
\(\ell=1\).
\end{proof}

\section{Proof of Theorem~\ref{thm:polynomial-logarithmic-edge-BE}}
\label{sec:proof-bulk-main-theorem}

In this section, we prove
Theorem~\ref{thm:polynomial-logarithmic-edge-BE}.  We first bound the
probability that \(N_n(I)\) differs from
\(N_{\infty,\mu_\ell}(I)\) and compare their means and standard deviations.
We then deduce the variance bounds and asymptotics for \(N_n(I)\) from those
for the stationary Weyl zero count.  Finally, we apply
Lemma~\ref{lem:finite-range-exact-coupling-transfer} to transfer the
Berry--Esseen bound for \(N_{\infty,\mu_\ell}(I)\) to \(N_n(I)\).

We fix \(I\subseteq I_n\) in the stated range.  Since \(\ell\) is
sufficiently large, we may assume that \(\ell\geq\ell_1\).  Hence all
conclusions of Corollary~\ref{cor:finite-range-BE-package} apply to \(I\).
Moreover, \(I\subseteq I_n\) implies \(\ell\leq2\sqrt n\).  We combine
Corollary~\ref{cor:finite-range-BE-package} with
Proposition~\ref{prop:logarithmic-edge-root-pairing} to obtain
\begin{equation}
 \begin{aligned}
  \Pp\!\left(
   N_{\infty,\mu_\ell}(I)\neq N_n(I)
  \right)
  &\leq
  \Pp\!\left(
   N_{\infty,\mu_\ell}(I)\neq N_\infty(I)
  \right)
  +
  \Pp\!\left(
   N_\infty(I)\neq N_n(I)
  \right)\\
  &\leq
  C\ell^{-16}+C\ell n^{-18}.
 \end{aligned}
 \label{eq:bulk-assembly-discrepancy}
\end{equation}
We apply Corollary~\ref{cor:finite-range-BE-package} and
Corollary~\ref{cor:logarithmic-edge-L2-comparison} to obtain
\begin{equation}
 \begin{aligned}
 &|\E N_{\infty,\mu_\ell}(I)-\E N_n(I)|
 +\left|\sqrt{\Var(N_{\infty,\mu_\ell}(I))}
 -\sqrt{\Var(N_n(I))}\right|\\
 \leq{}& C\ell^{-3}+C\ell^{1/4}n^{-7/2}.
 \end{aligned}
 \label{eq:bulk-assembly-normalization}
\end{equation}
Recall from \eqref{eq:finite-range-package-variance} that
\begin{equation}
 \frac{V_\infty}{8}\ell
 \leq
 \Var(N_{\infty,\mu_\ell}(I))
 \leq
 C\ell.
 \label{eq:bulk-assembly-finite-range-variance}
\end{equation}
We combine \eqref{eq:bulk-assembly-normalization} and
\eqref{eq:bulk-assembly-finite-range-variance} to obtain
\[
 \frac{C\ell^{-3}+C\ell^{1/4}n^{-7/2}}
 {\sqrt{V_\infty\ell/8}}
 \leq
 C\left(\ell^{-7/2}+\ell^{-1/4}n^{-7/2}\right).
\]
Obviously, the right-hand side is at most \(1/2\).
\[
 \begin{aligned}
 &|\E N_{\infty,\mu_\ell}(I)-\E N_n(I)|
 +\left|\sqrt{\Var(N_{\infty,\mu_\ell}(I))}
 -\sqrt{\Var(N_n(I))}\right|\\
 \leq{}& \frac12\sqrt{\Var(N_{\infty,\mu_\ell}(I))}.
 \end{aligned}
\]
In the proof of
Lemma~\ref{lem:finite-range-BE-variance-transfer}, we establish
\eqref{eq:stationary-two-sided-linear-variance}, which applies to the present
interval because \(\ell\geq\ell_1\).
 We then apply Corollary~\ref{cor:logarithmic-edge-L2-comparison} and obtain
\[
 \left|
  \sqrt{\Var(N_n(I))}-\sqrt{\Var(N_\infty(I))}
 \right|
 \leq C\ell^{1/4}n^{-7/2}.
\]
For all sufficiently large \(n\), uniformly over the intervals under
consideration,
the last error is at most one half of
\(\sqrt{\Var(N_\infty(I))}\).  Hence
\[
 \frac{V_\infty}{8}\ell
 \leq
 \Var(N_n(I))
 \leq
 C\ell.
\]
This proves the two-sided variance assertion with
\(c=V_\infty/8\).

To prove \eqref{eq:polynomial-bulk-subinterval-variance-asymptotic}, let
\(r_n\to\infty\) be as in the theorem.  We apply the stationary Weyl
variance limit in
Proposition~\ref{prop:stationary-Weyl-zero-count-results}.  Since
\(r_n\to\infty\), we obtain
\begin{equation}
 \sup_{u\geq r_n}
 \left|
  \frac{\Var(N_\infty([0,u]))}{u}-V_\infty
 \right|
 \longrightarrow0.
 \label{eq:bulk-assembly-stationary-variance-tail}
\end{equation}
For \(I\subseteq I_n\) with \(\ell\geq r_n\), we combine stationarity and the
upper bound in \eqref{eq:stationary-two-sided-linear-variance} with
Corollary~\ref{cor:logarithmic-edge-L2-comparison} and obtain
\begin{equation}
 \begin{aligned}
 \frac{|
  \Var(N_n(I))-\Var(N_\infty(I))
 |}{\ell}
 &\leq
 \frac{C\ell^{1/4}n^{-7/2}
  \bigl(\sqrt{\ell}+\ell^{1/4}n^{-7/2}\bigr)}{\ell}\\
 &\leq
 Cn^{-7/2}r_n^{-1/4}+Cn^{-7}r_n^{-1/2}.
 \end{aligned}
 \label{eq:bulk-assembly-polynomial-stationary-variance-comparison}
\end{equation}
We combine \eqref{eq:bulk-assembly-stationary-variance-tail} and
\eqref{eq:bulk-assembly-polynomial-stationary-variance-comparison} with the
triangle inequality to prove
\eqref{eq:polynomial-bulk-subinterval-variance-asymptotic}.

We apply Lemma~\ref{lem:finite-range-exact-coupling-transfer} with
\[
 X=N_{\infty,\mu_\ell}(I),
 \qquad
 Y=N_n(I).
\]
We then combine the lemma with
\eqref{eq:finite-range-package-BE},
\eqref{eq:bulk-assembly-discrepancy} and
\eqref{eq:bulk-assembly-normalization} to obtain the required bound
\[
 \begin{aligned}
  &d_{\mathrm K}\!\left(
   \frac{N_n(I)-\E N_n(I)}
   {\sqrt{\Var(N_n(I))}},Z
  \right)\\
  \leq{}&
  C\frac{\log\ell}{\sqrt\ell}
  +C\ell^{-16}
  +C\ell n^{-18}
  +
  2\sqrt{\frac{8}{V_\infty\ell}}
  \left(
   C\ell^{-3}+C\ell^{1/4}n^{-7/2}
  \right)\\
  \leq{}&
  C\frac{\log\ell}{\sqrt\ell}.
 \end{aligned}
\]
In the last inequality, we used
\(\ell_1\leq\ell\leq2\sqrt n\).  Since \(\ell_1\geq e\), the
terms \(\ell^{-16}\) and \(\ell^{-7/2}\) are bounded by
a constant times \(\log\ell/\sqrt\ell\).  The same is true of the two
terms containing \(n\) uniformly in this range.
The proof is therefore complete.

\section{Proof of Theorem~\ref{thm:full-line-polynomial-BE}}
\label{sec:proof-full-line-main-theorem}

In this section, we prove Theorem~\ref{thm:full-line-polynomial-BE}.  We first
replace \(N_n(I_n)\) and \(N_n(I_n^{\mathrm{out}})\) by
\(N_{P_{m_n}}(I_n)\) and
\(N_{P_n-P_{m_n}}(I_n^{\mathrm{out}})\), respectively, and use the fact that
the latter two zero counts depend on disjoint sets of coefficients.  We then
incorporate \(N_n(I_n^{\mathrm{edge}})\) and finally compare the resulting sum
with \(N_n(\mathbb R)\).

We first compare \(N_n(I_n)\) with \(N_{P_{m_n}}(I_n)\).
We apply Theorem~\ref{thm:polynomial-logarithmic-edge-BE} with
\(I=I_n\), use \eqref{eq:logarithmic-edge-length-bounds} and obtain
\begin{equation}
 d_{\mathrm K}\!\left(
  \frac{N_n(I_n)-\E N_n(I_n)}
  {\sqrt{\Var(N_n(I_n))}},Z
 \right)
 \leq
 C\frac{\log n}{n^{1/4}}.
 \label{eq:full-line-bulk-BE-bound}
\end{equation}
Recall from \eqref{eq:intro-growing-window-variance} that
\[
 \lim_{n\to\infty}
 \frac{\Var(N_n(I_n))}{\sqrt n}
 =
 2V_\infty.
\]
We therefore obtain
\begin{equation}
 \sqrt{\Var(N_n(I_n))}\geq c_0n^{1/4},
 \qquad c_0=\sqrt{V_\infty}.
 \label{eq:full-line-bulk-sd-from-growing-window}
\end{equation}

Since both zero counts are at most \(n\), we apply
Proposition~\ref{prop:full-line-low-band-localization} and obtain
\begin{equation}
 \|N_n(I_n)-N_{P_{m_n}}(I_n)\|_{L^2}
 \leq
 n\,\Pp\!\left(
  N_n(I_n)\neq N_{P_{m_n}}(I_n)
 \right)^{1/2}
 \leq
 Cn^{-7}.
 \label{eq:full-line-low-count-L2}
\end{equation}
We combine \eqref{eq:rare-disagreement-mean-sd-transfer} and
\eqref{eq:full-line-low-count-L2} to obtain
\begin{equation}
 |\E N_n(I_n)-\E N_{P_{m_n}}(I_n)|
 +
 \left|
  \sqrt{\Var(N_n(I_n))}
  -
  \sqrt{\Var(N_{P_{m_n}}(I_n))}
 \right|
 \leq Cn^{-7}.
 \label{eq:full-line-low-normalization}
\end{equation}
We use \eqref{eq:full-line-bulk-sd-from-growing-window} to bound the
left-hand side of \eqref{eq:full-line-low-normalization} by
\(\sqrt{\Var(N_n(I_n))}/2\).  Using the coupling estimate
\eqref{eq:full-line-low-band-localization},
\eqref{eq:full-line-bulk-BE-bound},
\eqref{eq:full-line-bulk-sd-from-growing-window} and
\eqref{eq:full-line-low-normalization}, we apply
Lemma~\ref{lem:finite-range-exact-coupling-transfer} with
\[
 X=N_n(I_n),
 \qquad
 Y=N_{P_{m_n}}(I_n).
\]
We obtain
\begin{equation}
 \sqrt{\Var(N_{P_{m_n}}(I_n))}
 \geq
 \frac{c_0}{2}n^{1/4}
 \label{eq:full-line-low-sd}
\end{equation}
and
\begin{equation}
 d_{\mathrm K}\!\left(
  \frac{
   N_{P_{m_n}}(I_n)-\E N_{P_{m_n}}(I_n)
  }{
   \sqrt{\Var(N_{P_{m_n}}(I_n))}
  },Z
 \right)
 \leq
 C\frac{\log n}{n^{1/4}}+Cn^{-16}+Cn^{-29/4}
 \leq
 C\frac{\log n}{n^{1/4}}.
 \label{eq:full-line-low-BE}
\end{equation}
We next add \(N_{P_n-P_{m_n}}(I_n^{\mathrm{out}})\), which is
independent of \(N_{P_{m_n}}(I_n)\).
We use this independence and \eqref{eq:full-line-low-sd} to obtain
\begin{equation}
 \begin{aligned}
 &\Var\!\left(
 N_{P_{m_n}}(I_n)
 +
 N_{P_n-P_{m_n}}(I_n^{\mathrm{out}})
 \right)\\
 {}={}&
  \Var(N_{P_{m_n}}(I_n))
  +
  \Var(N_{P_n-P_{m_n}}(I_n^{\mathrm{out}}))
  \geq
  \frac{c_0^2}{4}\sqrt n.
 \end{aligned}
 \label{eq:full-line-independent-sum-variance}
\end{equation}
Moreover, we apply
\eqref{eq:full-line-truncated-exterior-L2} and obtain
\[
 \Var(N_{P_n-P_{m_n}}(I_n^{\mathrm{out}}))
 \leq
 C(\log n)^2.
\]
We apply Lemma~\ref{lem:full-line-independent-transfer} with
\[
 U=N_{P_{m_n}}(I_n),
 \qquad
 V=N_{P_n-P_{m_n}}(I_n^{\mathrm{out}}).
\]
We combine \eqref{eq:full-line-low-sd} and
\eqref{eq:full-line-low-BE} to obtain
\begin{equation}
 d_{\mathrm K}\!\left(
  \frac{U+V-\E(U+V)}
  {\sqrt{\Var(U+V)}},Z
 \right)
 \leq
 C\frac{\log n}{n^{1/4}}
 +
 C\frac{(\log n)^2}{\sqrt n}
 \leq
 C\frac{\log n}{n^{1/4}}.
\label{eq:full-line-independent-sum-BE}
\end{equation}

We then incorporate \(N_n(I_n^{\mathrm{edge}})\).
We apply Proposition~\ref{prop:full-line-edge-strip-count} with \(A=8\).
Because we have fixed this threshold, there is an absolute \(C\) such that
\begin{equation}
 \Pp\!\left(N_n(I_n^{\mathrm{edge}})>C\log n\right)
 \leq n^{-8}.
 \label{eq:full-line-edge-tail-for-transfer}
\end{equation}
Since \(0\leq N_n(I_n^{\mathrm{edge}})\leq n\),
\[
 \E[N_n(I_n^{\mathrm{edge}})^2]
 \leq
 C(\log n)^2+n^2n^{-8}
 \leq
 C(\log n)^2.
\]
Consequently, we obtain
\begin{equation}
 \|N_n(I_n^{\mathrm{edge}})
   -\E N_n(I_n^{\mathrm{edge}})\|_{L^2}
 \leq C\log n,
 \qquad
 \E N_n(I_n^{\mathrm{edge}})\leq C\log n.
 \label{eq:full-line-edge-moments}
\end{equation}
On the event
\[
 \mathcal G_n^{\mathrm{edge}}
 =
 \{N_n(I_n^{\mathrm{edge}})\leq C\log n\},
\]
we therefore have
\begin{equation}
 |N_n(I_n^{\mathrm{edge}})
   -\E N_n(I_n^{\mathrm{edge}})|
 \leq C\log n.
 \label{eq:full-line-edge-centered-good-event}
\end{equation}

We define
\begin{equation}
 \widetilde N_n
 =
 N_{P_{m_n}}(I_n)
 +
 N_n(I_n^{\mathrm{edge}})
 +
 N_{P_n-P_{m_n}}(I_n^{\mathrm{out}}).
 \label{eq:full-line-approximating-count-definition}
\end{equation}
We use \eqref{eq:full-line-independent-sum-variance} and
\eqref{eq:full-line-edge-moments} to obtain
\[
 \|N_n(I_n^{\mathrm{edge}})
   -\E N_n(I_n^{\mathrm{edge}})\|_{L^2}
 \leq
 \frac12
 \sqrt{\Var\!\left(
  N_{P_{m_n}}(I_n)
  +
  N_{P_n-P_{m_n}}(I_n^{\mathrm{out}})
 \right)}.
\]
We apply Lemma~\ref{lem:full-line-additive-transfer} with
\[
 \begin{aligned}
 X&=N_{P_{m_n}}(I_n)
    +N_{P_n-P_{m_n}}(I_n^{\mathrm{out}}),
 \quad T=N_n(I_n^{\mathrm{edge}}),\\
 Y&=\widetilde N_n,
 \quad G=\mathcal G_n^{\mathrm{edge}},
 \quad p=n^{-8},
 \quad r=C\log n.
 \end{aligned}
\]
We use \eqref{eq:full-line-independent-sum-variance}--%
\eqref{eq:full-line-edge-centered-good-event} to obtain
\begin{equation}
 d_{\mathrm K}\!\left(
  \frac{\widetilde N_n-\E\widetilde N_n}
  {\sqrt{\Var(\widetilde N_n)}},Z
 \right)
 \leq
 C\frac{\log n}{n^{1/4}}
 \label{eq:full-line-approximating-total-BE}
\end{equation}
and
\begin{equation}
 \sqrt{\Var(\widetilde N_n)}
 \geq
 \frac12
 \sqrt{\Var\!\left(
  N_{P_{m_n}}(I_n)
  +
  N_{P_n-P_{m_n}}(I_n^{\mathrm{out}})
 \right)}
 \geq
 \frac{c_0}{4}n^{1/4}.
 \label{eq:full-line-approximating-total-sd}
\end{equation}

Finally, we compare \(\widetilde N_n\) with \(N_n(\R)\).
We define
\begin{equation}
 \mathcal D_n
 =
 \{N_n(I_n)=N_{P_{m_n}}(I_n)\}
 \cap
 \{N_n(I_n^{\mathrm{out}})
   =N_{P_n-P_{m_n}}(I_n^{\mathrm{out}})\}.
 \label{eq:full-line-final-coupling-event}
\end{equation}
We combine Propositions~\ref{prop:full-line-low-band-localization}
and~\ref{prop:full-line-upper-band-localization} with the definition
\eqref{eq:full-line-final-coupling-event} and obtain
\begin{equation}
 \Pp(\mathcal D_n^c)\leq Cn^{-16}.
 \label{eq:full-line-final-coupling-probability}
\end{equation}
On \(\mathcal D_n\), we combine the pointwise decomposition
\eqref{eq:full-line-count-identity} with the definition
\eqref{eq:full-line-approximating-count-definition} to obtain
\[
 \begin{aligned}
 N_n(\R)
 &=
 N_n(I_n)+N_n(I_n^{\mathrm{edge}})
 +N_n(I_n^{\mathrm{out}})\\
 &= 
  N_{P_{m_n}}(I_n)
  +N_n(I_n^{\mathrm{edge}})
  +N_{P_n-P_{m_n}}(I_n^{\mathrm{out}})
  =
  \widetilde N_n.
 \end{aligned}
\]
We apply the degree bounds and obtain
\[
 0\leq N_n(\R)\leq n,
 \qquad
 0\leq\widetilde N_n
 \leq m_n+n+d_n-1=2n-1<2n.
\]
We combine these degree bounds with
\eqref{eq:full-line-final-coupling-probability} and obtain
\begin{equation}
 \|N_n(\R)-\widetilde N_n\|_{L^2}
 \leq
 2n\,\Pp(\mathcal D_n^c)^{1/2}
 \leq
 Cn^{-7}.
 \label{eq:full-line-final-coupling-L2}
\end{equation}
We argue as in \eqref{eq:full-line-low-normalization} and use
\eqref{eq:full-line-final-coupling-L2} to obtain
\begin{equation}
 |\E N_n(\R)-\E\widetilde N_n|
 +
 \left|\sqrt{\Var(N_n(\R))}-\sqrt{\Var(\widetilde N_n)}\right|
 \leq Cn^{-7}.
 \label{eq:full-line-final-normalization}
\end{equation}
For all sufficiently large \(n\), we combine
\eqref{eq:full-line-approximating-total-sd} and
\eqref{eq:full-line-final-normalization} to obtain
\[
 |\E N_n(\R)-\E\widetilde N_n|
 +
 \left|\sqrt{\Var(N_n(\R))}-\sqrt{\Var(\widetilde N_n)}\right|
 \leq
 \frac12\sqrt{\Var(\widetilde N_n)}.
\]
We combine this bound with
\eqref{eq:full-line-approximating-total-BE}.  We apply
Lemma~\ref{lem:finite-range-exact-coupling-transfer} with
\[
 X=\widetilde N_n,
 \qquad
 Y=N_n(\R).
\]
We first use the lemma to conclude that \(\Var(N_n(\R))>0\) and then
obtain
\[
 d_{\mathrm K}\!\left(
  \frac{N_n(\R)-\E N_n(\R)}{\sqrt{\Var(N_n(\R))}},Z
 \right)
 \leq
 C\frac{\log n}{n^{1/4}}
 +Cn^{-16}
 +Cn^{-29/4}
 \leq
 C\frac{\log n}{n^{1/4}}.
\]
This proves \eqref{eq:full-line-polynomial-BE} and completes the proof of
Theorem~\ref{thm:full-line-polynomial-BE}.

\section{Proofs of the auxiliary results for Theorem~\ref{thm:polynomial-logarithmic-edge-BE}}
\label{sec:bulk-result-proofs}

We prove the results stated in Section~\ref{subsec:bulk-results} according
to the logical relationships among them.  We first prove
Proposition~\ref{prop:stationary-Weyl-zero-count-results}, then compare
\(N_n(I)\) with \(N_\infty(I)\), construct \(q_{\infty,m}\) and compare
\(N_{\infty,m}(I)\) with \(N_\infty(I)\).  We conclude by proving the
Berry--Esseen estimate for \(N_{\infty,\mu_\ell}(I)\).  Auxiliary results are
introduced immediately before their first use.

\subsection{Proof of Proposition~\ref{prop:stationary-Weyl-zero-count-results}}
We use stationarity of \(q_\infty\) to obtain
\eqref{eq:stationary-count-translation-in-law}.  For the other two
assertions, we apply Lemma~\ref{lem:appendix-do-vu-thm6} with
\(h=\mathbf1_{[0,1]}\) and \(R=\ell\).  The function \(h\) is nonzero, bounded
and compactly supported, with \(\|h\|_{L^{2}(\R)}^{2}=1\).
We use the simplicity and endpoint conclusions above to identify the
zero statistic in that theorem as
\[
 \sum_{x:P_\infty(x)=0}\mathbf1_{[0,1]}(x/\ell)
 =
 N_\infty([0,\ell])
 \qquad\text{almost surely}.
\]
We use the fourth-moment and variance assertions of the lemma to obtain
\eqref{eq:stationary-zero-count-fourth-moment} and
\eqref{eq:stationary-linear-variance-all-L}.

\subsection{Proofs of Proposition~\ref{prop:logarithmic-edge-root-pairing}
and Corollary~\ref{cor:logarithmic-edge-L2-comparison}}
\label{sec:zero-count-comparison}

We first prove Proposition~\ref{prop:logarithmic-edge-root-pairing} by
comparing \(N_n(I)\) with \(N_\infty(I)\) and then deduce
Corollary~\ref{cor:logarithmic-edge-L2-comparison}.

Throughout Sections~\ref{sec:zero-count-comparison}--%
\ref{sec:finite-range-berry-esseen}, \(I\) may depend on \(n\), although this
is not indicated in the notation.  Whenever \(I_{n,s}\) appears,
\(I\subseteq I_{n,s}\).

For \(1\leq s\leq\sqrt n/4\), we define
\begin{equation}
 I_{n,s}=[-\sqrt n+s,\sqrt n-s],
 \qquad
 R_n=P_\infty-P_n.
 \label{eq:variable-logarithmic-edge-interval}
\end{equation}
We first estimate \(R_n\) and its first two derivatives on \(I_{n,s}\).

\begin{lemma}
\label{lem:polynomial-tail-derivatives}
For every \(r\in\{0,1,2\}\), there is a finite absolute constant
\(C_r\) such that, for every \(1\leq s\leq\sqrt n/4\) and every
\(x\in I_{n,s}\), with \(I_{n,s}\) and \(R_n\) as in
\eqref{eq:variable-logarithmic-edge-interval},
\begin{equation}
 e^{-x^{2}}\E|R_n^{(r)}(x)|^{2}
 \leq
 C_r n^r e^{-s^{2}/4}.
 \label{eq:polynomial-tail-derivative-bound}
\end{equation}
\end{lemma}

\begin{proof}
We fix \(r\in\{0,1,2\}\) and let \(K\) be a Poisson random variable
with mean \(x^2\).  We apply
termwise differentiation, independence of the coefficients, Tonelli's
theorem and the substitution \(j=k-r\) to obtain
\begin{equation}
 e^{-x^{2}}\E|R_n^{(r)}(x)|^{2}
 =
 e^{-x^2}
 \sum_{k=n+1}^{\infty}
 \frac{(k)_r^{2}(x^2)^{k-r}}{k!}
 =
 \sum_{j=n+1-r}^{\infty}
 \frac{(j+r)!}{j!}\Pp(K=j),
 \label{eq:polynomial-tail-Poisson-identity}
\end{equation}
For \(r=0,1,2\), the factors in
\eqref{eq:polynomial-tail-Poisson-identity} are, respectively,
\[
 1,
 \qquad
 j+1=(j)_1+1,
 \qquad
 (j+1)(j+2)=(j)_2+4(j)_1+2.
\]
For integers \(a\geq0\) and \(m>a\), we sum directly and obtain the
truncated factorial-moment identity
\begin{equation}
 \E\!\left[(K)_a\mathbf 1_{\{K\geq m\}}\right]
 =
 \sum_{j=m}^{\infty}
 e^{-x^2}\frac{x^{2j}}{(j-a)!}
 =
 x^{2a}\sum_{u=m-a}^{\infty}
 e^{-x^2}\frac{x^{2u}}{u!}
 =
 x^{2a}\Pp(K\geq m-a).
 \label{eq:truncated-Poisson-factorial-moment}
\end{equation}
We apply \eqref{eq:truncated-Poisson-factorial-moment} term by term
to \eqref{eq:polynomial-tail-Poisson-identity}, use
\(x^2\leq n\) and obtain
\begin{equation}
 e^{-x^{2}}\E|R_n^{(r)}(x)|^{2}
 \leq
 C_r n^r\Pp(K\geq n+1-2r).
 \label{eq:polynomial-tail-reduced-to-Poisson}
\end{equation}

In the preceding identity, take \(m=n+1-2r\).  Since \(x\in I_{n,s}\),
\[
 x^2
 \leq
 (\sqrt n-s)^{2}
 =
 n-2s\sqrt n+s^{2}.
\]
For \(r\leq2\) and \(s\geq1\), we use the restriction
\(s\leq\sqrt n/4\) to obtain
\begin{equation}
 m-x^2
 \geq
 2s\sqrt n-s^{2}-3
 \geq
 s\sqrt n.
 \label{eq:polynomial-tail-Poisson-gap}
\end{equation}
In particular, \(m>x^2\) and \(m\leq2n\).

We obtain the claimed bound directly when \(x=0\).  For \(x\neq0\), we apply
Chernoff's argument with \(\theta=\log(m/x^2)\) and use
\(-\log(1-u)-u\geq u^2/2\), where \(u=(m-x^2)/m\).  We then use
\eqref{eq:polynomial-tail-Poisson-gap} and \(m\leq2n\) to obtain
\begin{equation}
 \Pp(K\geq m)
 \leq
 \exp\!\left\{-\frac{(m-x^2)^{2}}{2m}\right\}
 \leq
 \exp\!\left\{-\frac{s^{2}n}{4n}\right\}
 =
 e^{-s^{2}/4}.
 \label{eq:proved-Poisson-Chernoff-bound}
\end{equation}
We substitute \eqref{eq:proved-Poisson-Chernoff-bound} into
\eqref{eq:polynomial-tail-reduced-to-Poisson} to prove
\eqref{eq:polynomial-tail-derivative-bound}.
\end{proof}

The tail bounds now combine with the interval Sobolev estimate to control
\(q_n-q_\infty\) in \(C^1\) on a target interval whose length may grow with
\(n\).

\begin{lemma}
\label{lem:polynomial-series-growing-C1}
There is a finite absolute constant \(C\) such that, for every
\(1\leq s\leq\sqrt n/4\) and every \(I\subseteq I_{n,s}\) with
\(\ell\geq1\),
\begin{equation}
 \E\|q_n-q_\infty\|_{C^1(I)}^{2}
 \leq
 C\ell n^{2}e^{-s^{2}/4}.
 \label{eq:polynomial-series-growing-C1-L2}
\end{equation}
Consequently, for every \(\varepsilon>0\),
\begin{equation}
 \Pp\!\left(
  \|q_n-q_\infty\|_{C^1(I)}\geq\varepsilon
 \right)
 \leq
 C\ell n^{2}e^{-s^{2}/4}\varepsilon^{-2}.
 \label{eq:polynomial-series-growing-C1-probability}
\end{equation}
\end{lemma}

\begin{proof}
We set
\begin{equation}
 \rho_n(x)=q_\infty(x)-q_n(x)=e^{-x^{2}/2}R_n(x),
 \qquad x\in\R.
 \label{eq:polynomial-series-error-definition}
\end{equation}
We differentiate \eqref{eq:polynomial-series-error-definition}
directly and obtain
\begin{equation}
 \begin{aligned}
  \rho_n'(x)
  &=
  e^{-x^{2}/2}\{R_n'(x)-xR_n(x)\},\\
  \rho_n''(x)
  &=
  e^{-x^{2}/2}
  \{R_n''(x)-2xR_n'(x)+(x^{2}-1)R_n(x)\}.
 \end{aligned}
 \label{eq:polynomial-series-error-derivatives}
\end{equation}
We apply
\((u_1+\cdots+u_k)^{2}\leq k\sum_{i=1}^ku_i^{2}\),
Lemma~\ref{lem:polynomial-tail-derivatives} and \(x^{2}\leq n\) to
obtain
\begin{equation}
 \sup_{x\in I_{n,s}}\E|\rho_n^{(j)}(x)|^{2}
 \leq
 Cn^{2}e^{-s^{2}/4},
 \qquad j\in\{0,1,2\}.
 \label{eq:polynomial-series-pointwise-derivative-L2}
\end{equation}

We partition \(I\) into \(\lceil\ell\rceil\) compact intervals
\(J_1,\ldots,J_{\lceil\ell\rceil}\) with disjoint interiors and equal
length.  Every \(J_r\) satisfies
\[
 \frac12
 \leq
 |J_r|
 =
 \frac{\ell}{\lceil\ell\rceil}
 \leq1.
\]
We apply Lemma~\ref{lem:comparison-interval-Sobolev} to \(\rho_n\) and
\(\rho_n'\) on each \(J_r\), sum the resulting bounds, use
\((u+v)^{2}\leq2u^{2}+2v^{2}\) and obtain
\begin{equation}
 \|\rho_n\|_{C^1(I)}^{2}
 \leq
 C\sum_{r=1}^{\lceil\ell\rceil}
 \int_{J_r}
 \bigl(
   (\rho_n(t))^{2}+(\rho_n'(t))^{2}+(\rho_n''(t))^{2}
 \bigr)\dd t.
 \label{eq:polynomial-series-C1-interval-sum}
\end{equation}
The intervals \(J_r\) have disjoint interiors, so their integrals add to
the integral over \(I\).  We take expectations in
\eqref{eq:polynomial-series-C1-interval-sum}, apply Tonelli's theorem
and use \eqref{eq:polynomial-series-pointwise-derivative-L2} to obtain
\eqref{eq:polynomial-series-growing-C1-L2},
\[
 \E\|\rho_n\|_{C^1(I)}^{2}
 \leq
 C\sum_{r=1}^{\lceil\ell\rceil}\int_{J_r}\sum_{j=0}^2
 \E|\rho_n^{(j)}(t)|^2\dd t
 \leq
 C\ell n^2e^{-s^2/4}.
\]
We apply Markov's inequality to
\(\|\rho_n\|_{C^1(I)}^{2}\) and use
\eqref{eq:polynomial-series-growing-C1-L2} to obtain
\eqref{eq:polynomial-series-growing-C1-probability},
\[
 \Pp\!\left(
  \|\rho_n\|_{C^1(I)}\geq\varepsilon
 \right)
 \leq
 \varepsilon^{-2}
 \E\|\rho_n\|_{C^1(I)}^2
 \leq
 C\ell n^2e^{-s^2/4}\varepsilon^{-2}.
\]
\end{proof}

We next establish the derivative supremum and small-ball estimates for
\(q_\infty\) used below.

\begin{lemma}
\label{lem:comparison-stationary-derivative-supremum}
For every even integer \(p\geq2\), there is a finite constant \(C_p\)
such that, for every \(I\subset\R\) with \(\ell\geq1\) and every
\(M\geq1\),
\begin{equation}
 \Pp\!\left(
  \sup_{t\in I}
  \bigl(|q_\infty'(t)|+|q_\infty''(t)|\bigr)>M
 \right)
 \leq
 C_p(\ell+1)M^{-p}.
 \label{eq:comparison-stationary-derivative-supremum}
\end{equation}
\end{lemma}

\begin{proof}
We differentiate the covariance function \(e^{-u^{2}/2}\) from
\eqref{eq:stationary-weyl-covariance}
in its two variables and then set the variables equal to obtain
\[
 \Var(q_\infty'(t))=1,
 \qquad
 \Var(q_\infty''(t))=3,
 \qquad
 \Var(q_\infty^{(3)}(t))=15.
\]
Each of these derivatives is centered Gaussian.  Hence we apply the
Gaussian even-moment identity and obtain
\begin{equation}
 \E|q_\infty^{(j)}(t)|^p
 =
 (p-1)!!\,\Var(q_\infty^{(j)}(t))^{p/2},
 \qquad
 j\in\{1,2,3\}.
 \label{eq:comparison-stationary-derivative-moments}
\end{equation}
We apply \((u+v)^p\leq2^{p-1}(u^p+v^p)\), then apply
\eqref{eq:comparison-interval-Sobolev} on a unit interval to
\(q_\infty'\) and
\(q_\infty''\) and finally apply Tonelli's theorem and
\eqref{eq:comparison-stationary-derivative-moments} to obtain
\begin{equation}
 \E\!\left[\sup_{t\in[c,c+1]}
 \bigl(|q_\infty'(t)|+|q_\infty''(t)|\bigr)^p\right]
 \leq 2^{2p-2}(p-1)!!
 \left(1+2\cdot3^{p/2}+15^{p/2}\right)=C_p,
 \label{eq:comparison-unit-derivative-supremum-moment}
\end{equation}
where \(C_p\) is the constant in the statement of the lemma.

Let \(J_1,\ldots,J_{\lceil\ell\rceil}\) be unit intervals covering
\(I\).  We use
\eqref{eq:comparison-unit-derivative-supremum-moment}, Markov's
inequality and a union bound to obtain
\eqref{eq:comparison-stationary-derivative-supremum},
\[
 \begin{aligned}
 &\Pp\!\left(
  \sup_{t\in I}
  \bigl(|q_\infty'(t)|+|q_\infty''(t)|\bigr)>M
 \right)\\
 {}\leq{}&
 \sum_{j=1}^{\lceil\ell\rceil}
 \Pp\!\left(
  \sup_{t\in J_j}
  \bigl(|q_\infty'(t)|+|q_\infty''(t)|\bigr)>M
 \right)\\
 {}\leq{}&
 M^{-p}\sum_{j=1}^{\lceil\ell\rceil}
 \E\!\left[
  \sup_{t\in J_j}
  \bigl(|q_\infty'(t)|+|q_\infty''(t)|\bigr)^p
 \right]\\
 {}\leq{}&
 C_p\lceil\ell\rceil M^{-p}
 \leq C_p(\ell+1)M^{-p}.
 \end{aligned}
\]
\end{proof}

Combining the preceding derivative control with the nondegenerate Gaussian
law of \((q_\infty(t),q_\infty'(t))\), we bound the probability that there
exists \(t\in I\) for which both \(|q_\infty(t)|\) and
\(|q_\infty'(t)|\) are small, as well as the probability that
\(|q_\infty|\) is small at either endpoint of \(I\).  These estimates will
be used to verify \eqref{eq:comparison-zero-stability-transversality} and
\eqref{eq:comparison-zero-stability-endpoints} when
Lemma~\ref{lem:comparison-order-preserving-stability} is applied.

\begin{lemma}
\label{lem:comparison-near-critical-small-ball}
Let \(C_p\) be as in the preceding lemma.  There is a finite absolute
constant \(C_0\) such that, for every even integer \(p\geq2\), every
\(I=[\varphi,\psi]\) with \(\ell\geq1\), every \(M\geq1\) and every
\(0<\delta<1\),
\begin{equation}
 \Pp\!\left(
  \inf_{t\in I}
  \max\{|q_\infty(t)|,|q_\infty'(t)|\}
  \leq2\delta
 \right)
 \leq C_p(\ell+1)M^{-p}
 +C_0(\ell+1)M\delta
 +C_0\delta^{2}
 \label{eq:comparison-near-critical-small-ball}
\end{equation}
and
\begin{equation}
 \Pp\!\left(
  \min\{|q_\infty(\varphi)|,|q_\infty(\psi)|\}
  \leq2\varepsilon
 \right)
 \leq
  C_0\varepsilon,
  \qquad 0<\varepsilon<1.
 \label{eq:comparison-endpoint-small-ball}
\end{equation}
\end{lemma}

\begin{proof}
We take \(C_0=36/\pi\).
We define
\begin{equation}
 \mathcal H_{I,M}
 =
 \left\{
  \sup_{t\in I}
  \bigl(|q_\infty'(t)|+|q_\infty''(t)|\bigr)
  \leq M
 \right\}.
 \label{eq:comparison-derivative-good-event}
\end{equation}
We set
\begin{equation}
 Q=\left\lceil\frac{\ell M}{\delta}\right\rceil,
 \qquad
 \mathcal D_{I,M,\delta}
 =
  \left\{\varphi+\frac{j\ell}{Q}:0\leq j\leq Q\right\}.
 \label{eq:comparison-grid-definition}
\end{equation}
The grid in \eqref{eq:comparison-grid-definition} contains both
endpoints and satisfies
\begin{equation}
 \frac{\ell}{Q}\leq\frac{\delta}{M},
 \qquad
 \operatorname{card}(\mathcal D_{I,M,\delta})
 =
 Q+1
 \leq
 2+\frac{\ell M}{\delta}.
 \label{eq:comparison-grid-size-and-mesh}
\end{equation}
The quantity
\(\max\{|q_\infty(t)|,|q_\infty'(t)|\}\) depends continuously on
\(t\) and therefore attains its minimum on \(I\).  Thus, on the event
appearing on the left-hand side of
\eqref{eq:comparison-near-critical-small-ball}, we choose a point
\(t_*\in I\) satisfying
\[
 |q_\infty(t_*)|\leq2\delta,
 \qquad
 |q_\infty'(t_*)|\leq2\delta.
\]
 We use Equations~\eqref{eq:comparison-grid-definition}
and~\eqref{eq:comparison-grid-size-and-mesh} to choose
\(v\in\mathcal D_{I,M,\delta}\) with
\(|v-t_*|\leq\delta/M\).  On the event
\(\mathcal H_{I,M}\) from \eqref{eq:comparison-derivative-good-event},
we apply the fundamental theorem of calculus and obtain
\[
 |q_\infty(v)|\leq3\delta,
 \qquad
 |q_\infty'(v)|\leq3\delta.
\]
We have therefore proved the event inclusion
\begin{equation}
 \begin{aligned}
  &\left\{
   \inf_{t\in I}
   \max\{|q_\infty(t)|,|q_\infty'(t)|\}
   \leq2\delta
  \right\}
  \cap\mathcal H_{I,M}\\
  {}\subseteq{}&
  \bigcup_{v\in\mathcal D_{I,M,\delta}}
  \left\{
   |q_\infty(v)|\leq3\delta,
   |q_\infty'(v)|\leq3\delta
  \right\}.
 \end{aligned}
 \label{eq:comparison-near-critical-event-inclusion}
\end{equation}

We use Equation~\eqref{eq:stationary-value-derivative-covariance} and
joint Gaussianity to identify \(q_\infty(v)\) and \(q_\infty'(v)\) as
independent standard Gaussian random variables.  Their joint density is at most
\((2\pi)^{-1}\), so
\begin{equation}
 \Pp\!\left(
  |q_\infty(v)|\leq3\delta,
  |q_\infty'(v)|\leq3\delta
 \right)
 \leq
 \frac{18}{\pi}\delta^{2}.
 \label{eq:comparison-near-critical-grid-small-ball}
\end{equation}
Splitting according to \(\mathcal H_{I,M}\), we use
Lemma~\ref{lem:comparison-stationary-derivative-supremum},
\eqref{eq:comparison-grid-size-and-mesh}--%
\eqref{eq:comparison-near-critical-grid-small-ball} and a union bound to obtain
\[
 \begin{aligned}
  &\Pp\!\left(
   \inf_{t\in I}
   \max\{|q_\infty(t)|,|q_\infty'(t)|\}
   \leq2\delta
  \right)\\
  {}\leq{}&
  C_p(\ell+1)M^{-p}
  +
  \frac{18}{\pi}
  \left(2+\frac{\ell M}{\delta}\right)\delta^{2}\\
  {}\leq{}&
  C_p(\ell+1)M^{-p}
  +
  C_0(\ell+1)M\delta
  +
  C_0\delta^{2}.
 \end{aligned}
\]
We thereby obtain \eqref{eq:comparison-near-critical-small-ball}.

Finally, \(q_\infty(\varphi)\) and \(q_\infty(\psi)\) are both standard
Gaussian random variables.  The standard Gaussian density is at most
\((2\pi)^{-1/2}\).  We apply a union bound and obtain
\[
 \Pp\!\left(
  \min\{|q_\infty(\varphi)|,|q_\infty(\psi)|\}
  \leq2\varepsilon
 \right)
 \leq
 \frac{8}{\sqrt{2\pi}}\varepsilon.
\]
We thereby obtain \eqref{eq:comparison-endpoint-small-ball}.
\end{proof}

We then prove a deterministic criterion for matching the ordered zeros of
two functions that are close in the \(C^1\) norm.

\begin{lemma}
\label{lem:comparison-order-preserving-stability}
Let \(I=[\varphi,\psi]\) be a compact interval.  Suppose that the functions
\(f,h\in C^1(I)\) and the numbers
\(0<\varepsilon<\delta\) satisfy
 \begin{equation}
  \|f-h\|_{C^1(I)}<\varepsilon,
  \label{eq:comparison-zero-stability-C1}
 \end{equation}
 \begin{equation}
  \inf_{t\in I}\max\{|f(t)|,|f'(t)|\}>2\delta
  \label{eq:comparison-zero-stability-transversality}
 \end{equation}
 and
 \begin{equation}
  \min\{|f(\varphi)|,|f(\psi)|\}>2\varepsilon.
  \label{eq:comparison-zero-stability-endpoints}
 \end{equation}
Then \(f\) and \(h\) have the same finite number \(k\) of distinct
zeros in \([\varphi,\psi]\).  Neither function vanishes at an endpoint.  If
\(k\geq1\), let
\[
 x_1<\cdots<x_k,
 \qquad
 y_1<\cdots<y_k
\]
denote the zeros of \(f\) and \(h\), respectively, in increasing order.
Then
\begin{equation}
 \max_{1\leq j\leq k}|x_j-y_j|
 \leq
 \frac{\varepsilon}{2\delta-\varepsilon}.
 \label{eq:comparison-zero-stability-displacement}
\end{equation}
\end{lemma}

\begin{proof}
We define
\begin{equation}
 f_u=f+u(h-f),
 \label{eq:comparison-homotopy-definition}
\end{equation}
where \(u\in\R\), and restrict attention to \(u\in[0,1]\) below.  If the
function in
\eqref{eq:comparison-homotopy-definition} satisfies \(f_u(t)=0\), then
 \begin{equation}
  |f(t)|
  =
  u|f(t)-h(t)|
  <
  \varepsilon
  <
  \delta,
  \label{eq:comparison-homotopy-reference-value}
 \end{equation}
 \begin{equation}
  |f'(t)|>2\delta
  \label{eq:comparison-homotopy-reference-derivative}
 \end{equation}
 and
 \begin{equation}
  |f_u'(t)|
  \geq
  |f'(t)|-u|h'(t)-f'(t)|
  >
  2\delta-\varepsilon
  >
  0.
  \label{eq:comparison-homotopy-derivative}
 \end{equation}
We combine \eqref{eq:comparison-zero-stability-transversality} with
\eqref{eq:comparison-homotopy-reference-value} to obtain
\eqref{eq:comparison-homotopy-reference-derivative}.  We then combine
\eqref{eq:comparison-zero-stability-C1} with
\eqref{eq:comparison-homotopy-reference-derivative} to obtain
\eqref{eq:comparison-homotopy-derivative}.
For \(a\in\{\varphi,\psi\}\), we use
\eqref{eq:comparison-zero-stability-C1} and
\eqref{eq:comparison-zero-stability-endpoints} to obtain
\begin{equation}
 |f_u(a)|
 \geq
 |f(a)|-u|h(a)-f(a)|
 >
 2\varepsilon-\varepsilon
 =
 \varepsilon.
 \label{eq:comparison-homotopy-endpoint-separation}
\end{equation}
We use \eqref{eq:comparison-homotopy-endpoint-separation} to conclude
that every zero of every \(f_u\), \(u\in[0,1]\), lies in \((\varphi,\psi)\)
and is simple.

For \(u\in[0,1]\), we define
\begin{equation}
  Z_u=\{t\in I:f_u(t)=0\}.
 \label{eq:comparison-homotopy-zero-set}
\end{equation}
The zero set in \eqref{eq:comparison-homotopy-zero-set} is compact.
If it were infinite, we would use compactness to choose distinct
points \(t_r\in Z_u\) converging to a point
\(t_*\in Z_u\).  Since \(t_*\in(\varphi,\psi)\), we would use differentiability
to obtain
\[
 f_u'(t_*)
 =
 \lim_{r\to\infty}
 \frac{f_u(t_r)-f_u(t_*)}{t_r-t_*}
 =
 0,
 \]
 contradicting \eqref{eq:comparison-homotopy-derivative}.  We conclude
 that
\begin{equation}
 \operatorname{card}(Z_u)<\infty,
 \qquad u\in[0,1].
 \label{eq:comparison-homotopy-finite-zero-set}
\end{equation}

We fix \(u_0\in[0,1]\), set \(r=\operatorname{card}(Z_{u_0})\) and,
when \(r\geq1\), list the zeros as
\begin{equation}
 Z_{u_0}=\{z_1<\cdots<z_r\}.
 \label{eq:comparison-local-zero-enumeration}
\end{equation}
At each point \((u_0,z_i)\) in
\eqref{eq:comparison-local-zero-enumeration}, the function
\(F(u,t)=f_u(t)\) is \(C^1\) and
has nonzero \(t\)-derivative.
We therefore apply the implicit function theorem and, after shrinking
finitely many neighborhoods, obtain a common \(\eta>0\), pairwise
disjoint open intervals
\begin{equation}
  J_1,\ldots,J_r\subset(\varphi,\psi),
 \qquad
 z_i\in J_i,
 \qquad
 \sup J_i<\inf J_{i+1}
 \label{eq:comparison-local-branch-intervals}
\end{equation}
and \(C^1\) functions \(\zeta_i\), each defined on
\((u_0-\eta,u_0+\eta)\) and taking values in \(J_i\),
such that, whenever \(u\in[0,1]\) and \(|u-u_0|<\eta\),
\begin{equation}
 \{t\in J_i:f_u(t)=0\}
 =
 \{\zeta_i(u)\}.
 \label{eq:comparison-local-zero-branches}
\end{equation}
We use the intervals in \eqref{eq:comparison-local-branch-intervals}
and set
\begin{equation}
 K=I\setminus\bigcup_{i=1}^rJ_i.
 \label{eq:comparison-zero-free-complement}
\end{equation}
The set \(K\) in \eqref{eq:comparison-zero-free-complement} is compact
and contains no zero of \(f_{u_0}\), so
\[
 \min_{t\in K}|f_{u_0}(t)|>0.
\]
We reduce \(\eta\), if necessary, so that
\[
 \eta\varepsilon
 <
 \frac12\min_{t\in K}|f_{u_0}(t)|.
\]
For \(u\in[0,1]\) with \(|u-u_0|<\eta\),
\[
 \inf_{t\in K}|f_u(t)|
 \geq
 \min_{t\in K}|f_{u_0}(t)|
 -
 |u-u_0|\|h-f\|_\infty
 >
 \frac12\min_{t\in K}|f_{u_0}(t)|
 >
 0.
\]
Thus \(Z_u\cap K=\varnothing\).  We combine this fact with
\eqref{eq:comparison-local-zero-branches} and obtain
\begin{equation}
 Z_u
 =
 \bigcup_{i=1}^r\{\zeta_i(u)\},
 \qquad
 \operatorname{card}(Z_u)=r
 \label{eq:comparison-local-branch-cover}
\end{equation}
whenever \(u\in[0,1]\) and \(|u-u_0|<\eta\).

For \(j\in\mathbb N_0\), we define
\begin{equation}
 A_j
 =
 \{u\in[0,1]:\operatorname{card}(Z_u)=j\}.
 \label{eq:comparison-cardinality-level-set}
\end{equation}
We use the local representation in
\eqref{eq:comparison-local-branch-cover} to conclude that every set
in \eqref{eq:comparison-cardinality-level-set} is relatively open in
\([0,1]\).  Since
\[
 [0,1]\setminus A_j
 =
 \bigcup_{r\in\mathbb N_0\setminus\{j\}}A_r,
\]
each \(A_j\) is also relatively closed.  We use
\eqref{eq:comparison-homotopy-finite-zero-set} and
\eqref{eq:comparison-cardinality-level-set} to conclude that the
sets \(A_j\) form a disjoint partition of \([0,1]\).  Because
\([0,1]\) is connected, we conclude that
\begin{equation}
 \operatorname{card}(Z_u)
 =
 \operatorname{card}(Z_0)
 =
 \operatorname{card}(Z_1)
 =:k,
 \qquad u\in[0,1].
 \label{eq:comparison-homotopy-constant-count}
\end{equation}
Since \(f_0=f\) and \(f_1=h\), we obtain equality and finiteness of the
two zero counts.

We assume that \(k\geq1\) and write
\begin{equation}
 Z_u=\{\gamma_1(u)<\cdots<\gamma_k(u)\},
 \qquad u\in[0,1].
 \label{eq:comparison-global-zero-branches}
\end{equation}
We use the local representation above and
\eqref{eq:comparison-homotopy-constant-count} to conclude that the
globally ordered zero \(\gamma_i\) agrees locally with the
corresponding implicit-function branch \(\zeta_i\).  Hence local branches agree on
their overlaps and
\begin{equation}
 \gamma_i\in C^1([0,1]),
 \qquad
 \gamma_i(0)=x_i,
 \qquad
 \gamma_i(1)=y_i.
 \label{eq:comparison-global-branch-endpoints}
\end{equation}
We differentiate the identity
\(f_u(\gamma_i(u))=0\) from
\eqref{eq:comparison-global-zero-branches} and apply
\eqref{eq:comparison-homotopy-derivative} to obtain
\[
 |\gamma_i'(u)|
 =
 \frac{|h(\gamma_i(u))-f(\gamma_i(u))|}
 {|f_u'(\gamma_i(u))|}
 <
 \frac{\varepsilon}{2\delta-\varepsilon}.
\]
We integrate the derivative estimate and use
\eqref{eq:comparison-global-branch-endpoints} to obtain
\[
 |x_i-y_i|
 \leq
 \int_0^1|\gamma_i'(u)|\dd u
 \leq
 \frac{\varepsilon}{2\delta-\varepsilon}.
\]
We take the maximum over \(1\leq i\leq k\) to prove
\eqref{eq:comparison-zero-stability-displacement}.
\end{proof}

We now combine these estimates with the common coefficient coupling to prove
Proposition~\ref{prop:logarithmic-edge-root-pairing}.

\begin{proof}[Proof of Proposition~\ref{prop:logarithmic-edge-root-pairing}]
We fix \(I=[\varphi,\psi]\subseteq I_n\) with \(\ell\geq1\).  We use the same
estimates below for every
such interval.  Only the resulting good event depends on \(I\).
Recall the deterministic scales \(\delta_n\) and \(\varepsilon_n\) fixed in
\eqref{eq:logarithmic-edge-parameter-choice}.
From Equations \eqref{eq:logarithmic-edge-parameter-choice} and
\eqref{eq:logarithmic-edge-length-bounds}, we obtain
\[
 1\leq B\sqrt{\log n}\leq\frac{\sqrt n}{4},
 \qquad
 |I_n|=2L_n\leq2\sqrt n
\]
and
\[
 1\leq\ell\leq|I_n|,
 \qquad
 0<\varepsilon_n<\delta_n<1.
\]

We apply Lemma~\ref{lem:polynomial-series-growing-C1} on \(I\), with
\(s=B\sqrt{\log n}\) and
\(\varepsilon=\varepsilon_n\).  This is the first estimate in which
the size of \(B\) matters.  The condition
\[
 2-\frac{B^{2}}4+76\leq-22
\]
is equivalent to \(B^{2}\geq400\).  We therefore take the convenient
admissible threshold \(B_0=20\).  For
\(B\geq B_0\),
 \begin{equation}
  \begin{aligned}
   &\Pp\!\left(
    \|q_n-q_\infty\|_{C^1(I)}\geq\varepsilon_n
   \right)\\
   {}\leq{}&
   C\ell n^{2-B^{2}/4+76}
   \leq
   C\ell n^{2-100+76}
   =
   C\ell n^{-22}
   \leq
   C\ell n^{-18}.
  \end{aligned}
  \label{eq:logarithmic-edge-bad-C1}
 \end{equation}

We apply Lemma~\ref{lem:comparison-near-critical-small-ball} to \(I\)
with \(p=36\), \(M=\sqrt n\) and
\(\delta=\delta_n\).  We retain the dependence on the target
length and obtain
 \begin{equation}
  \begin{aligned}
   &\Pp\!\left(
    \inf_{t\in I}
    \max\{|q_\infty(t)|,|q_\infty'(t)|\}
    \leq2\delta_n
   \right)\\
   {}\leq{}&
   C_{36}(\ell+1)n^{-18}
   +
   C(\ell+1)n^{-37/2}
   +
   Cn^{-38}\\
   {}\leq{}&
   C\ell n^{-18}.
  \end{aligned}
  \label{eq:logarithmic-edge-bad-critical}
 \end{equation}
We apply \eqref{eq:comparison-endpoint-small-ball} with
\(\varepsilon=\varepsilon_n\) to obtain
\begin{equation}
 \Pp\!\left(
  \min\{|q_\infty(\varphi)|,|q_\infty(\psi)|\}
  \leq2\varepsilon_n
 \right)
 \leq
 Cn^{-38}
 \leq
 C\ell n^{-18}.
 \label{eq:logarithmic-edge-bad-endpoint}
\end{equation}

We use Equations~\eqref{eq:logarithmic-edge-bad-C1}--%
\eqref{eq:logarithmic-edge-bad-endpoint} to bound the probabilities
of the three bad events.  We define \(\mathcal G_{n,I}\) as the
intersection of \(\Omega_{\mathrm W}\) and the complements of these
three bad events.
Since \(\Pp(\Omega_{\mathrm W})=1\), we apply the union bound and obtain
\[
 \Pp(\mathcal G_{n,I}^c)
 \leq
 C\ell n^{-18}
 \leq
 Cn^{-17},
\]
where the last inequality uses \(\ell\leq|I_n|\leq2\sqrt n\).

We fix an outcome in \(\mathcal G_{n,I}\).  In
Lemma~\ref{lem:comparison-order-preserving-stability} on \(I\), we
take
\[
 f=q_\infty,
 \qquad
 h=q_n,
 \qquad
 \delta=\delta_n,
 \qquad
 \varepsilon=\varepsilon_n.
\]
On \(\mathcal G_{n,I}\), we use the complements of the three bad events
to verify, respectively, the conditions in
\eqref{eq:comparison-zero-stability-C1}--%
\eqref{eq:comparison-zero-stability-endpoints}.  We therefore apply
Lemma~\ref{lem:comparison-order-preserving-stability} and obtain that the zero
sets of \(q_\infty\) and \(q_n\) in \(I\) have the same cardinality.
 We also use the proof of that lemma to conclude that both functions have only
 simple zeros in \(I\) on \(\mathcal G_{n,I}\), so these numbers equal their respective
zero counts with multiplicity.  Since multiplication by
\(e^{-x^{2}/2}>0\) does not change zeros, these counts are precisely
\(N_n(I)\) and \(N_\infty(I)\).

If the common count is positive, we apply
Lemma~\ref{lem:comparison-order-preserving-stability} and obtain
\[
 \max_{1\leq j\leq k}|x_j-y_j|
 \leq
 \frac{\varepsilon_n}{2\delta_n-\varepsilon_n}
 =
 \frac{n^{-19}}{2-n^{-19}}
 \leq
 n^{-19}
 \leq
 n^{-1}.
\]
We combine \eqref{eq:logarithmic-edge-good-event-probability} with the
deterministic conclusions on \(\mathcal G_{n,I}\) to obtain
\eqref{eq:logarithmic-edge-root-displacement} and complete the proof.
\end{proof}

We finally use the resulting high-probability pairing to prove
Corollary~\ref{cor:logarithmic-edge-L2-comparison}.

\begin{proof}[Proof of Corollary~\ref{cor:logarithmic-edge-L2-comparison}]
We use the degree bound \eqref{eq:polynomial-zero-count-degree-bound}
to obtain \(\E[N_n(I)^4]\leq n^4\).  We use
Proposition~\ref{prop:stationary-Weyl-zero-count-results} and
\(\ell\leq|I_n|\leq2\sqrt n\) to obtain
\[
 \E[N_\infty(I)^4]
 =
 \E[N_\infty([0,\ell])^4]
 \leq
 C\ell^4
 \leq
 Cn^{2}.
\]
We apply Proposition~\ref{prop:logarithmic-edge-root-pairing} and
Lemma~\ref{lem:rare-disagreement-moment-transfer} to obtain
\eqref{eq:logarithmic-edge-L2-comparison},
\[
 \|N_n(I)-N_\infty(I)\|_{L^{2}}
 \leq
 C(n^4+\ell^4)^{1/4}(\ell n^{-18})^{1/4}
 \leq C\ell^{1/4}n^{-7/2}
 \leq Cn^{-3}.
\]
We apply \eqref{eq:rare-disagreement-mean-sd-transfer} and enlarge
\(C\) to obtain \eqref{eq:logarithmic-edge-normalization-comparison},
\[
 \begin{aligned}
  &|\E N_n(I)-\E N_\infty(I)|
  +
  \left|
   \sqrt{\Var(N_n(I))}
   -
   \sqrt{\Var(N_\infty(I))}
  \right|\\
  {}\leq{}&
  2\|N_n(I)-N_\infty(I)\|_{L^{2}}
  \leq C\ell^{1/4}n^{-7/2}
  \leq Cn^{-3}.
 \end{aligned}
\]
\end{proof}

\subsection{Finite-range approximation and comparison of zero sets}
\label{sec:finite-range-approximation}

We construct a stationary Gaussian process with finite dependence
range and compare its zeros with those of \(q_\infty\).  We give the
complete compactly supported moving-average construction and combine
it with the small-ball estimates and the deterministic comparison of
zeros from
Section~\ref{sec:zero-count-comparison}.

\subsubsection{A common moving-average realization}

We begin the promised construction of the white noise \(W\) fixed
notationally in Section~\ref{subsec:bulk-results}.  Recall that it is
required to be an isonormal Gaussian process over \(L^{2}(\R)\).  Thus,
for every \(h\in L^{2}(\R)\), \(W(h)\) is centered Gaussian and
\begin{equation}
 \E[W(h)W(k)]
 =\langle h,k\rangle_{L^{2}(\R)},
 \qquad h,k\in L^{2}(\R).
 \label{eq:white-noise-isometry}
\end{equation}
We also use the stochastic-integral notation
\(W(h)=\int_{\R}h(u)\,W(\mathrm du)\).
Recall that the deterministic kernel introduced in
Section~\ref{subsec:bulk-results} is
\begin{equation}
 g(u)=\left(\frac2\pi\right)^{1/4}e^{-u^{2}}.
 \label{eq:gaussian-moving-average-kernel}
\end{equation}
For the kernel in \eqref{eq:gaussian-moving-average-kernel}, we
evaluate the Gaussian integrals and obtain
\begin{equation}
 \|g'\|_2^{2}=1,
 \qquad
 \|g''\|_2^{2}=3,
 \qquad
 \|g^{(3)}\|_2^{2}=15.
 \label{eq:explicit-kernel-derivative-norms}
\end{equation}
We first realize \(q_\infty\) as a Gaussian moving average on an enlargement
of the coefficient space, so that the compactly supported approximations can
later be built from the same noise and coupled pathwise to \(q_\infty\).

\begin{lemma}
\label{lem:common-white-noise-realization}
After enlarging the coefficient probability space by independent
Gaussian randomness, the white-noise notation above can be realized so
that, almost surely,
\begin{equation}
 q_\infty(t)
 =
 W(g(t-\cdot))
 =
 \int_{\R}g(t-u)\,W(\mathrm du),
 \qquad t\in\R.
 \label{eq:common-white-noise-realization}
\end{equation}
\end{lemma}

\begin{proof}
For \(t\in\R\), we define
\begin{equation}
 a_t=
 \left(
  e^{-t^{2}/2}\frac{t^k}{\sqrt{k!}}
  \right)_{k\geq0}
 \in\ell^{2}(\mathbb N_0).
 \label{eq:coefficient-feature-vector}
\end{equation}
For the vector in \eqref{eq:coefficient-feature-vector}, we sum the
exponential series and obtain
\begin{equation}
 \langle a_t,a_s\rangle_{\ell^{2}}
 =
 e^{-(t^{2}+s^{2})/2}
 \sum_{k=0}^{\infty}\frac{(ts)^k}{k!}
 =
 e^{-(t-s)^{2}/2}.
 \label{eq:coefficient-feature-inner-product}
\end{equation}
Moreover,
 \begin{equation}
  \|g\|_{L^{2}(\R)}^{2}
  =
  \left(\frac2\pi\right)^{1/2}
  \int_{\R}e^{-2u^{2}}\dd u
  =1
  \label{eq:moving-kernel-normalization}
 \end{equation}
 and
 \begin{equation}
  \langle g(t-\cdot),g(s-\cdot)\rangle_{L^{2}(\R)}
  =
  \left(\frac2\pi\right)^{1/2}
  \int_{\R}e^{-(t-u)^{2}-(s-u)^{2}}\dd u
  =
  e^{-(t-s)^{2}/2}.
  \label{eq:moving-kernel-inner-product}
 \end{equation}
We obtain the last equality from the identity
\((t-u)^{2}+(s-u)^{2}
=2(u-(t+s)/2)^{2}+(t-s)^{2}/2\).
We use
\eqref{eq:coefficient-feature-inner-product} and
\eqref{eq:moving-kernel-inner-product} to conclude that
for arbitrary \(r\geq1\), \(t_1,\ldots,t_r\in\R\) and
\(\lambda_1,\ldots,\lambda_r\in\R\),
\begin{equation}
 \left\|\sum_{i=1}^{r}\lambda_i a_{t_i}\right\|_{\ell^{2}}^{2}
 =
 \sum_{i,j=1}^{r}
 \lambda_i\lambda_j e^{-(t_i-t_j)^{2}/2}
 =
 \left\|
  \sum_{i=1}^{r}\lambda_i g(t_i-\cdot)
  \right\|_2^{2}.
 \label{eq:feature-span-isometry}
\end{equation}
In particular, we use \eqref{eq:feature-span-isometry} to verify that if
\(\sum_{i=1}^{r}\lambda_i a_{t_i}=0\), then the last norm is zero.
We therefore define
\begin{equation}
 T_0\!\left(\sum_{i=1}^{r}\lambda_i a_{t_i}\right)
 =
 \sum_{i=1}^{r}\lambda_i g(t_i-\cdot)
\label{eq:feature-isometry-definition}
\end{equation}
and use \eqref{eq:feature-span-isometry} to verify that \(T_0\) is well
defined and norm preserving on the linear span of the vectors \(a_t\).
We denote by \(H_a\) the closed span of these vectors in
\(\ell^{2}(\mathbb N_0)\) and by \(H_g\) the closed span of the
translates of \(g\) in \(L^{2}(\R)\).  We extend the isometry in
\eqref{eq:feature-isometry-definition} uniquely by continuity to an
isometry \(T\) from \(H_a\) into \(H_g\).
Its range contains the span of the translates of \(g\).  Since the
range of an isometry on the complete space \(H_a\) is closed, its
range equals \(H_g\).  Thus \(T\) maps \(H_a\) isometrically onto
\(H_g\), so \(T^{-1}\) acts on all of \(H_g\).

The coefficient sequence defines an isonormal process \(X\) on
\(\ell^{2}(\mathbb N_0)\) by
\begin{equation}
 X(h)=\sum_{k=0}^{\infty}h_k\xi_k,
 \qquad h=(h_k)_{k\geq0}\in\ell^{2}(\mathbb N_0),
\label{eq:coefficient-isonormal-process}
\end{equation}
where the series converges in \(L^{2}\).  We denote by \(\Pi_g\) the
orthogonal projection from \(L^{2}(\R)\) onto \(H_g\).  On an enlargement of the
probability space, we choose an independent isonormal process \(W_0\)
on \(H_g^\perp\) and we define
\begin{equation}
 W(h)=
 X\!\left(T^{-1}\Pi_g h\right)
 +W_0\!\left(h-\Pi_g h\right),
 \qquad h\in L^{2}(\R).
\label{eq:white-noise-extension-definition}
\end{equation}
For the process in \eqref{eq:white-noise-extension-definition}, we use
the orthogonality of \(H_g\) and \(H_g^\perp\), the
independence of \(X\) and \(W_0\) and the isometric property of \(T\)
to obtain
\[
 \E[W(h)W(k)]=\langle h,k\rangle_{L^{2}(\R)}.
\]
Every finite collection of the random variables \(W(h)\) forms a
Gaussian vector because \(X\) and \(W_0\) are independent isonormal
processes.  Thus \(W\) is real Gaussian white noise.  For every fixed
\(t\),
we have, as an equality in \(L^{2}\),
\begin{equation}
 W(g(t-\cdot))
 =
 X(a_t)
 =
 e^{-t^{2}/2}
 \sum_{k=0}^{\infty}\frac{\xi_k}{\sqrt{k!}}t^k
 =
 q_\infty(t).
 \label{eq:common-white-noise-L2-identity}
\end{equation}
Here we used \eqref{eq:coefficient-feature-vector},
\eqref{eq:coefficient-isonormal-process} and
\eqref{eq:white-noise-extension-definition}.

We establish simultaneous equality for all \(t\).  We use
Section~\ref{sec:preliminaries} to choose \(q_\infty\) with continuous
paths.  On the other hand,
\[
 \E\left|
  W(g(t-\cdot))-W(g(s-\cdot))
 \right|^{2}
 \leq
 |t-s|^{2}\|g'\|_{L^{2}(\R)}^{2}.
\]
Since the increment is centered Gaussian, we use the Gaussian
fourth-moment identity to obtain
\[
 \E\left|
  W(g(t-\cdot))-W(g(s-\cdot))
 \right|^4
 \leq
 3\|g'\|_2^4|t-s|^4.
\]
For each integer \(r\geq1\), we apply Kolmogorov's continuity theorem
and obtain a continuous version \(Y_r\) of
\(t\mapsto W(g(t-\cdot))\) on \([-r,r]\).  If \(1\leq r<s\), then both
versions agree almost surely
with the original process at every rational point and hence
\[
 \Pp\!\left(
  Y_r(t)=Y_s(t)
  \text{ for every }t\in\mathbb Q\cap[-r,r]
 \right)
 =1.
\]
We take the countable intersection over the integer pairs \(r<s\) and
use continuity to obtain
\[
 Y_s\big|_{[-r,r]}=Y_r,
 \qquad 1\leq r<s,
\]
on a probability-one event.  We define
\[
 Y(t)=Y_r(t)
 \qquad\text{whenever }r\geq1\text{ is an integer and }|t|\leq r.
\]
This compatibility makes \(Y\) well defined and continuous on \(\R\).
We denote this global version again by \(W(g(t-\cdot))\).

For every \(t\in\mathbb Q\), we use
\eqref{eq:common-white-noise-L2-identity} and the fact that the chosen
continuous version agrees with the original process at each fixed
\(t\) to obtain \(\Pp(q_\infty(t)=W(g(t-\cdot)))=1\).  Since
\(\mathbb Q\) is countable,
\[
 \Pp\!\left(
  q_\infty(t)=W(g(t-\cdot))
  \text{ for every }t\in\mathbb Q
 \right)
 =1.
\]
We intersect this event with the probability-one event from
Section~\ref{sec:preliminaries} on which \(q_\infty\) has continuous
sample paths.  On the resulting event, we use continuity to obtain
\[
 q_\infty(t)
 =
 \lim_{u\to t}q_\infty(u)
 =
 \lim_{u\to t}W(g(u-\cdot))
 =
 W(g(t-\cdot)),
\]
where \(u\in\mathbb Q\) and \(t\in\R\).
We have proved \eqref{eq:common-white-noise-realization}.
\end{proof}

Recall that the cutoff \(\chi\in C_c^\infty(\R)\) fixed in
Section~\ref{subsec:bulk-results} satisfies
\begin{equation}
 0\leq\chi\leq1,
 \qquad
 \chi=1\text{ on }[-1/4,1/4],
 \qquad
 \operatorname{supp}\chi\subseteq[-1/2,1/2].
 \label{eq:finite-range-cutoff}
\end{equation}
For \(m\geq1\), introduce the auxiliary notation
\begin{equation}
 \widetilde g_m(u)=g(u)\chi(u/m),
 \qquad
 \alpha_m=\|\widetilde g_m\|_{L^{2}(\R)}^{-1},
 \qquad\text{so that}\qquad
 g_m=\alpha_m\widetilde g_m.
 \label{eq:finite-range-kernel}
\end{equation}
We use the properties of \(\chi\) and the strict positivity of \(g\)
to obtain \(\|\widetilde g_m\|_2>0\), so \(\alpha_m\) is well defined.
Thus \(g_m\) is the normalized cutoff kernel fixed in
Section~\ref{subsec:bulk-results}.  With the common noise \(W\) now
constructed, the process \(q_{\infty,m}\) introduced there is realized
by
\begin{equation}
 q_{\infty,m}(t)
 =
 W(g_m(t-\cdot))
 =
 \int_{\R}g_m(t-u)\,W(\mathrm du),
 \qquad t\in\R.
 \label{eq:m-dependent-process}
\end{equation}
We write
\begin{equation}
 Y_m=q_{\infty,m}-q_\infty
\label{eq:finite-range-error-process}
\end{equation}
for the approximation error.
We use the definition of \(\alpha_m\) to obtain
\begin{equation}
 \Var(q_{\infty,m}(t))=\|g_m\|_{L^{2}(\R)}^{2}=1.
\label{eq:m-dependent-unit-variance}
\end{equation}

Having realized \(q_{\infty,m}\), we record the smoothness and joint
stationarity needed to compare \(q_{\infty,m}\) with \(q_\infty\),
together with the exact \(m\)-dependence needed later to apply
Lemma~\ref{lem:appendix-chen-shao-thm2-6} to the decomposition of
\(N_{\infty,m}(I)-\E N_{\infty,m}(I)\) over subintervals of \(I\) of
length at most one.

\begin{lemma}
\label{lem:m-dependent-process-properties}
For every \(m\geq1\), we select and henceforth use a \(C^\infty\)
modification of the process realized in
\eqref{eq:m-dependent-process}.  This selected version is centered,
stationary and has unit variance.
Moreover, the pair \((q_\infty,q_{\infty,m})\) is jointly stationary.
Furthermore, if Borel sets \(E,F\subset\R\) satisfy
\(\operatorname{dist}(E,F)>m\), then
\[
 \sigma\{q_{\infty,m}(t):t\in E\}
 \quad\text{and}\quad
 \sigma\{q_{\infty,m}(t):t\in F\}
 \quad\text{are independent}.
\]
\end{lemma}

\begin{proof}
We fix \(m\geq1\).  For every integer \(j\geq0\), the map
\(t\mapsto g_m^{(j)}(t-\cdot)\) is continuously differentiable as an
\(L^{2}(\R)\)-valued map, with derivative
\(g_m^{(j+1)}(t-\cdot)\).  We use the fundamental theorem of calculus
in \(L^{2}(\R)\) and the linear isometry of \(W\) to obtain
\[
 W(g_m^{(j)}(t-\cdot))-W(g_m^{(j)}(s-\cdot))
 =
 \int_s^t W(g_m^{(j+1)}(u-\cdot))\dd u
\]
in \(L^{2}(\Omega)\).  We use the Gaussian fourth-moment identity to
obtain
\[
 \E\bigl|W(g_m^{(j)}(t-\cdot))-W(g_m^{(j)}(s-\cdot))\bigr|^4
 \leq 3\|g_m^{(j+1)}\|_2^4|t-s|^4.
\]
For every \(j\geq0\), we apply Kolmogorov's continuity theorem to
construct compatible continuous versions on compact intervals and
combine them into globally continuous versions of
\(t\mapsto W(g_m^{(j)}(t-\cdot))\), chosen simultaneously.  We use
Fubini's theorem, a countable intersection over rational \(s,t,j\)
and continuity to extend the preceding identity pathwise to all
\(s,t\in\R\).  We then apply the ordinary fundamental theorem of
calculus iteratively and obtain a \(C^\infty\) version of
\(q_{\infty,m}\).

For \(s,t\in\R\), the covariance
\[
 \E[q_{\infty,m}(t)q_{\infty,m}(s)]
 =
 \int_{\R}g_m(t-u)g_m(s-u)\dd u
\]
depends only on \(t-s\).  Since \(q_{\infty,m}\) is centered
Gaussian, we conclude from this covariance identity that the process
is stationary.  From Equation~\eqref{eq:m-dependent-unit-variance},
we obtain unit variance.
We also use the common white-noise realization to compute
\[
 \E[q_\infty(t)q_{\infty,m}(s)]
 =
 \int_{\R}g(t-u)g_m(s-u)\dd u,
\]
which depends only on \(t-s\).  The pair
\((q_\infty,q_{\infty,m})\) is jointly centered Gaussian.  Its two
marginal covariance functions and the preceding cross-covariance
are invariant under a common translation of all time arguments.
We therefore conclude that every finite-dimensional distribution of
the pair is invariant under such translations and that the pair is
jointly stationary.

If either \(E\) or \(F\) is empty, one of these two
\(\sigma\)-fields is trivial, so they are
independent.  We therefore assume that both sets are nonempty.  We use
\eqref{eq:finite-range-cutoff} and \eqref{eq:finite-range-kernel} to
obtain
\[
 \operatorname{supp}g_m\subseteq[-m/2,m/2].
\]
For \(t\in E\) and \(s\in F\),
\[
 |t-s|
 \geq
 \operatorname{dist}(E,F)
 >m.
\]
The supports of the functions \(g_m(t-u)\) and \(g_m(s-u)\) lie in
\([t-m/2,t+m/2]\) and \([s-m/2,s+m/2]\), respectively.  These
intervals are disjoint and hence
\[
 \E[q_{\infty,m}(t)q_{\infty,m}(s)]
 =
 \langle g_m(t-\cdot),g_m(s-\cdot)\rangle_{L^{2}(\R)}
 =0.
\]
We fix \(t_1,\ldots,t_r\in E\) and \(s_1,\ldots,s_v\in F\).  The two
random vectors
\[
 \bigl(q_{\infty,m}(t_1),\ldots,q_{\infty,m}(t_r)\bigr)
 \quad\text{and}\quad
 \bigl(q_{\infty,m}(s_1),\ldots,q_{\infty,m}(s_v)\bigr)
\]
are jointly Gaussian.  We compute every covariance between a
coordinate of the first vector and a coordinate of the second vector
and obtain zero.  Since the vectors are jointly Gaussian, we conclude that
they are independent.  The corresponding cylinder events form
generating \(\pi\)-systems for these two \(\sigma\)-fields.  We apply
the monotone-class theorem first
to one generating \(\pi\)-system and then to the other, thereby
extending the finite-dimensional independence to the two generated
\(\sigma\)-fields.
\end{proof}

We repeat the construction of compatible continuous versions on
compact intervals from the preceding proof with \(g\) in place of
\(g_m\).  For \(j=0\), we choose the continuous version from
Lemma~\ref{lem:common-white-noise-realization}.  Since
\(g^{(j)}\in L^{2}(\R)\) for every \(j\geq0\), we use the
\(L^{2}(\R)\)-valued fundamental theorem of calculus, Fubini's theorem
and a countable intersection over rational \(s,t\) to obtain, on one
probability-one event,
\[
 W(g^{(j)}(t-\cdot))-W(g^{(j)}(s-\cdot))
 =
 \int_s^tW(g^{(j+1)}(u-\cdot))\dd u
\]
for every \(j\geq0\) and \(s,t\in\R\).  We then apply the ordinary
fundamental theorem of calculus and choose versions for which
\begin{equation}
 q_\infty^{(j)}(t)=W(g^{(j)}(t-\cdot))
 \qquad
 \text{for every }j\geq0\text{ and }t\in\R
 \label{eq:stationary-white-noise-derivatives}
\end{equation}
simultaneously almost surely.
We fix these versions throughout the remainder of the paper.

We define \(\Omega_{\mathrm{reg},\infty}\) as the intersection of
\(\Omega_{\mathrm W}\) with the probability-one events on which the
common white-noise realization and
\eqref{eq:stationary-white-noise-derivatives} hold and with the
probability-one event on which there are no multiple zeros, as supplied by
Section~\ref{sec:preliminaries}.  On
\(\Omega_{\mathrm{reg},\infty}\), the chosen version of \(q_\infty\)
agrees with \(e^{-t^{2}/2}P_\infty(t)\), is smooth and has only
finitely many simple zeros on every compact interval.

The next lemma shows that the zeros of \(q_{\infty,m}\) are almost surely simple
and locally finite and that the zero count on every bounded interval is
measurable with respect to the corresponding process restriction.  The
measurability assertion is used below to prove independence of zero counts on
sufficiently separated intervals.

\begin{lemma}
\label{lem:finite-range-zero-count-regularity}
For every \(m\geq1\), almost surely all real zeros of \(q_{\infty,m}\)
are simple and the number of real zeros of \(q_{\infty,m}\) in every
compact interval is finite.
Moreover, for every bounded interval \(K\subset\R\), its zero count
\(N_{\infty,m}(K)\) is a finite-valued random variable measurable with
respect to the \(\Pp\)-augmentation of
\(\sigma\{q_{\infty,m}(t):t\in K\}\).
\end{lemma}

\begin{proof}
We fix \(m\geq1\).  We use
Lemma~\ref{lem:m-dependent-process-properties} to obtain
\[
 \sup_{t,x\in\R}p_{q_{\infty,m}(t)}(x)=(2\pi)^{-1/2}.
\]
For every \(r\in\mathbb N\), we apply
Lemma~\ref{lem:appendix-azais-wschebor-prop1-20} and obtain
\[
 \Pp\bigl(\exists t\in[-r,r]:
 q_{\infty,m}(t)=q_{\infty,m}'(t)=0\bigr)=0.
\]
We take the countable intersection over \(r\) and the \(C^\infty\)
sample-path event to obtain an event \(\Omega_{\mathrm{reg},m}\) such
that
\[
 \Pp(\Omega_{\mathrm{reg},m})=1,\qquad
 q_{\infty,m}\in C^\infty(\R),\qquad
 q_{\infty,m}(t)=0\Longrightarrow q_{\infty,m}'(t)\neq0
 \quad(t\in\R).
\]
If a compact interval contained infinitely many zeros, we could use
compactness to choose distinct zeros \(t_k\to t\).  On
\(\Omega_{\mathrm{reg},m}\),
\[
 q_{\infty,m}(t)=\lim_{k\to\infty}q_{\infty,m}(t_k)=0,\qquad
 q_{\infty,m}'(t)
 =
 \lim_{k\to\infty}
 \frac{q_{\infty,m}(t_k)-q_{\infty,m}(t)}{t_k-t}
 =0,
\]
which is impossible.  Hence
\[
 \#\{t\in K:q_{\infty,m}(t)=0\}<\infty
 \qquad\text{for every compact interval }K\subset\R.
\]
For every bounded interval \(K\subset\R\), we select an
everywhere-defined version of the zero count \(N_{\infty,m}(K)\)
introduced in Section~\ref{subsec:bulk-results} by setting
\[
 N_{\infty,m}(K)(\omega)
 =
 \begin{cases}
  \#\{t\in K:q_{\infty,m}(t,\omega)=0\},
  & \omega\in\Omega_{\mathrm{reg},m},\\
 0,
 & \omega\notin\Omega_{\mathrm{reg},m}.
 \end{cases}
\]
We fix a compact interval \(K\) with nonempty interior and an integer
\(k\geq1\).  We use continuity on \(\Omega_{\mathrm{reg},m}\) to
obtain
\begin{equation}
 \{N_{\infty,m}(K)\geq k\}
 =
 \Omega_{\mathrm{reg},m}\cap
 \bigcup_{r=1}^{\infty}
 \bigcap_{M=1}^{\infty}
 \bigcup_{\substack{
  t_1,\ldots,t_k\in K\cap\mathbb Q\\
  |t_i-t_j|\geq1/r,\ i\neq j
 }}
 \bigcap_{i=1}^{k}
 \{|q_{\infty,m}(t_i)|<1/M\}.
 \label{eq:zero-count-measurability}
\end{equation}
For one inclusion in \eqref{eq:zero-count-measurability}, we fix
distinct zeros \(z_1,\ldots,z_k\).  We use the density of
\(K\cap\mathbb Q\) in \(K\) and continuity to choose a fixed \(r\) and
rational points \(t_i^{(M)}\) such that
\[
 t_i^{(M)}\longrightarrow z_i,\qquad
 |t_i^{(M)}-t_j^{(M)}|\geq1/r,\qquad
 |q_{\infty,m}(t_i^{(M)})|<1/M.
\]
Conversely, we fix \(r\) and choose the corresponding points for every
\(M\) from the right-hand side of
\eqref{eq:zero-count-measurability}.  We use compactness of \(K^k\) to
choose a subsequence \(M_j\to\infty\) for which
\[
 (t_1^{(M_j)},\ldots,t_k^{(M_j)})
 \longrightarrow(s_1,\ldots,s_k)\in K^k,
\]
where \(|s_i-s_j|\geq1/r\) and
\(q_{\infty,m}(s_i)
 =\lim_{j\to\infty}q_{\infty,m}(t_i^{(M_j)})=0\).
Thus \(s_1,\ldots,s_k\) are distinct zeros.  All operations in
\eqref{eq:zero-count-measurability} are countable and
\(\Omega_{\mathrm{reg},m}^{\mathsf c}\) is a \(\Pp\)-null event.
Hence \(N_{\infty,m}(K)\) is measurable with respect to the
\(\Pp\)-augmentation of
\(\sigma\{q_{\infty,m}(t):t\in K\}\).

For \(a<b\), we use monotone approximation by compact intervals to
obtain
\[
 \begin{aligned}
 N_{\infty,m}((a,b))
 &=\lim_{M\to\infty}
 N_{\infty,m}([a+M^{-1},b-M^{-1}]),\\
 N_{\infty,m}([a,b))
 &=\lim_{M\to\infty}
 N_{\infty,m}([a,b-M^{-1}]),\\
 N_{\infty,m}((a,b])
 &=\lim_{M\to\infty}
 N_{\infty,m}([a+M^{-1},b]),\\
 N_{\infty,m}(\{a\})
 &=\mathbf 1_{\Omega_{\mathrm{reg},m}}
 \mathbf 1_{\{q_{\infty,m}(a)=0\}},
 \end{aligned}
\]
where \(M\) is sufficiently large.  This proves the local
measurability for every bounded interval.
\end{proof}

For each \(m\geq1\), we henceforth fix the event
\(\Omega_{\mathrm{reg},m}\) constructed in the preceding proof.  It has
probability one.  We also use the everywhere-defined zero-count versions
fixed there.

 Recall from Section~\ref{subsec:useful-tools} that the \(q_\infty\) zero
 counts are measurable and from the Introduction that the corresponding
 random-polynomial zero counts are measurable.  Combining these facts with
 the preceding lemma and the polynomial degree bound
 \eqref{eq:polynomial-zero-count-degree-bound}, we conclude that every zero
 count used below is a finite-valued measurable random variable.

\subsubsection{Quantitative \texorpdfstring{$C^1$}{C1} approximation}

We begin the quantitative coupling by showing that cutoff and renormalization
perturb each fixed derivative of the Gaussian kernel by an exponentially
small \(L^2\) error.

\begin{lemma}
\label{lem:finite-range-kernel-approximation}
For every integer \(j\geq0\), there is a constant \(C_{j,\chi}>0\)
such that
\begin{equation}
 \|g_m^{(j)}-g^{(j)}\|_{L^{2}(\R)}
 \leq C_{j,\chi}e^{-m^{2}/64},
 \qquad m\geq1.
 \label{eq:finite-range-kernel-approximation}
\end{equation}
\end{lemma}

\begin{proof}
For every \(j\geq0\), the function \(g^{(j)}\) is a polynomial of
degree \(j\) times \(e^{-u^{2}}\).  In particular,
\[
 \sup_{u\in\R}
 \frac{e^{u^{2}}|g^{(j)}(u)|}{(1+|u|)^j}
 <\infty.
\]
We apply Leibniz's formula and obtain
\begin{equation}
 \widetilde g_m^{(j)}(u)
 =
 \sum_{r=0}^{j}
 \binom jr m^{-r}
 g^{(j-r)}(u)\chi^{(r)}(u/m).
 \label{eq:cutoff-kernel-Leibniz-expansion}
\end{equation}
The difference between the term with \(r=0\) and \(g^{(j)}\) is
supported on \(\{|u|\geq m/4\}\).  Every term with \(r\geq1\) is
supported on
\[
 \{m/4\leq|u|\leq m/2\}.
\]
For an integer \(d\geq0\) and \(x\geq0\), we differentiate
\(2d\log(1+x)-x^{2}\) and obtain
\[
 \frac{\mathrm d}{\mathrm dx}
 \left(2d\log(1+x)-x^{2}\right)
 =
 \frac{2d}{1+x}-2x.
\]
For \(d\geq1\), the unique critical point is
\[
 x=\frac{\sqrt{1+4d}-1}{2}
\]
and the second derivative is strictly negative.  For \(d=0\), the
function attains its maximum at \(x=0\).  Consequently, for every integer
\(d\geq0\),
\[
 \sup_{u\in\R}(1+|u|)^{2d}e^{-u^{2}}
 =
 \left(\frac{1+\sqrt{1+4d}}{2}\right)^{2d}
 \exp\!\left[
 -\left(\frac{\sqrt{1+4d}-1}{2}\right)^{2}
 \right].
\]
Moreover,
\[
 \int_{|u|\geq m/4}e^{-u^{2}}\dd u
 \leq
 e^{-m^{2}/32}\int_{\R}e^{-u^{2}/2}\dd u
 =
 \sqrt{2\pi}\,e^{-m^{2}/32}.
\]
We combine the preceding two estimates and obtain
\[
 \begin{aligned}
  &\int_{|u|\geq m/4}(1+|u|)^{2d}e^{-2u^{2}}\dd u\\
  {}\leq{}&
  \left(\sup_{u\in\R}(1+|u|)^{2d}e^{-u^{2}}\right)
  \int_{|u|\geq m/4}e^{-u^{2}}\dd u\\
  {}\leq{}&
  \sqrt{2\pi}
  \left(\frac{1+\sqrt{1+4d}}{2}\right)^{2d}
  \exp\!\left[
   -\left(\frac{\sqrt{1+4d}-1}{2}\right)^{2}
  \right]
  e^{-m^{2}/32}.
 \end{aligned}
\]
Thus we take
\[
 C_d
 =
 \sqrt{2\pi}
 \left(\frac{1+\sqrt{1+4d}}{2}\right)^{2d}
 \exp\!\left[
  -\left(\frac{\sqrt{1+4d}-1}{2}\right)^{2}
 \right],
\]
With this choice,
\begin{equation}
 \int_{|u|\geq m/4}(1+|u|)^{2d}e^{-2u^{2}}\dd u
 \leq C_de^{-m^{2}/32}.
 \label{eq:explicit-polynomial-gaussian-tail}
\end{equation}
We apply \eqref{eq:explicit-polynomial-gaussian-tail} to each term in
the Leibniz expansion \eqref{eq:cutoff-kernel-Leibniz-expansion} and
use \(m^{-r}\leq1\) and the boundedness of
\(\chi^{(r)}\).  We then apply the triangle inequality in \(L^{2}(\R)\)
and
\[
 \left(\sum_{r=0}^{j}a_r\right)^{2}
 \leq
 (j+1)\sum_{r=0}^{j}a_r^{2},
 \qquad a_0,\ldots,a_j\geq0,
\]
to obtain a constant \(C_{j,\chi}>0\) such that
\[
 \|\widetilde g_m^{(j)}-g^{(j)}\|_2^{2}
 \leq
 C_{j,\chi}^{2}e^{-m^{2}/32}.
\]
We take square roots and obtain
\begin{equation}
 \|\widetilde g_m^{(j)}-g^{(j)}\|_2
 \leq
 C_{j,\chi}e^{-m^{2}/64}.
 \label{eq:unnormalized-kernel-approximation}
\end{equation}

In particular, we take \(j=0\) in
\eqref{eq:unnormalized-kernel-approximation}, apply the reverse
triangle inequality and obtain
\[
 \left|\|\widetilde g_m\|_2-1\right|
 \leq\|\widetilde g_m-g\|_2
 \leq C_{0,\chi}e^{-m^{2}/64}.
\]
Since \(\chi(u/m)=1\) for \(|u|\leq1/4\), we also have
\[
 \inf_{m\geq1}\|\widetilde g_m\|_2^{2}
 \geq
 \int_{-1/4}^{1/4}(g(u))^{2}\dd u
 >0.
\]
Consequently, we obtain
\begin{equation}
 |\alpha_m-1|
 =
 \frac{|1-\|\widetilde g_m\|_2|}
 {\|\widetilde g_m\|_2}
 \leq
 \frac{C_{0,\chi}}
 {\left(\displaystyle
   \int_{-1/4}^{1/4}(g(u))^{2}\dd u
  \right)^{1/2}}
 e^{-m^{2}/64}.
 \label{eq:cutoff-normalizing-factor-bound}
\end{equation}
We use \eqref{eq:cutoff-kernel-Leibniz-expansion} and obtain
\begin{equation}
 \begin{aligned}
 \sup_{m\geq1}\|\widetilde g_m^{(j)}\|_2
 &\leq
 \sum_{r=0}^{j}\binom jr
 \sup_{m\geq1}
 m^{-r}
 \bigl\|g^{(j-r)}\chi^{(r)}(\cdot/m)\bigr\|_2\\
 &\leq
 \sum_{r=0}^{j}\binom jr
 \|\chi^{(r)}\|_\infty\|g^{(j-r)}\|_2
 <\infty.
 \end{aligned}
 \label{eq:unnormalized-kernel-uniform-bound}
\end{equation}
We combine \eqref{eq:unnormalized-kernel-approximation}--%
\eqref{eq:unnormalized-kernel-uniform-bound} and obtain
\[
 \begin{aligned}
 \|g_m^{(j)}-g^{(j)}\|_2
 &=
 \bigl\|
  (\alpha_m-1)\widetilde g_m^{(j)}
  +\bigl(\widetilde g_m^{(j)}-g^{(j)}\bigr)
 \bigr\|_2\\
 &\leq
 |\alpha_m-1|\|\widetilde g_m^{(j)}\|_2
 +\|\widetilde g_m^{(j)}-g^{(j)}\|_2\\
 &\leq
 C_{j,\chi}e^{-m^{2}/64}.
 \end{aligned}
\]
This proves \eqref{eq:finite-range-kernel-approximation}.
\end{proof}

We now combine the kernel estimate, the common white-noise realization and
the interval Sobolev bound to obtain the high-probability \(C^1\) coupling
used in zero matching.

\begin{lemma}
\label{lem:finite-range-C1-coupling}
There are constants \(C,c>0\) such that, for every
\(I=[\varphi,\psi]\subset\R\) with \(\ell\geq1\), every \(m\geq1\) and every
\(\varepsilon>0\),
\begin{equation}
 \Pp\!\left(
  \|q_{\infty,m}-q_\infty\|_{C^1(I)}\geq\varepsilon
 \right)
 \leq
 C(\ell+1)e^{-cm^{2}}\varepsilon^{-2}.
 \label{eq:finite-range-C1-coupling}
\end{equation}
\end{lemma}

\begin{proof}
The \(\lfloor\ell\rfloor+1\) unit intervals
\([\varphi+r,\varphi+r+1]\), \(0\leq r\leq\lfloor\ell\rfloor\), cover \(I\).
We combine \eqref{eq:m-dependent-process} and
\eqref{eq:stationary-white-noise-derivatives} to obtain, for the error
process in \eqref{eq:finite-range-error-process},
\begin{equation}
 Y_m^{(j)}(t)
 =
 W\!\left(
  (g_m^{(j)}-g^{(j)})(t-\cdot)
 \right),
 \qquad j=0,1,2.
 \label{eq:coupling-error-derivatives}
\end{equation}
For each \(j\in\{0,1,2\}\), we use
\eqref{eq:coupling-error-derivatives} to identify \(Y_m^{(j)}\) as a
moving average with a translated deterministic kernel.  Hence
\(Y_m^{(j)}\) is
stationary and
\begin{equation}
 \E|Y_m^{(j)}(t)|^{2}
 =
 \|g_m^{(j)}-g^{(j)}\|_2^{2},
 \qquad t\in\R.
\label{eq:coupling-error-derivative-variance}
\end{equation}
We apply Lemma~\ref{lem:comparison-interval-Sobolev} first to \(Y_m\)
and then to \(Y_m'\) on every interval, use
\((a+b)^{2}\leq2a^{2}+2b^{2}\) and obtain
\begin{equation}
 \begin{aligned}
  \|Y_m\|_{C^1(I)}^{2}
  &\leq
  8\sum_{r=0}^{\lfloor\ell\rfloor}
  \int_{\varphi+r}^{\varphi+r+1}
  \left(
   (Y_m(t))^{2}+2(Y_m'(t))^{2}+(Y_m''(t))^{2}
  \right)\dd t\\
  &\leq
  16\sum_{r=0}^{\lfloor\ell\rfloor}
  \int_{\varphi+r}^{\varphi+r+1}
  \left(
   (Y_m(t))^{2}+(Y_m'(t))^{2}+(Y_m''(t))^{2}
  \right)\dd t.
 \end{aligned}
\label{eq:finite-range-C1-interval-sum}
\end{equation}
We take expectations in \eqref{eq:finite-range-C1-interval-sum} and use
stationarity, \eqref{eq:coupling-error-derivative-variance} and the
white-noise isometry.  We then apply
\eqref{eq:finite-range-kernel-approximation} for \(j=0,1,2\) and obtain
\begin{equation}
 \begin{aligned}
  \E\|Y_m\|_{C^1(I)}^{2}
  &\leq
  C(\ell+1)
  \sum_{j=0}^{2}\E|Y_m^{(j)}(0)|^{2}\\
  &=
  C(\ell+1)
   \sum_{j=0}^{2}\|g_m^{(j)}-g^{(j)}\|_2^{2}\\
  &\leq C(\ell+1)e^{-cm^{2}}.
 \end{aligned}
\label{eq:finite-range-C1-L2}
\end{equation}
In the last inequality of \eqref{eq:finite-range-C1-L2}, we squared
the estimates from
\eqref{eq:finite-range-kernel-approximation} and then changed the
positive constants \(C\) and \(c\).
We use \eqref{eq:finite-range-C1-L2} and Markov's inequality to obtain
\[
 \Pp\!\left(
  \|q_{\infty,m}-q_\infty\|_{C^1(I)}\geq\varepsilon
 \right)
 \leq \varepsilon^{-2}\E\|Y_m\|_{C^1(I)}^{2}
 \leq C(\ell+1)e^{-cm^{2}}\varepsilon^{-2}.
\]
This proves \eqref{eq:finite-range-C1-coupling}.  We used
\(\geq\varepsilon\) in the definition of the exceptional event.  Hence, on
its complement, the \(C^1\) distance is strictly smaller than
\(\varepsilon\), which verifies
\eqref{eq:comparison-zero-stability-C1} when
Lemma~\ref{lem:comparison-order-preserving-stability} is applied below.
\end{proof}

\subsubsection{High-probability matching of the two zero sets}

Choosing \(m\) proportional to \(\sqrt{\log\ell}\), we combine
Lemmas~\ref{lem:finite-range-C1-coupling},
\ref{lem:comparison-near-critical-small-ball}
and~\ref{lem:comparison-order-preserving-stability} to prove
\(N_{\infty,m}(I)=N_\infty(I)\) and to pair, in increasing order, the zeros of
\(q_{\infty,m}\) and \(q_\infty\), except on an event whose probability is
bounded by a negative power of \(\ell\).

\begin{proposition}
\label{prop:finite-range-zero-matching}
For every \(A,R>0\), there exist constants \(b_{A,R},C_{A,R}>0\) and
\(\ell_{A,R}\geq e\) such that the following assertion holds.  For every
\(I=[\varphi,\psi]\subset\R\) with \(\ell\geq\ell_{A,R}\), we set
\begin{equation}
 m_{\ell,A,R}=b_{A,R}\sqrt{\log\ell}.
 \label{eq:finite-range-choice}
\end{equation}
With probability at least \(1-C_{A,R}\ell^{-A}\), the two counts
satisfy
\begin{equation}
 N_{\infty,m_{\ell,A,R}}(I)=N_\infty(I).
 \label{eq:finite-range-count-equality}
\end{equation}
On the same event, if the common count \(k\) is positive and
\[
 x_1<\cdots<x_k
 \quad\text{and}\quad
 y_1<\cdots<y_k
\]
are the zeros of \(q_\infty\) and \(q_{\infty,m_{\ell,A,R}}\), respectively,
in \(I\), then
\begin{equation}
 \max_{1\leq j\leq k}|x_j-y_j|\leq\ell^{-R}.
 \label{eq:finite-range-root-matching}
\end{equation}
\end{proposition}

\begin{proof}
We set
\begin{equation}
 \beta=\max\{R,A+3\}+1,
 \qquad
 \delta_\ell=\ell^{-\beta},
 \qquad
 \varepsilon_\ell=\ell^{-2\beta}
 \label{eq:finite-range-stability-parameters}
\end{equation}
and we choose the even integer
\begin{equation}
 p=2\left\lceil\frac{A+3}{2}\right\rceil>A+2.
\label{eq:finite-range-small-ball-moment-choice}
\end{equation}
We denote by \(c>0\) the constant in
\eqref{eq:finite-range-C1-coupling} and we choose \(b_{A,R}\geq1\) so
large that
\begin{equation}
 cb_{A,R}^{2}>A+4\beta+2.
 \label{eq:finite-range-b-choice}
\end{equation}
For every \(\ell\geq e\), we use the choices in
\eqref{eq:finite-range-choice},
\eqref{eq:finite-range-stability-parameters} and
\eqref{eq:finite-range-small-ball-moment-choice} to obtain
\[
 m_{\ell,A,R}\geq1,\qquad
 0<\varepsilon_\ell<\delta_\ell<1.
\]

We apply Lemma~\ref{lem:finite-range-C1-coupling} on \(I\), with
\(m=m_{\ell,A,R}\) from \eqref{eq:finite-range-choice} and
\(\varepsilon=\varepsilon_\ell\) from
\eqref{eq:finite-range-stability-parameters} to obtain
\begin{equation}
 \begin{aligned}
  &\Pp\!\left(
   \|q_{\infty,m_{\ell,A,R}}-q_\infty\|_{C^1(I)}
   \geq\varepsilon_\ell
  \right)\\
  {}\leq{}&
  C(\ell+1)
  e^{-cb_{A,R}^{2}\log\ell}
  \ell^{4\beta}\\
  {}\leq{}&
  C\ell^{1-cb_{A,R}^{2}+4\beta}
  \leq C_{A,R}\ell^{-A-1}
 \end{aligned}
 \label{eq:finite-range-C1-bad-event-bound}
\end{equation}
In the last line, we used \(\ell+1\leq2\ell\) and
\eqref{eq:finite-range-b-choice}, which gives
\[
 1-cb_{A,R}^{2}+4\beta<-A-1,
\]
as required.

We next apply
Lemma~\ref{lem:comparison-near-critical-small-ball} with
\(I=[\varphi,\psi]\), \(\delta=\delta_\ell\), \(M=\ell\) and \(p\) as in
\eqref{eq:finite-range-small-ball-moment-choice}.  We obtain
\begin{equation}
 \begin{aligned}
  &\Pp\!\left(
   \inf_{t\in I}
   \max\{|q_\infty(t)|,|q_\infty'(t)|\}
   \leq2\delta_\ell
  \right)\\
  {}\leq{}&
  C_p(\ell+1)\ell^{-p}
  +C_0(\ell+1)\ell\,\ell^{-\beta}
  +C_0\ell^{-2\beta}\\
  {}\leq{}&
  C_{A,R}
  \left(
   \ell^{1-p}+\ell^{2-\beta}+\ell^{-2\beta}
  \right)
  \leq C_{A,R}\ell^{-A-1}.
 \end{aligned}
 \label{eq:finite-range-transversality-bad-event-bound}
\end{equation}
To verify the last inequality, we use
\(p>A+2\) and \(\beta>A+3\).  We also apply
\eqref{eq:comparison-endpoint-small-ball} with
\(\varepsilon=\varepsilon_\ell\) and obtain
\begin{equation}
 \Pp\!\left(
  \min\{|q_\infty(\varphi)|,|q_\infty(\psi)|\}
  \leq2\varepsilon_\ell
 \right)
 \leq
 C_0\ell^{-2\beta}
 \leq C_{A,R}\ell^{-A-1}.
 \label{eq:finite-range-endpoint-bad-event-bound}
\end{equation}

We use \eqref{eq:finite-range-C1-bad-event-bound}--%
\eqref{eq:finite-range-endpoint-bad-event-bound} to bound the
probabilities of the three bad events.  We define \(\mathcal G_I\) as
the intersection of
\(\Omega_{\mathrm{reg},\infty}\),
\(\Omega_{\mathrm{reg},m_{\ell,A,R}}\) and the complements of these
three bad events.
On \(\mathcal G_I\), the functions \(q_\infty\) and
\(q_{\infty,m_{\ell,A,R}}\) satisfy every hypothesis of
Lemma~\ref{lem:comparison-order-preserving-stability} on
\([\varphi,\psi]=I\), with
\(\varepsilon=\varepsilon_\ell\) and \(\delta=\delta_\ell\) from
\eqref{eq:finite-range-stability-parameters}.  Since
\(\varepsilon_\ell<\delta_\ell\), we apply
Lemma~\ref{lem:comparison-order-preserving-stability} to obtain equality
of the two zero counts.  If the common count \(k\) is positive, we
apply Lemma~\ref{lem:comparison-order-preserving-stability} and obtain
an order-preserving matching with
\[
\max_{1\leq j\leq k}|x_j-y_j|
\leq
\frac{\varepsilon_\ell}{2\delta_\ell-\varepsilon_\ell}
=
\frac{\ell^{-\beta}}{2-\ell^{-\beta}}
\leq\ell^{-\beta}
\leq\ell^{-R}
\]
Both regularity events have probability one.  Hence, a union bound over the
three bad events gives
\[
 \begin{aligned}
 \Pp(\mathcal G_I^c)
 &\leq
 \Pp\!\left(
  \|q_{\infty,m_{\ell,A,R}}-q_\infty\|_{C^1(I)}
  \geq\varepsilon_\ell
 \right)\\
 &\quad+
 \Pp\!\left(
  \inf_{t\in I}
  \max\{|q_\infty(t)|,|q_\infty'(t)|\}
  \leq2\delta_\ell
 \right)\\
 &\quad+
 \Pp\!\left(
  \min\{|q_\infty(\varphi)|,|q_\infty(\psi)|\}
  \leq2\varepsilon_\ell
 \right)\\
 &\leq C_{A,R}\ell^{-A}.
 \end{aligned}
\]
We absorb the three constants into \(C_{A,R}\) and take
\(\ell_{A,R}=e\).  We thereby obtain
\eqref{eq:finite-range-count-equality} and
\eqref{eq:finite-range-root-matching}.
\end{proof}

We next specialize Proposition~\ref{prop:finite-range-zero-matching} to
\(A=16\) and \(R=1\).  The following corollary records the resulting
probability bound for unequal zero counts and the bound on the displacement
between the zeros listed in increasing order.

\begin{corollary}
\label{cor:finite-range-zero-matching-package}
There are constants \(b,C>0\) and \(\ell_0\geq e\) such that the
following holds.  For every \(I\subset\R\) with \(\ell\geq\ell_0\), we take
\begin{equation}
 \mu_\ell=b\sqrt{\log\ell}.
 \label{eq:finite-range-BE-range-choice}
\end{equation}
\begin{equation}
 \Pp\!\left(
  N_{\infty,\mu_\ell}(I)\neq N_\infty(I)
 \right)
 \leq
 C\ell^{-16}.
 \label{eq:finite-range-package-matching}
\end{equation}
Moreover, outside an event of probability at most \(C\ell^{-16}\), the zeros
of \(q_\infty\) and \(q_{\infty,\mu_\ell}\), listed in increasing order, are
paired within distance \(\ell^{-1}\).
\end{corollary}

\begin{proof}
We apply Proposition~\ref{prop:finite-range-zero-matching} with
\(A=16\) and \(R=1\), and we set
\(b=b_{16,1}\), \(C=C_{16,1}\) and \(\ell_0=\ell_{16,1}\).
The two assertions follow directly.
\end{proof}

\subsection{Proofs of Corollary~\ref{cor:finite-range-BE-package}
and Lemma~\ref{lem:finite-range-exact-coupling-transfer}}
\label{sec:finite-range-berry-esseen}

We establish a Berry--Esseen bound for \(N_{\infty,m}(I)\).  For a general pair
\((m,\ell)\), the bound retains the exact variance of
\(N_{\infty,m}(I)\).  We then take
\(m=\mu_\ell=b\sqrt{\log\ell}\) as in
 \eqref{eq:finite-range-BE-range-choice}.  For this choice, we use the
 comparison with \(q_\infty\) to obtain a linear lower bound for
\(\Var(N_{\infty,m}(I))\).  All estimates are uniform in the location of
\(I\).

\subsubsection{Local convergence, nondegeneracy and zero-count moments}

We begin by proving that \(q_{\infty,m}\) converges to \(q_\infty\) in mean
square with respect to the \(C^4([-1,2])\) norm and that the covariance matrix of
\((q_{\infty,m}(t),q_{\infty,m}'(t))\) has a uniformly positive smallest
eigenvalue for all sufficiently large \(m\).

\begin{lemma}
\label{lem:finite-range-C4-and-jet}
For every \(m\geq1\),
\begin{equation}
 \E\|q_{\infty,m}-q_\infty\|_{C^4([-1,2])}^{2}
 \leq
 60\sum_{j=0}^{5}
 \|g_m^{(j)}-g^{(j)}\|_{L^{2}(\R)}^{2},
 \label{eq:finite-range-C4-convergence}
\end{equation}
where the right-hand side tends to zero as \(m\to\infty\).
Moreover, there exists \(m_*\geq1\) such that, for every
\(m\geq m_*\) and \(t\in\R\), the smallest eigenvalue of the
covariance matrix of
\begin{equation}
 \bigl(
  q_{\infty,m}(t),
  q_{\infty,m}'(t),
  q_{\infty,m}''(t),
  q_{\infty,m}^{(3)}(t)
 \bigr)
 \label{eq:finite-range-order-three-jet}
\end{equation}
is at least
\begin{equation}
 \frac{8-\sqrt{58}}{2}>0.
 \label{eq:finite-range-jet-eigenvalue-lower-bound}
\end{equation}
\end{lemma}

\begin{proof}
We use the derivative constructions in
Lemma~\ref{lem:m-dependent-process-properties} and
\eqref{eq:stationary-white-noise-derivatives} to obtain, for
\(0\leq j\leq5\),
\begin{equation}
 Y_m^{(j)}(t)
 =
 W\!\left(
  (g_m^{(j)}-g^{(j)})(t-\cdot)
 \right),
 \qquad t\in\R.
 \label{eq:finite-range-C4-error-derivatives}
\end{equation}
For \(0\leq j\leq4\), we apply
Lemma~\ref{lem:comparison-interval-Sobolev} to \(Y_m^{(j)}\) on each
of the
three intervals
\[
 [-1,0],\qquad[0,1],\qquad[1,2].
\]
Because each of these intervals has length one, we apply the final
assertion of Lemma~\ref{lem:comparison-interval-Sobolev} with \(p=2\)
and obtain the constant \(2\) on every interval.
Since the supremum over their union is at most the sum of the three
individual squared suprema, we obtain
\begin{equation}
 \sup_{t\in[-1,2]}|Y_m^{(j)}(t)|^{2}
 \leq
 2\sum_{r=-1}^{1}
 \int_{r}^{r+1}
 \left(
  |Y_m^{(j)}(t)|^{2}
  +
  |Y_m^{(j+1)}(t)|^{2}
 \right)\dd t.
 \label{eq:finite-range-C4-Sobolev-one-derivative}
\end{equation}
For nonnegative numbers \(a_0,\ldots,a_4\), we apply the
Cauchy--Schwarz inequality to obtain
\[
 \left(\sum_{j=0}^{4}a_j\right)^{2}
 \leq5\sum_{j=0}^{4}a_j^{2}.
\]
We apply this inequality with
\(a_j=\sup_{t\in[-1,2]}|Y_m^{(j)}(t)|\), substitute
\eqref{eq:finite-range-C4-Sobolev-one-derivative}, use the norm
definition \eqref{eq:comparison-Cr-norm} and then take
expectations.  We use the white-noise isometry in
\eqref{eq:white-noise-isometry} together with
\eqref{eq:finite-range-C4-error-derivatives} to obtain
\[
 \E|Y_m^{(j)}(t)|^{2}
 =
 \|g_m^{(j)}-g^{(j)}\|_{L^{2}(\R)}^{2},
 \qquad t\in\R.
\]
Consequently, we obtain
\[
 \begin{aligned}
  &\E\|Y_m\|_{C^4([-1,2])}^{2}\\
  {}\leq{}&
  30\sum_{j=0}^{4}
  \left(
   \|g_m^{(j)}-g^{(j)}\|_{L^{2}(\R)}^{2}
   +
   \|g_m^{(j+1)}-g^{(j+1)}\|_{L^{2}(\R)}^{2}
  \right)\\
  {}\leq{}&
  60\sum_{j=0}^{5}
  \|g_m^{(j)}-g^{(j)}\|_{L^{2}(\R)}^{2}.
 \end{aligned}
\]
We apply Lemma~\ref{lem:finite-range-kernel-approximation} for
\(0\leq j\leq5\) and conclude that the last expression converges to
zero.  We have therefore proved
\eqref{eq:finite-range-C4-convergence}.

We next verify the uniform nondegeneracy in
\eqref{eq:finite-range-jet-eigenvalue-lower-bound}.  We use the
white-noise isometry to conclude that the covariance matrix of
\eqref{eq:finite-range-order-three-jet} is independent of \(t\) and
equals
\begin{equation}
 G_m=
 \left(
  \langle g_m^{(i)},g_m^{(j)}\rangle_{L^{2}(\R)}
 \right)_{0\leq i,j\leq3}.
 \label{eq:finite-range-jet-Gram-matrix}
\end{equation}
We replace \(g_m\) with \(g\) in
\eqref{eq:finite-range-jet-Gram-matrix} and obtain
\begin{equation}
 G_\infty
 =
 \left(
  \langle g^{(i)},g^{(j)}\rangle_{L^{2}(\R)}
 \right)_{0\leq i,j\leq3}.
 \label{eq:stationary-jet-Gram-definition}
\end{equation}
By a direct calculation using
\eqref{eq:explicit-kernel-derivative-norms},
\eqref{eq:moving-kernel-normalization} and
\eqref{eq:stationary-jet-Gram-definition}, we have
\begin{equation}
 G_\infty
 =
 \begin{pmatrix}
  1&0&-1&0\\
  0&1&0&-3\\
  -1&0&3&0\\
  0&-3&0&15
 \end{pmatrix}.
 \label{eq:stationary-order-three-jet-matrix}
\end{equation}
We order the coordinates as \(0,2,1,3\).  The matrix in
\eqref{eq:stationary-order-three-jet-matrix} then becomes the direct sum
\[
 \begin{pmatrix}1&-1\\-1&3\end{pmatrix}
 \oplus
 \begin{pmatrix}1&-3\\-3&15\end{pmatrix}.
\]
Here \(\oplus\) denotes the block-diagonal direct sum.  Hence the
spectrum of the reordered matrix is the union of the spectra of the
two blocks.
We use the quadratic formula to find that the smaller eigenvalues of
the two blocks are \(2-\sqrt2\) and \(8-\sqrt{58}\), respectively.
Since
\(\sqrt{58}>6+\sqrt2\), we conclude that
\begin{equation}
 \lambda_{\min}(G_\infty)
 =
 \min\{2-\sqrt2,8-\sqrt{58}\}
 =
 8-\sqrt{58}>0.
 \label{eq:stationary-jet-smallest-eigenvalue}
\end{equation}

For \(0\leq i,j\leq3\), we add and subtract
\(\langle g^{(i)},g_m^{(j)}\rangle\) to obtain
\[
 \begin{aligned}
  &\left|
   \langle g_m^{(i)},g_m^{(j)}\rangle
   -
   \langle g^{(i)},g^{(j)}\rangle
  \right|\\
  {}\leq{}&
  \|g_m^{(i)}-g^{(i)}\|_{L^{2}(\R)}
  \|g_m^{(j)}\|_{L^{2}(\R)}
  +
  \|g^{(i)}\|_{L^{2}(\R)}
  \|g_m^{(j)}-g^{(j)}\|_{L^{2}(\R)}.
 \end{aligned}
\]
We apply Lemma~\ref{lem:finite-range-kernel-approximation} to conclude
that both difference norms on the right-hand side converge to zero.
Moreover, we apply the triangle inequality
\[
 \|g_m^{(j)}\|_{L^{2}(\R)}
 \leq
 \|g^{(j)}\|_{L^{2}(\R)}
 +
 \|g_m^{(j)}-g^{(j)}\|_{L^{2}(\R)}
\]
to conclude that \(\|g_m^{(j)}\|_{L^{2}(\R)}\) remains bounded.
Therefore
every entry of \(G_m-G_\infty\) converges to zero.

For \(|x|=1\), we apply the Cauchy--Schwarz inequality to obtain
\(\sum_{j=1}^{4}|x_j|\leq2\) and hence
\[
 |Ax|
 \leq
 4\max_{1\leq r,s\leq4}|A_{rs}|,
 \qquad
 \|A\|_{\mathrm{op}}
 \leq
 4\max_{1\leq r,s\leq4}|A_{rs}|.
\]
Consequently, we obtain
\begin{equation}
 \lim_{m\to\infty}\|G_m-G_\infty\|_{\mathrm{op}}
 =
 0.
 \label{eq:finite-range-jet-matrix-convergence}
\end{equation}
We use the convergence in
\eqref{eq:finite-range-jet-matrix-convergence} to choose \(m_*\geq1\)
such that
\begin{equation}
 \|G_m-G_\infty\|_{\mathrm{op}}
 \leq
 \frac{8-\sqrt{58}}{2},
 \qquad m\geq m_*.
 \label{eq:finite-range-jet-threshold-choice}
\end{equation}
For \(|x|=1\), we apply the Cauchy--Schwarz inequality to obtain
\[
 \bigl|x^{\mathsf T}(G_m-G_\infty)x\bigr|
 \leq
 \|G_m-G_\infty\|_{\mathrm{op}}.
\]
Since \(G_m\) is symmetric, we apply the variational characterization
of its smallest eigenvalue together with
\eqref{eq:stationary-jet-smallest-eigenvalue} and
\eqref{eq:finite-range-jet-threshold-choice} to obtain
\[
\lambda_{\min}(G_m)
=
\min_{|x|=1}x^{\mathsf T}G_mx
\geq
\lambda_{\min}(G_\infty)
-
\|G_m-G_\infty\|_{\mathrm{op}}
\geq
\frac{8-\sqrt{58}}{2}.
\]
This proves \eqref{eq:finite-range-jet-eigenvalue-lower-bound}.
\end{proof}

For every \(I=[\varphi,\psi]\subset\R\) with \(\ell\geq1\), we set
\begin{equation}
 K_\ell=\lceil\ell\rceil
 \label{eq:finite-range-number-of-blocks}
\end{equation}
and define the disjoint intervals
\begin{equation}
 \begin{aligned}
  \mathcal J_{I,j}
  &=
  \begin{cases}
   [\varphi+j-1,\varphi+j),&1\leq j<K_\ell,\\
   [\varphi+K_\ell-1,\psi],&j=K_\ell,
  \end{cases}\\
  I
  &=
  \bigcup_{j=1}^{K_\ell}\mathcal J_{I,j}.
 \end{aligned}
 \label{eq:finite-range-block-partition}
\end{equation}
We partition \(I\) according to
\eqref{eq:finite-range-block-partition} into disjoint intervals of
length at most one.  Hence every
zero in \(I\) belongs to exactly one block, including a zero at a
block endpoint.

We next derive the uniform local moment estimates needed for the blockwise
Gaussian approximation.

\begin{proposition}
\label{prop:finite-range-uniform-local-moments}
There exist \(m_1\geq m_*\) and \(M_4<\infty\) such that
\begin{equation}
 \sup_{m\geq m_1}
 \E\!\left[N_{\infty,m}([0,1])^4\right]\leq M_4.
 \label{eq:finite-range-uniform-unit-fourth-moment}
\end{equation}
For every \(m\geq m_1\) and every interval \(J\subset\R\) of length
at most one,
\begin{equation}
 \E\left|
  N_{\infty,m}(J)-\E N_{\infty,m}(J)
 \right|^3
 \leq8M_4^{3/4}.
 \label{eq:finite-range-uniform-centered-third-moment}
\end{equation}
Moreover, for every \(I\subset\R\) with \(\ell\geq1\) and every
\(m\geq m_1\),
\begin{equation}
 \E\!\left[N_{\infty,m}(I)^4\right]\leq M_4K_\ell^4.
 \label{eq:finite-range-growing-fourth-moment}
\end{equation}
\end{proposition}

\begin{proof}
We apply Lemmas~\ref{lem:appendix-gass-stecconi-thm1-1}
and~\ref{lem:appendix-gass-stecconi-rem1-3} with \(p=4\), \(d=1\),
\(U=(-1,2)\) and \(K=[0,1]\).
We use Lemma~\ref{lem:m-dependent-process-properties} and
Section~\ref{sec:preliminaries} to obtain
\[
 q_{\infty,m},q_\infty\in C^\infty(U)\subset C^4(U)
 \qquad\text{almost surely}.
\]
We use \eqref{eq:finite-range-jet-eigenvalue-lower-bound} and
\eqref{eq:stationary-jet-smallest-eigenvalue} to obtain
\[
 \begin{aligned}
  \inf_{t\in U}\lambda_{\min}
  \Cov\!\left((q_{\infty,m}^{(j)}(t))_{j=0}^{3}\right)
  &\geq\frac{8-\sqrt{58}}2,\\
  \inf_{t\in U}\lambda_{\min}
  \Cov\!\left((q_\infty^{(j)}(t))_{j=0}^{3}\right)
  &=8-\sqrt{58}.
 \end{aligned}
\]
For every \(\varepsilon>0\), we use Markov's inequality and
\eqref{eq:finite-range-C4-convergence} to obtain
\[
 \Pp\!\left(\|q_{\infty,m}-q_\infty\|_{C^4([-1,2])}>\varepsilon\right)
 \leq\varepsilon^{-2}
 \E\|q_{\infty,m}-q_\infty\|_{C^4([-1,2])}^{2},
\]
where the right-hand side tends to zero as \(m\to\infty\).  We
conclude that \(q_{\infty,m}\) converges to \(q_\infty\) in
distribution in the local \(C^4\) topology on \(U\).  We use their
unit variances to obtain
\[
 \Pp(q_{\infty,m}(t)=0)=\Pp(q_\infty(t)=0)=0,
 \qquad t\in\{0,1\}.
\]
We use the preceding formulas to verify the regularity, jet
nondegeneracy and convergence hypotheses of the cited results and to
see that the endpoint convention is immaterial.  We use
Section~\ref{sec:preliminaries} and
Lemma~\ref{lem:finite-range-zero-count-regularity} to conclude that
the zeros are almost surely simple, so the cardinality counts in the
cited results agree with our multiplicity counts.  We therefore apply
the moment-continuity assertion in
Lemma~\ref{lem:appendix-gass-stecconi-rem1-3} and obtain
\[
 \lim_{m\to\infty}
 \E\!\left[N_{\infty,m}([0,1])^4\right]
 =
 \E\!\left[N_\infty([0,1])^4\right].
\]
We set
\[
 M_4=1+\E\!\left[N_\infty([0,1])^4\right]
\]
and choose \(m_1\geq m_*\) such that
\[
 \sup_{m\geq m_1}
 \E\!\left[N_{\infty,m}([0,1])^4\right]\leq M_4.
\]
This proves \eqref{eq:finite-range-uniform-unit-fourth-moment}.

We fix an interval \(J\subset\R\) of length at most one and choose
\(a\in\R\) such that \(J\subseteq[a,a+1]\).  For \(m\geq m_1\), we
use monotonicity and stationarity to obtain
\[
 \E\!\left[N_{\infty,m}(J)^4\right]
 \leq\E\!\left[N_{\infty,m}([a,a+1])^4\right]
 =\E\!\left[N_{\infty,m}([0,1])^4\right]\leq M_4.
\]
We apply \(|x-y|^3\leq4(|x|^3+|y|^3)\), Jensen's inequality and
H\"older's inequality to obtain
\[
 \begin{aligned}
  \E\left|
   N_{\infty,m}(J)-\E N_{\infty,m}(J)\right|^3
  &\leq
  4\left(
   \E N_{\infty,m}(J)^3
   +
   \bigl(\E N_{\infty,m}(J)\bigr)^3
  \right)\\
  &\leq
  8\E N_{\infty,m}(J)^3
  \leq
  8\left(\E N_{\infty,m}(J)^4\right)^{3/4}
  \leq8M_4^{3/4}.
 \end{aligned}
\]
This proves \eqref{eq:finite-range-uniform-centered-third-moment}.

For \(\ell\geq1\) and \(m\geq m_1\), we use
\eqref{eq:finite-range-block-partition}, the preceding fourth-moment
bound and Minkowski's inequality to obtain
\[
 \left(
  \E\!\left[N_{\infty,m}(I)^4\right]
 \right)^{1/4}
 =\left\|
  \sum_{j=1}^{K_\ell}N_{\infty,m}(\mathcal J_{I,j})
 \right\|_{L^4}
 \leq \sum_{j=1}^{K_\ell}
 \left\|N_{\infty,m}(\mathcal J_{I,j})\right\|_{L^4}
 \leq M_4^{1/4}K_\ell.
\]
We take fourth powers to prove
\eqref{eq:finite-range-growing-fourth-moment}.
\end{proof}

\subsubsection{Variance comparison for \texorpdfstring{\(\mu_\ell=b\sqrt{\log\ell}\)}{muell=b sqrt(log ell)}}

We take \(b\) and \(\mu_\ell\) as in
Corollary~\ref{cor:finite-range-zero-matching-package} and
\eqref{eq:finite-range-BE-range-choice}.

Combining Corollary~\ref{cor:finite-range-zero-matching-package} with the
preceding fourth-moment bounds, we transfer the mean and variance estimates
for \(N_\infty(I)\) to \(N_{\infty,\mu_\ell}(I)\).

\begin{lemma}
\label{lem:finite-range-BE-variance-transfer}
For every \(I\subset\R\) with \(\ell\) sufficiently large,
\begin{align}
 \left\|
  N_{\infty,\mu_\ell}(I)-N_\infty(I)
 \right\|_{L^{2}}
 &\leq C\ell^{-3},
 \label{eq:finite-range-BE-L2-transfer}\\
 \left|
  \E N_{\infty,\mu_\ell}(I)-\E N_\infty(I)
 \right|
 &\leq C\ell^{-3},
 \label{eq:finite-range-BE-mean-transfer}\\
 \left|
  \sqrt{\Var(N_{\infty,\mu_\ell}(I))}
  -
  \sqrt{\Var(N_\infty(I))}
 \right|
 &\leq C\ell^{-3}.
 \label{eq:finite-range-BE-standard-deviation-transfer}
\end{align}
Moreover,
\begin{equation}
 \frac{V_\infty}{8}\ell
 \leq
 \Var(N_{\infty,\mu_\ell}(I))
 \leq
 C\ell,
 \label{eq:finite-range-BE-linear-variance}
\end{equation}
where \(V_\infty\) is the constant in
\eqref{eq:stationary-linear-variance-all-L}.
\end{lemma}

\begin{proof}
Since \(\mu_\ell\to\infty\), we may use the preceding fourth-moment bounds
with \(m=\mu_\ell\).
We use \eqref{eq:stationary-zero-count-fourth-moment},
\eqref{eq:finite-range-package-matching},
\eqref{eq:finite-range-growing-fourth-moment} and \(K_\ell\leq2\ell\) to obtain
\[
 \begin{aligned}
  \Pp\!\left(
   N_{\infty,\mu_\ell}(I)\neq N_\infty(I)
  \right)
  &\leq C\ell^{-16},\\
  \E\!\left[N_{\infty,\mu_\ell}(I)^4\right]
  +\E\!\left[N_\infty(I)^4\right]
  &\leq M_4K_\ell^4+C\ell^4\leq C\ell^4.
 \end{aligned}
\]
We substitute these estimates into
\eqref{eq:rare-disagreement-L2-transfer} and obtain
\[
 \begin{aligned}
  \left\|
   N_{\infty,\mu_\ell}(I)-N_\infty(I)
  \right\|_{L^{2}}
  &\leq C(\ell^4)^{1/4}(\ell^{-16})^{1/4}
  =C\ell^{-3},
 \end{aligned}
\]
which proves \eqref{eq:finite-range-BE-L2-transfer}.  We use
\eqref{eq:rare-disagreement-mean-sd-transfer} and
\eqref{eq:finite-range-BE-L2-transfer} to obtain
\[
 \begin{aligned}
  \left|
   \E N_{\infty,\mu_\ell}(I)-\E N_\infty(I)
  \right|
  &\leq 2\left\|
    N_{\infty,\mu_\ell}(I)-N_\infty(I)
  \right\|_{L^{2}}
  \leq C\ell^{-3},\\
  \left|
   \sqrt{\Var(N_{\infty,\mu_\ell}(I))}
   -\sqrt{\Var(N_\infty(I))}
  \right|
  &\leq 2\left\|
    N_{\infty,\mu_\ell}(I)-N_\infty(I)
  \right\|_{L^{2}}
  \leq C\ell^{-3},
 \end{aligned}
\]
which prove \eqref{eq:finite-range-BE-mean-transfer} and
\eqref{eq:finite-range-BE-standard-deviation-transfer}.

We use stationarity and \eqref{eq:stationary-linear-variance-all-L} to
obtain
\begin{equation}
 \frac{V_\infty}{2}\ell
 \leq
 \Var(N_\infty(I))
 \leq
 \frac{3V_\infty}{2}\ell.
 \label{eq:stationary-two-sided-linear-variance}
\end{equation}
We combine
\eqref{eq:finite-range-BE-standard-deviation-transfer} and
\eqref{eq:stationary-two-sided-linear-variance} to obtain
\[
 \begin{aligned}
  \sqrt{\Var(N_{\infty,\mu_\ell}(I))}
  &\geq
  \sqrt{\Var(N_\infty(I))}-C\ell^{-3}\\
  &\geq
  \sqrt{\frac{V_\infty}{2}}\,\ell^{1/2}-C\ell^{-3}\\
  &\geq
  \frac12\sqrt{\frac{V_\infty}{2}}\,\ell^{1/2},\\
  \sqrt{\Var(N_{\infty,\mu_\ell}(I))}
  &\leq
  \sqrt{\Var(N_\infty(I))}+C\ell^{-3}\\
  &\leq
  \sqrt{\frac{3V_\infty}{2}}\,\ell^{1/2}+C\ell^{-3}\\
  &\leq
  C\ell^{1/2}.
 \end{aligned}
\]
We square these bounds and obtain
\[
 \frac{V_\infty}{8}\ell
 \leq
 \Var(N_{\infty,\mu_\ell}(I))
 \leq
 C\ell,
\]
which proves \eqref{eq:finite-range-BE-linear-variance}.
\end{proof}

\subsubsection{Block decomposition and Gaussian approximation for
\texorpdfstring{\(N_{\infty,m}(I)\)}{N infinity,m(I)}}

For \(m\geq m_1\), every \(I\subset\R\) with \(\ell\geq1\) and
\(1\leq j\leq K_\ell\), we define
\begin{equation}
 U_{m,I,j}
 =
 N_{\infty,m}(\mathcal J_{I,j})
 -
 \E N_{\infty,m}(\mathcal J_{I,j}).
 \label{eq:finite-range-centered-block-count}
\end{equation}
We sum the zero counts over the disjoint partition in
\eqref{eq:finite-range-block-partition} and obtain the identity
\begin{equation}
 N_{\infty,m}(I)
 -
 \E N_{\infty,m}(I)
 =
 \sum_{j=1}^{K_\ell}U_{m,I,j}.
 \label{eq:finite-range-centered-block-decomposition}
\end{equation}

The finite support of the moving-average kernel turns the centered block
decomposition into a locally dependent sequence.

\begin{lemma}
\label{lem:finite-range-block-dependence}
For every \(m\geq m_1\) and every \(I\subset\R\) with \(\ell\geq1\),
the sequence
\((U_{m,I,j})_{1\leq j\leq K_\ell}\) is \(r_m\)-dependent in the sense
of Chen and Shao, where
\begin{equation}
 r_m=\lceil m\rceil+1.
 \label{eq:finite-range-block-dependence-radius}
\end{equation}
\end{lemma}

\begin{proof}
We apply Lemma~\ref{lem:finite-range-zero-count-regularity} and conclude
that every random variable in
\eqref{eq:finite-range-centered-block-count} is
measurable with respect to the \(\Pp\)-augmented \(\sigma\)-field generated
by the restriction of \(q_{\infty,m}\) to its block.  We use this
measurability to transfer independence of the process restrictions to
the block zero counts.

We fix nonempty \(A,D\subseteq\{1,\ldots,K_\ell\}\) satisfying
\[
 \min_{i\in A,\,j\in D}|i-j|>r_m.
\]
For every \(i\in A\) and \(j\in D\),
we use \eqref{eq:finite-range-block-partition} to obtain
\[
 \operatorname{dist}
 \bigl(\mathcal J_{I,i},\mathcal J_{I,j}\bigr)
 \geq |i-j|-1.
\]
We use \eqref{eq:finite-range-block-dependence-radius} to rewrite the
preceding assumption as \(|i-j|>\lceil m\rceil+1\).  Since \(|i-j|\) is an
integer, we have
\[
 |i-j|\geq\lceil m\rceil+2.
\]
Therefore, we obtain
\[
 \operatorname{dist}\left(
  \bigcup_{i\in A}\mathcal J_{I,i},
  \bigcup_{j\in D}\mathcal J_{I,j}
 \right)
 \geq
 \min_{i\in A,\,j\in D}(|i-j|-1)
 \geq\lceil m\rceil+1
 >m.
\]
We apply Lemma~\ref{lem:m-dependent-process-properties} and conclude
that the two process restrictions are independent.  Their
\(\Pp\)-augmented \(\sigma\)-fields remain independent.  Indeed, an
event \(\overline E\) in the first augmentation and an event
\(\overline F\) in the second augmentation differ by null sets from
events \(E\) and \(F\), respectively, in the unaugmented
\(\sigma\)-fields.  Therefore
\[
 \Pp(\overline E\cap\overline F)
 =
 \Pp(E\cap F)
 =
 \Pp(E)\Pp(F)
 =
 \Pp(\overline E)\Pp(\overline F).
\]
 Consequently,
\[
 (U_{m,I,i})_{i\in A}
 \quad\text{and}\quad
 (U_{m,I,j})_{j\in D}
\]
are independent.  This is precisely \(r_m\)-dependence.
\end{proof}

The preceding dependence and moment estimates yield the following
quantitative Gaussian approximation for every \(m\geq m_1\).

\begin{theorem}
\label{thm:finite-range-zero-count-BE}
With \(m_1\) and \(M_4\) as in
Proposition~\ref{prop:finite-range-uniform-local-moments}, every
\(m\geq m_1\) and every \(I\subset\R\) with \(\ell\geq1\) for which
\(\Var(N_{\infty,m}(I))>0\) satisfy
\begin{equation}
 \begin{aligned}
  &\sup_{z\in\R}
  \left|
   \Pp\!\left(
    \frac{
     N_{\infty,m}(I)
     -
     \E N_{\infty,m}(I)
    }{
     \sqrt{\Var(N_{\infty,m}(I))}
    }
    \leq z
   \right)
   -
   \Phi(z)
  \right|\\
  {}\leq{}&
  600M_4^{3/4}
  (10\lceil m\rceil+11)^{2}
  \frac{
    K_\ell
  }{
    \Var(N_{\infty,m}(I))^{3/2}
  }.
 \end{aligned}
 \label{eq:finite-range-general-BE-bound}
\end{equation}
\end{theorem}

\begin{proof}
We fix \(m\geq m_1\) and \(I\) with \(\ell\geq1\), and we assume that
\(\Var(N_{\infty,m}(I))>0\).  We define
\begin{equation}
 \widetilde U_{m,I,j}
 =
 \frac{U_{m,I,j}}
 {\sqrt{\Var(N_{\infty,m}(I))}},
 \qquad 1\leq j\leq K_\ell.
 \label{eq:finite-range-normalized-block-count}
\end{equation}
We combine
\eqref{eq:finite-range-centered-block-count} with
\eqref{eq:finite-range-normalized-block-count} and obtain
\[
 \E\widetilde U_{m,I,j}=0.
\]
We combine
\eqref{eq:finite-range-centered-block-decomposition} with
\eqref{eq:finite-range-normalized-block-count} and obtain
\begin{equation}
 \Var\!\left(
  \sum_{j=1}^{K_\ell}\widetilde U_{m,I,j}
 \right)
 =1.
 \label{eq:finite-range-normalized-block-sum-variance}
\end{equation}
We apply Lemma~\ref{lem:finite-range-block-dependence} and conclude
that the normalized sequence remains \(r_m\)-dependent.  Moreover, we
apply \eqref{eq:finite-range-uniform-centered-third-moment} to every
block and obtain
\begin{equation}
 \sum_{j=1}^{K_\ell}
 \E|\widetilde U_{m,I,j}|^3
 =
 \frac{
  \displaystyle
  \sum_{j=1}^{K_\ell}\E|U_{m,I,j}|^3
 }{
  \Var(N_{\infty,m}(I))^{3/2}
 }
 \leq
  8M_4^{3/4}
  \frac{K_\ell}
  {\Var(N_{\infty,m}(I))^{3/2}}.
 \label{eq:finite-range-normalized-third-moment-sum}
\end{equation}

 We use Equations~\eqref{eq:finite-range-centered-block-count} and
 \eqref{eq:finite-range-normalized-block-count} to write the normalized block
 decomposition.  We combine
\eqref{eq:finite-range-normalized-block-sum-variance},
\eqref{eq:finite-range-normalized-third-moment-sum} and
Lemma~\ref{lem:finite-range-block-dependence} to verify the required
hypotheses.  We then apply
Lemma~\ref{lem:appendix-chen-shao-thm2-6} with \(d=1\) and \(p=3\) to obtain
\begin{equation}
 \sup_{z\in\R}
 \left|
  \Pp\!\left(
   \sum_{j=1}^{K_\ell}\widetilde U_{m,I,j}\leq z
  \right)
  -
  \Phi(z)
 \right|
 \leq
 75(10r_m+1)^{2}
 \sum_{j=1}^{K_\ell}\E|\widetilde U_{m,I,j}|^3.
 \label{eq:finite-range-Chen-Shao-bound}
\end{equation}
We substitute
\eqref{eq:finite-range-normalized-third-moment-sum} and
 \(10r_m+1=10(\lceil m\rceil+1)+1=10\lceil m\rceil+11\)
into \eqref{eq:finite-range-Chen-Shao-bound}.  We therefore obtain
\eqref{eq:finite-range-general-BE-bound}.
\end{proof}

\begin{proof}[Proof of Corollary~\ref{cor:finite-range-BE-package}]
Let \(b\) and \(\ell_0\) be as in
Corollary~\ref{cor:finite-range-zero-matching-package}.  We choose
\(\ell_1\geq\ell_0\) so that, whenever \(\ell\geq\ell_1\),
Lemma~\ref{lem:finite-range-BE-variance-transfer} applies,
\(\mu_\ell\geq m_1\) and
\(10\lceil\mu_\ell\rceil+11\leq31\mu_\ell\).
For \(I\) as in the statement, we apply
Corollary~\ref{cor:finite-range-zero-matching-package}
and obtain
\eqref{eq:finite-range-package-count-comparison}.  We apply
Lemma~\ref{lem:finite-range-BE-variance-transfer} to obtain
\eqref{eq:finite-range-package-L2-comparison}--%
\eqref{eq:finite-range-package-variance}.  We use
\(10\lceil\mu_\ell\rceil+11\leq31\mu_\ell\),
\(K_\ell\leq2\ell\),
\eqref{eq:finite-range-BE-linear-variance} and
\(\mu_\ell^2=b^2\log\ell\) in
\eqref{eq:finite-range-general-BE-bound} to bound the left-hand side of
\eqref{eq:finite-range-package-BE} by
\[
 600M_4^{3/4}(10\lceil\mu_\ell\rceil+11)^2
 \frac{K_\ell}
 {\Var(N_{\infty,\mu_\ell}(I))^{3/2}}
 \leq C\mu_\ell^2\frac{\ell}{\ell^{3/2}}
 =Cb^2\frac{\log\ell}{\sqrt\ell}
 \leq C\frac{\log\ell}{\sqrt\ell}.
\]
This proves \eqref{eq:finite-range-package-BE}.
\end{proof}

\subsubsection{Proof of Lemma~\ref{lem:finite-range-exact-coupling-transfer}}

 Using the assumption, we obtain
\[
 \sqrt{\Var(Y)}
 \geq
 \sqrt{\Var(X)}
 -
 \left|\sqrt{\Var(X)}-\sqrt{\Var(Y)}\right|
 \geq
 \frac12\sqrt{\Var(X)}
 >0.
\]
 We also obtain
\[
 \frac12
 \leq
 \sqrt{\frac{\Var(Y)}{\Var(X)}}
 \leq
 \frac32
\]
and
\[
 \max\left\{
  \left|
   \sqrt{\frac{\Var(Y)}{\Var(X)}}-1
  \right|,
  \frac{|\E Y-\E X|}{\sqrt{\Var(X)}}
 \right\}
 \leq
 \frac{
  |\E X-\E Y|
  +
  \left|\sqrt{\Var(X)}-\sqrt{\Var(Y)}\right|
 }{
  \sqrt{\Var(X)}
 }.
\]
For every \(z\in\R\), the two indicators
\[
 \mathbf 1_{\{Y\leq\E Y+\sqrt{\Var(Y)}z\}}
 \quad\text{and}\quad
 \mathbf 1_{\{X\leq\E Y+\sqrt{\Var(Y)}z\}}
\]
agree on \(\{X=Y\}\).  Consequently,
\begin{equation}
 \left|
  \Pp\!\left(
   \frac{Y-\E Y}{\sqrt{\Var(Y)}}\leq z
  \right)
  -
  \Pp\!\left(
   \frac{X-\E X}{\sqrt{\Var(X)}}
   \leq
   \sqrt{\frac{\Var(Y)}{\Var(X)}}\,z
   +
   \frac{\E Y-\E X}{\sqrt{\Var(X)}}
  \right)
 \right|
 \leq
 \Pp(X\neq Y).
 \label{eq:exact-coupling-probability-comparison}
\end{equation}

We apply the mean value theorem and use
\(\sqrt{\Var(Y)/\Var(X)}\in[1/2,3/2]\) to obtain
\[
 \begin{aligned}
 &\left|
  \Phi\!\left(
   \sqrt{\frac{\Var(Y)}{\Var(X)}}\,z
  \right)
  -
  \Phi(z)
 \right|
 \leq
 \left|
  \sqrt{\frac{\Var(Y)}{\Var(X)}}-1
 \right|
 |z|
 \frac{e^{-z^2/8}}{\sqrt{2\pi}}
 \leq
 \left|
  \sqrt{\frac{\Var(Y)}{\Var(X)}}-1
 \right|,\\
 &\left|
  \Phi\!\left(
   \sqrt{\frac{\Var(Y)}{\Var(X)}}\,z
   +
   \frac{\E Y-\E X}{\sqrt{\Var(X)}}
  \right)
  -
  \Phi\!\left(
   \sqrt{\frac{\Var(Y)}{\Var(X)}}\,z
  \right)
 \right|
 \leq
 \frac{|\E Y-\E X|}{\sqrt{2\pi\Var(X)}}.
 \end{aligned}
\]
We combine these estimates with the preceding maximum bound and obtain
\begin{equation}
 \begin{aligned}
 &\sup_{z\in\R}
 \left|
  \Phi\!\left(
   \sqrt{\frac{\Var(Y)}{\Var(X)}}\,z
   +
   \frac{\E Y-\E X}{\sqrt{\Var(X)}}
  \right)
  -
  \Phi(z)
 \right|\\
 \leq{}&
 2
 \frac{
  |\E X-\E Y|
  +
  \left|\sqrt{\Var(X)}-\sqrt{\Var(Y)}\right|
 }{
  \sqrt{\Var(X)}
 }.
 \end{aligned}
 \label{eq:exact-coupling-Gaussian-comparison}
\end{equation}
We combine \eqref{eq:Kolmogorov-distance-definition} with
\eqref{eq:exact-coupling-probability-comparison} and
\eqref{eq:exact-coupling-Gaussian-comparison} to obtain
\[
 \begin{aligned}
 d_{\mathrm K}\!\left(
  \frac{Y-\E Y}{\sqrt{\Var(Y)}},Z
 \right)
 &=
 \sup_{z\in\R}
 \left|
  \Pp\!\left(
   \frac{Y-\E Y}{\sqrt{\Var(Y)}}\leq z
  \right)
  -
  \Phi(z)
 \right|\\
 &\leq
 d_{\mathrm K}\!\left(
  \frac{X-\E X}{\sqrt{\Var(X)}},Z
 \right)
 +
 \Pp(X\neq Y)\\
 &\quad+
 2
 \frac{
  |\E X-\E Y|
  +
  \left|\sqrt{\Var(X)}-\sqrt{\Var(Y)}\right|
 }{
  \sqrt{\Var(X)}
 }.
 \end{aligned}
\]
This proves the assertion.

This completes the proofs of the auxiliary results used in
Section~\ref{sec:proof-bulk-main-theorem}.

\section{Proofs of the additional results for Theorem~\ref{thm:full-line-polynomial-BE}}
\label{sec:full-line-BE}

We prove the five results from
Section~\ref{subsec:full-line-results} in their stated order.
We first derive estimates for \(H_n\) and \(\widetilde H_n\) and apply
Jensen's formula.  We then prove
Propositions~\ref{prop:full-line-edge-strip-count},
\ref{prop:full-line-low-band-localization}
and~\ref{prop:full-line-upper-band-localization}, followed by
Lemmas~\ref{lem:full-line-independent-transfer}
and~\ref{lem:full-line-additive-transfer}.

We retain \(B=20\), the regions and the coefficient cutoff from
Section~\ref{subsec:full-line-results}.  Since \(n\) is
sufficiently large, \(0\leq m_n<n\).  We now turn to the comparison of
\(N_n(I_n^{\mathrm{out}})\) with
\(N_{P_n-P_{m_n}}(I_n^{\mathrm{out}})\).

For \(0\leq j\leq n\), we define
\begin{equation}
 \kappa_{n,j}
 =
 \left(\frac{n!}{(n-j)!n^j}\right)^{1/2},
 \qquad
 \kappa_{n,j}^2
 =
 \prod_{r=0}^{j-1}
 \left(1-\frac{r}{n}\right),
 \qquad 0\leq j\leq n.
 \label{eq:full-line-reciprocal-coefficients}
\end{equation}
Here and below we interpret the product for \(j=0\) as empty.  In particular,
\(\kappa_{n,0}=\kappa_{n,1}=1\).
Using these coefficients, define the following two real polynomials:
\begin{equation}
 H_n(u)=\sum_{j=0}^n\kappa_{n,j}\xi_{n-j}u^j,
 \qquad
 \widetilde H_n(u)=
 \sum_{j=0}^{d_n-1}\kappa_{n,j}\xi_{n-j}u^j,
 \qquad u\in\R.
 \label{eq:full-line-reciprocal-expansions}
\end{equation}
 For every \(u\in\R\setminus\{0\}\), we directly reindex the sums to obtain
\begin{equation}
 H_n(u)
 =
 \frac{\sqrt{n!}}{n^{n/2}}u^n
 P_n\!\left(\frac{\sqrt n}{u}\right),
 \qquad
 \widetilde H_n(u)
 =
 \frac{\sqrt{n!}}{n^{n/2}}u^n
 (P_n-P_{m_n})\!\left(\frac{\sqrt n}{u}\right).
 \label{eq:full-line-reciprocal-identities}
\end{equation}
We use \eqref{eq:full-line-reciprocal-identities} to conclude that the real
change of variables \(x=\sqrt n/u\) preserves zero multiplicities and maps
\((0,1)\) onto \((\sqrt n,\infty)\) and \((-1,0)\) onto
\((-\infty,-\sqrt n)\).  Moreover,
\(H_n(0)=\widetilde H_n(0)=\xi_n\neq0\) almost surely.  Thus
\begin{equation}
 N_n(I_n^{\mathrm{out}})
 =
 N_{H_n}((-1,1)),
 \qquad
 N_{P_n-P_{m_n}}(I_n^{\mathrm{out}})
 =
 N_{\widetilde H_n}((-1,1))
 \quad\text{almost surely}.
 \label{eq:full-line-exterior-reciprocal-identity}
\end{equation}
The upper index \(d_n-1\) in
\eqref{eq:full-line-reciprocal-expansions} ensures that
\(\widetilde H_n\) depends exactly on
\(\xi_{m_n+1},\ldots,\xi_n\).

Finally, replacing \(\xi_k\) by
\((-1)^{n-k}\xi_k\) preserves their joint law and simultaneously
replaces \(H_n(u),\widetilde H_n(u)\) by
\(H_n(-u),\widetilde H_n(-u)\).  Therefore
\begin{equation}
 (H_n(-\mathord\cdot),\widetilde H_n(-\mathord\cdot))
 \stackrel{\mathrm d}=
 (H_n,\widetilde H_n)
 \label{eq:full-line-pair-symmetry}
\end{equation}
as a pair of random real polynomials.  Consequently,
\[
 \bigl(N_{H_n}((0,1)),N_{\widetilde H_n}((0,1))\bigr)
 \stackrel{\mathrm d}=
 \bigl(N_{H_n}((-1,0)),N_{\widetilde H_n}((-1,0))\bigr).
\]
This is only a distributional identity.

\subsection{Jensen bounds for zeros outside
\texorpdfstring{\([-\sqrt n,\sqrt n]\)}{[-sqrt(n), sqrt(n)]}}
\label{sec:full-line-jensen-estimates}

For the Jensen arguments in this subsection, we regard \(H_n\) as the complex
polynomial given by the same formula in
\eqref{eq:full-line-reciprocal-expansions}, retaining the notation \(H_n\).

For \(c\in\mathbb C\) and \(r>0\), we write
\begin{equation}
 D(c,r)=\{z\in\mathbb C:|z-c|<r\}.
 \label{eq:full-line-Jensen-disk-notation}
\end{equation}
 If \(F\) is a nonzero entire function, let \(\nu_F(D)\) denote its number of
 complex zeros in a disk \(D\), counted with multiplicity.  In each
 application, every real zero in the relevant interval is among the complex
 zeros in the covering disk.  Thus, by estimating \(\nu_F(D)\), we also
 control the relevant real-zero count.

 We first derive the required bound from second-moment control and a Gaussian
 small-ball estimate at the center.

\begin{lemma}
\label{lem:full-line-Gaussian-Jensen}
Let \(F\) be a random entire function, let \(c\in\R\) and suppose
that \(F(c)\) is a centered real Gaussian random variable.  With the
disk notation in \eqref{eq:full-line-Jensen-disk-notation}, let
\begin{equation}
 0<r<R_1<R_2,
 \qquad
 \Delta_R=R_2-R_1.
 \label{eq:full-line-Jensen-radius-parameters}
\end{equation}
Suppose further that there are deterministic numbers \(v_0,V>0\)
such that
\begin{equation}
 \Var(F(c))\geq v_0,
 \qquad
 \sup_{z\in D(c,R_2)}\E|F(z)|^2\leq V.
 \label{eq:full-line-Jensen-variance-assumptions}
\end{equation}
Then, for every \(0<\eta<1\), outside an event of probability at
most \(2\eta\),
\begin{equation}
 \nu_F(D(c,r))
 \leq
 \frac{1}{\log(R_1/r)}
 \left\{
  \log\frac{R_2}{\Delta_R}
  +
  \frac12\log\frac{V}{v_0}
  +
  \frac32\log(\eta^{-1})
 \right\}.
 \label{eq:full-line-Gaussian-Jensen-conclusion}
\end{equation}
\end{lemma}

\begin{proof}
We set
\[
 M=\sup_{z\in\overline{D(c,R_1)}}|F(z)|.
\]
We use \eqref{eq:full-line-Jensen-radius-parameters} to see that every
\(z\in\overline{D(c,R_1)}\) satisfies
\(D(z,\Delta_R)\subset D(c,R_2)\).  Since \(|F|^2\) is
subharmonic, we apply the mean-value inequality and obtain
\[
 |F(z)|^2
 \leq
 \frac{1}{\pi\Delta_R^2}
 \int_{D(z,\Delta_R)}|F(w)|^2\dd w
 \leq
 \frac{1}{\pi\Delta_R^2}
 \int_{D(c,R_2)}|F(w)|^2\dd w,
\]
where \(\dd w\) denotes planar Lebesgue measure.  We take the
supremum over \(z\), take expectations and use
\eqref{eq:full-line-Jensen-variance-assumptions} to obtain
\[
 \E M^2\leq\frac{R_2^2}{\Delta_R^2}V.
\]
We therefore apply Markov's inequality and obtain
\begin{equation}
 \Pp\!\left(
  M^2>\frac{R_2^2V}{\Delta_R^2\eta}
 \right)
 \leq\eta.
 \label{eq:full-line-Jensen-maximum-event}
\end{equation}

The density of \(F(c)\) is at most
\((2\pi v_0)^{-1/2}\).  Consequently,
\begin{equation}
 \Pp\bigl(|F(c)|<\eta\sqrt{v_0}\bigr)
 \leq
 \sqrt{\frac{2}{\pi}}\,\eta
 <\eta.
 \label{eq:full-line-Jensen-small-ball}
\end{equation}
On the complement of the two events in
\eqref{eq:full-line-Jensen-maximum-event} and
\eqref{eq:full-line-Jensen-small-ball}, we apply Jensen's formula to the
complex zeros of \(F\) and obtain
\[
 \nu_F(D(c,r))\log\frac{R_1}{r}
 \leq
 \log M-\log|F(c)|.
\]
If a zero lies on the circle of radius \(R_1\), we apply Jensen's
formula first at radii increasing to \(R_1\) and avoiding the
finitely many zero moduli and then pass to the limit.  We combine the
two good-event bounds and obtain
\[
 \log M-\log|F(c)|
 \leq
 \log\frac{R_2}{\Delta_R}
 +
 \frac12\log\frac{V}{v_0}
 +
 \frac32\log(\eta^{-1}).
\]
We divide by \(\log(R_1/r)>0\) to prove
\eqref{eq:full-line-Gaussian-Jensen-conclusion}.  We then use a union
bound and \(\sqrt{2/\pi}<1\) to obtain the asserted failure probability.
\end{proof}

To apply the preceding Jensen estimate to \(H_n\), we next establish its
second-moment profile near the unit interval.

\begin{lemma}
\label{lem:full-line-reciprocal-variance-profile}
For every \(\rho\geq0\),
\[
 \E|H_n(\rho)|^2
 =
 \sum_{j=0}^n\kappa_{n,j}^2\rho^{2j}.
\]
For \(0\leq\rho\leq1\),
\begin{equation}
 \frac{1}{8e}
 \frac{1}{1-\rho+n^{-1/2}}
 \leq
 \E|H_n(\rho)|^2
 \leq
 8
 \frac{1}{1-\rho+n^{-1/2}}.
 \label{eq:full-line-reciprocal-variance-inside}
\end{equation}
For \(1\leq\rho\leq1+1/n\),
\begin{equation}
 \frac{\sqrt n}{8e}
 \leq
 \E|H_n(\rho)|^2
 \leq
 8e^2\sqrt n.
 \label{eq:full-line-reciprocal-variance-outside}
\end{equation}
\end{lemma}

\begin{proof}
We combine the product identity in
\eqref{eq:full-line-reciprocal-coefficients} with
\(\log(1-x)\leq-x\) and obtain
\begin{equation}
 \kappa_{n,j}^2
 \leq
 \exp\left\{-\frac{j(j-1)}{2n}\right\},
 \qquad 0\leq j\leq n.
 \label{eq:full-line-reciprocal-Gaussian-upper}
\end{equation}
Equality holds in this estimate for \(j=0,1\).  We use Gaussian decay
only from \(j=2\) onward.  We first suppose that
\(0\leq\rho\leq1\) and we set
\[
 a=1-\rho,
 \qquad
 h=n^{-1/2}.
\]
For integers \(2\leq j\leq n\),
\[
 j(j-1)\geq\frac{j^2}{2},
 \qquad
 \rho^{2j}\leq e^{-2aj}.
\]
Consequently, we obtain
\[
 \begin{aligned}
  \E|H_n(\rho)|^2
  &=
  1+\rho^2+\sum_{j=2}^n
  \kappa_{n,j}^2\rho^{2j}\\
  &\leq
  2+\sum_{j=2}^\infty
  \exp\left\{-2aj-\frac{j^2}{4n}\right\}\\
  &\leq
  2+\min\left\{\frac{1}{2a},\sqrt{\pi n}\right\}
  \leq
  \frac{8}{a+h}.
 \end{aligned}
\]
When \(a=0\), we interpret the first entry in the minimum as infinite.
Here we used
\[
 \sum_{j=1}^\infty e^{-j^2/(4n)}
 \leq
 \int_0^\infty e^{-x^2/(4n)}\dd x
 =
 \sqrt{\pi n}
\]
and
\(\min\{a^{-1},h^{-1}\}\leq2/(a+h)\).  At \(\rho=0\), we use
\(\E|H_n(0)|^2=1\) to obtain the upper bound.

For integers \(0\leq j\leq\lfloor\sqrt n\rfloor\), we apply the elementary
product inequality
\(\prod_i(1-x_i)\geq1-\sum_i x_i\), for \(x_i\in[0,1]\), and obtain
\[
 \kappa_{n,j}^2
 \geq
 1-\frac{j(j-1)}{2n}
 \geq
 \frac12.
\]
If \(\rho\geq1/2\), then
\(\log(1-a)\geq-2a\), so
\(\rho^{2j}\geq e^{-4aj}\).  We set
\[
 K=\left\lfloor\frac{1}{4(a+h)}\right\rfloor.
\]
For \(0\leq j\leq K\), we have \(0\leq K\leq\sqrt n/4\leq n\) and
\(4aj\leq1\).  We therefore obtain
\[
 \E|H_n(\rho)|^2
 \geq
 \sum_{j=0}^K\kappa_{n,j}^2\rho^{2j}
 \geq
 \frac{K+1}{2e}
 \geq
 \frac{1}{8e(a+h)}.
\]
If \(0\leq\rho<1/2\), then \(a+h>1/2\).  We retain only the constant
term and obtain the same lower bound.  This proves
\eqref{eq:full-line-reciprocal-variance-inside}.

We finally suppose that \(1\leq\rho\leq1+1/n\).  We use positivity and
\eqref{eq:full-line-reciprocal-variance-inside} to obtain
\[
 \E|H_n(\rho)|^2
 \geq
 \E|H_n(1)|^2
 \geq
 \frac{\sqrt n}{8e}.
\]
For \(0\leq j\leq n\),
\[
 \rho^{2j}
 \leq
 \left(1+\frac1n\right)^{2n}
 \leq e^2.
\]
Hence
\[
 \E|H_n(\rho)|^2
 \leq
 e^2\E|H_n(1)|^2
 \leq
 8e^2\sqrt n,
\]
which proves \eqref{eq:full-line-reciprocal-variance-outside}.
\end{proof}

The next proposition combines this profile with a dyadic covering argument
based on Jensen's formula to prove an \(L^2\) bound for
\(N_n(I_n^{\mathrm{out}})\).

\begin{proposition}
\label{prop:full-line-exterior-second-moment}
There is a finite absolute constant \(C\) such that, for all
sufficiently large \(n\),
\begin{equation}
 \E[N_n(I_n^{\mathrm{out}})^2]
 \leq
 C(\log n)^2.
 \label{eq:full-line-exterior-second-moment}
\end{equation}
\end{proposition}

\begin{proof}
We set
\[
 J_n=\lceil\log_2n\rceil,
 \qquad
 \vartheta_j=2^{-j},
 \quad 1\leq j\leq J_n.
\]
We define the centers
\[
 c_j=1-\frac32\vartheta_j,\quad
 1\leq j\leq J_n,
 \qquad
 c_*=1-\frac12\vartheta_{J_n}.
\]
For \(1\leq j\leq J_n\), we use the inclusions
\[
 \bigl[1-2\vartheta_j,1-\vartheta_j\bigr]
 \subset
 D\left(c_j,\frac35\vartheta_j\right)
 \quad\text{and}\quad
 \bigl[1-\vartheta_{J_n},1\bigr)
 \subset
 D\left(c_*,\frac35\vartheta_{J_n}\right)
\]
to see that these \(J_n+1\) disks cover \((0,1)\).  For
\(1\leq j\leq J_n\), we set
\((c,\vartheta)=(c_j,\vartheta_j)\), and for the disk centered at \(c_*\),
we set \((c,\vartheta)=(c_*,\vartheta_{J_n})\).  In all \(J_n+1\)
applications of Lemma~\ref{lem:full-line-Gaussian-Jensen}, we choose
\[
 (r,R_1,R_2)
 =
 \vartheta\left(\frac35,\frac45,1\right).
\]
Thus \(R_1/r=4/3\) and
\(R_2/(R_2-R_1)=5\).  We use independence of the coefficients to obtain
\(\E|H_n(z)|^2=\E|H_n(|z|)|^2\).

For \(1\leq j\leq J_n\) and \(z\in D(c_j,\vartheta_j)\), we have
\[
 |z|
 \leq
 c_j+\vartheta_j
 =
 1-\frac12\vartheta_j.
\]
We therefore apply
Lemma~\ref{lem:full-line-reciprocal-variance-profile} and obtain
\[
 \frac{
  \sup_{z\in D(c_j,\vartheta_j)}\E|H_n(z)|^2
 }{
  \Var(H_n(c_j))
 }
 \leq
 64e
 \frac{(3/2)\vartheta_j+n^{-1/2}}
 {(1/2)\vartheta_j+n^{-1/2}}
 \leq
 192e.
\]

For the choice \((c,\vartheta)=(c_*,\vartheta_{J_n})\), we have
\[
 c_*+\vartheta_{J_n}
 =
 1+\frac12\vartheta_{J_n}
 \leq
 1+\frac1{2n}.
\]
Although the real zeros being counted lie in \(0<u<1\), Jensen's
estimate requires a second-moment upper bound throughout the outer disk.
We have just shown that its radial range lies in
\([0,1+1/n]\).  This is precisely the role of
\eqref{eq:full-line-reciprocal-variance-outside}.  Since
\[
 1-c_*+n^{-1/2}
 =
 \frac12\vartheta_{J_n}+n^{-1/2}
 \leq
 \frac{3}{2\sqrt n},
\]
we combine the inside and outside estimates in
Lemma~\ref{lem:full-line-reciprocal-variance-profile} with the
monotonicity of \(\E|H_n(\rho)|^2\) in \(\rho\) and obtain
\[
 \Var(H_n(c_*))\geq\frac{\sqrt n}{12e},
 \qquad
 \sup_{z\in D(c_*,\vartheta_{J_n})}\E|H_n(z)|^2
 \leq8e^2\sqrt n.
\]
Consequently, the ratio \(V/v_0\) in
Lemma~\ref{lem:full-line-Gaussian-Jensen} is at most \(96e^3\) for
every disk.  We apply Lemma~\ref{lem:full-line-Gaussian-Jensen} with
\[
 F=H_n,\qquad
 v_0=\Var(H_n(c)),\qquad
 V=\sup_{z\in D(c,\vartheta)}\E|H_n(z)|^2,\qquad
 \eta=e^{-(1+t)}.
\]
The radius ratios are fixed, so we can choose a numerical constant
\(C_1\), independent of \(n\) and of the chosen disk, such that
\[
 \Pp\!\left(
   \nu_{H_n}\left(D\left(c,\frac35\vartheta\right)\right)>C_1(1+t)
 \right)
 \leq
 2e^{-(1+t)},
 \qquad t\geq0.
\]
For \(x\geq C_1\), we substitute \(t=x/C_1-1\) and obtain
\[
 \Pp\!\left(
  \nu_{H_n}\left(D\left(c,\frac35\vartheta\right)\right)>x
 \right)
 \leq2e^{-x/C_1}.
\]
We use the tail integral identity and obtain
\begin{equation}
 \begin{aligned}
 &\E\left[
  \nu_{H_n}\left(D\left(c,\frac35\vartheta\right)\right)^2
 \right]\\
 ={}&
 2\int_0^\infty
 x\,\Pp\!\left(
  \nu_{H_n}\left(D\left(c,\frac35\vartheta\right)\right)>x
 \right)\dd x\\
 {}\leq{}&
 C_1^2+4\int_{C_1}^\infty xe^{-x/C_1}\dd x
 \leq C.
 \end{aligned}
 \label{eq:full-line-uniform-disk-second-moment}
\end{equation}

Recall that \(\nu_{H_n}(D)\) is the complex-zero count introduced at
the beginning of Section~\ref{sec:full-line-jensen-estimates}.  Every real zero of \(H_n\) in a
 covered subinterval is a complex zero in the corresponding disk, with
 the same multiplicity.  Since the disks cover \((0,1)\), we have
\[
 N_{H_n}((0,1))
 \leq
 \sum_{j=1}^{J_n}
 \nu_{H_n}\left(D\left(c_j,\frac35\vartheta_j\right)\right)
 +
 \nu_{H_n}\left(
  D\left(c_*,\frac35\vartheta_{J_n}\right)
 \right).
\]
We apply Cauchy--Schwarz and
\eqref{eq:full-line-uniform-disk-second-moment} to obtain
\begin{equation}
 \E[N_{H_n}((0,1))^2]
 \leq
 C(J_n+1)^2
 \leq
 C(\log n)^2.
 \label{eq:full-line-positive-exterior-second-moment}
\end{equation}
 We use the distributional identity \eqref{eq:full-line-pair-symmetry} to
 conclude that \eqref{eq:full-line-positive-exterior-second-moment} also holds with
\(N_{H_n}((0,1))\) replaced by \(N_{H_n}((-1,0))\).  Since
\(H_n(0)\neq0\) almost surely,
we combine \eqref{eq:full-line-exterior-reciprocal-identity} with
\((x+y)^2\leq2x^2+2y^2\) to prove
\eqref{eq:full-line-exterior-second-moment}.
\end{proof}

\subsection{Proofs of Propositions~\ref{prop:full-line-edge-strip-count},
\ref{prop:full-line-low-band-localization} and
\ref{prop:full-line-upper-band-localization}}

We first bound \(N_n(I_n^{\mathrm{edge}})\).

\begin{proof}[Proof of Proposition~\ref{prop:full-line-edge-strip-count}]
We first consider \([L_n,\sqrt n]\).  For the following Jensen argument, we
extend \(P_n\) and \(q_n\), retaining the same symbols, by
\[
 P_n(z)=\sum_{k=0}^{n}\frac{\xi_k}{\sqrt{k!}}z^k,
 \qquad
 q_n(z)=e^{-z^2/2}P_n(z),
 \qquad z\in\mathbb C.
\]
These extensions agree on \(\R\) with \eqref{eq:weyl-polynomial} and
 \eqref{eq:normalized-weyl-functions}.  For \(z=x+iy\in\mathbb C\), we use
 independence of the coefficients to obtain
\begin{equation}
 \E|q_n(z)|^2
 =
 e^{-\operatorname{Re}(z^2)}
 \sum_{k=0}^n\frac{|z|^{2k}}{k!}
 \leq
 e^{-\operatorname{Re}(z^2)}e^{|z|^2}
 =
 e^{2y^2}.
 \label{eq:full-line-complex-qn-variance}
\end{equation}

We apply
Lemma~\ref{lem:full-line-Gaussian-Jensen} to \(F=q_n\) with
\begin{equation}
 c=L_n,
 \qquad
 r=2B\sqrt{\log n},
 \qquad
 R_1=3B\sqrt{\log n},
 \qquad
 R_2=4B\sqrt{\log n}.
 \label{eq:full-line-edge-disk-radii}
\end{equation}
The interval \([L_n,\sqrt n]\) lies in
\(D(L_n,2B\sqrt{\log n})\).  We let
\(K\sim\operatorname{Poisson}(L_n^2)\).  At the real center,
\begin{equation}
 \Var(q_n(L_n))
 =
 e^{-L_n^2}\sum_{k=0}^n\frac{L_n^{2k}}{k!}
 =
 \Pp(K\leq n),
 \label{eq:full-line-edge-center-Poisson}
\end{equation}
We use \eqref{eq:logarithmic-edge-length-bounds} to obtain
\[
 n+1-L_n^2
 =
 2B\sqrt{n\log n}-B^2\log n+1
 \geq
 B\sqrt{n\log n}.
\]
We apply the Poisson estimate
\eqref{eq:proved-Poisson-Chernoff-bound} at the threshold
\(n+1\leq2n\) to obtain
\[
 \Pp(K\geq n+1)\leq n^{-B^2/4}.
\]
We combine this tail bound with
\eqref{eq:full-line-edge-center-Poisson} and use
\(n^{-B^2/4}\leq1/2\) to obtain
\begin{equation}
 \Var(q_n(L_n))\geq\frac12.
 \label{eq:full-line-edge-center-lower-bound}
\end{equation}

Every \(z\in D(L_n,4B\sqrt{\log n})\) satisfies
\(|\operatorname{Im}z|\leq4B\sqrt{\log n}\).  Hence we apply
\eqref{eq:full-line-complex-qn-variance} and obtain
\begin{equation}
 \sup_{z\in D(L_n,4B\sqrt{\log n})}\E|q_n(z)|^2
 \leq
 e^{32B^2\log n}
 =
 n^{32B^2}.
 \label{eq:full-line-edge-disk-upper-bound}
\end{equation}
Using \eqref{eq:full-line-edge-center-lower-bound} and
\eqref{eq:full-line-edge-disk-upper-bound}, we apply
Lemma~\ref{lem:full-line-Gaussian-Jensen} with
\[
 v_0=\frac12,
 \qquad
 V=n^{32B^2},
 \qquad
 \eta=\frac18n^{-A-2}.
\]
All radius ratios in \eqref{eq:full-line-edge-disk-radii} are
constant.  Recall that \(\nu_F(D)\) denotes the complex-zero count introduced
in Section~\ref{sec:full-line-jensen-estimates}.  Since \(e^{-z^2/2}\) never
vanishes, \(q_n\) and \(P_n\) have the same complex zeros, with the same
multiplicities.  Every real zero in
\([L_n,\sqrt n]\) lies in
\(D(L_n,2B\sqrt{\log n})\) and is therefore counted by the corresponding
complex-zero count, with the same multiplicity.  We substitute the
radius ratios into \eqref{eq:full-line-Gaussian-Jensen-conclusion} and obtain
\begin{equation}
 N_n([L_n,\sqrt n])
 \leq
 \nu_{q_n}(D(L_n,2B\sqrt{\log n}))
 \leq
 C_{A,B}\log n
 \label{eq:full-line-positive-edge-count}
\end{equation}
outside an event of probability at most
\(\frac14n^{-A-2}\).

We repeat the argument with center \(-L_n\) to obtain
\begin{equation}
 N_n([-\sqrt n,-L_n])
 \leq
 C_{A,B}\log n
 \label{eq:full-line-negative-edge-count}
\end{equation}
outside an event of probability at most
\(\frac14n^{-A-2}\).  Since
\[
 N_n(I_n^{\mathrm{edge}})
 \leq
 N_n([-\sqrt n,-L_n])
 +
 N_n([L_n,\sqrt n]),
\]
we combine \eqref{eq:full-line-positive-edge-count} and
\eqref{eq:full-line-negative-edge-count}.  We apply a union bound and
then enlarge \(C_{A,B}\) to prove
\eqref{eq:full-line-edge-strip-tail}.
\end{proof}

We first compare the numbers of zeros of \(P_{m_n}\) and \(P_n\) on
\(I_n\).  Recall \(\delta_n,\varepsilon_n\) from
\eqref{eq:logarithmic-edge-parameter-choice}.

The comparison is based on the following quantitative \(C^1\) approximation
on \(I_n\).

\begin{lemma}
\label{lem:full-line-low-band-C1}
For all sufficiently large \(n\),
\begin{equation}
 \Pp\!\left(
  \|q_\infty-q_{m_n}\|_{C^1(I_n)}
  \geq\varepsilon_n
 \right)
 \leq Cn^{-21}.
 \label{eq:full-line-low-band-C1}
\end{equation}
\end{lemma}

\begin{proof}
We fix \(r\in\{0,1,2\}\) and \(x\in I_n\), set
\(\lambda=x^2\) and let \(K\) be Poisson with mean \(\lambda\).
We use \(R_{m_n}=P_\infty-P_{m_n}\).  The algebraic computation in
\eqref{eq:polynomial-tail-Poisson-identity} and
\eqref{eq:truncated-Poisson-factorial-moment} is uniform in the truncation
index.  Applying these identities with cutoff \(m_n\), for every
\(0\leq a\leq r\), we have
\[
 \E\!\left[(K)_a\mathbf 1_{\{K\geq m_n+1-r\}}\right]
 =
 \lambda^a\Pp(K\geq m_n+1-r-a)
 \leq
 \lambda^a\Pp(K\geq m_n+1-2r).
\]
The numerical coefficients in the decompositions for \(r=0,1,2\) depend only
on \(r\).  Since \(\lambda\leq n\), we also have
\(\lambda^a\leq n^r\).  Therefore,
\begin{equation}
 e^{-x^2}
 \E\left|R_{m_n}^{(r)}(x)\right|^2
 \leq
 C_rn^r
 \Pp(K\geq m_n+1-2r).
 \label{eq:full-line-low-band-Poisson-reduction}
\end{equation}

From the definition of \(I_n\), we obtain
\begin{equation}
 \lambda
 \leq
 L_n^2
 =
 n-2B\sqrt{n\log n}+B^2\log n.
 \label{eq:full-line-low-band-Poisson-mean}
\end{equation}
Also,
\[
 m_n+1-2r
 \geq
 n-B\sqrt{n\log n}-2r.
\]
We combine \eqref{eq:full-line-low-band-Poisson-mean} with this lower
bound and obtain
\begin{equation}
 m_n+1-2r-\lambda
 \geq
 \frac{3B}{4}\sqrt{n\log n}.
 \label{eq:full-line-low-band-Poisson-gap}
\end{equation}
The threshold \(m_n+1-2r\) is at most \(n+1\leq9n/8\).
We combine the Poisson Chernoff bound
\eqref{eq:proved-Poisson-Chernoff-bound} with
\eqref{eq:full-line-low-band-Poisson-gap} and obtain
\begin{equation}
 \Pp(K\geq m_n+1-2r)
 \leq
 \exp\left\{
  -\frac{(m_n+1-2r-\lambda)^2}
  {2(m_n+1-2r)}
 \right\}
 \leq
 n^{-B^2/4}.
 \label{eq:full-line-low-band-Poisson-tail}
\end{equation}

The normalized remainder \(\rho_{m_n}\) is the specialization of
\eqref{eq:polynomial-series-error-definition} to the index \(m_n\).
We combine \eqref{eq:polynomial-series-error-derivatives},
\eqref{eq:full-line-low-band-Poisson-reduction} and
\eqref{eq:full-line-low-band-Poisson-tail} with
\(|x|\leq\sqrt n\) and obtain
\begin{equation}
 \sup_{x\in I_n}
 \E|\rho_{m_n}^{(j)}(x)|^2
 \leq
 Cn^{2-B^2/4},
 \qquad j\in\{0,1,2\}.
 \label{eq:full-line-low-band-normalized-derivatives}
\end{equation}
We use \eqref{eq:full-line-low-band-normalized-derivatives} and
partition \(I_n\) into \(O(\sqrt n)\) intervals whose lengths lie
between \(1/2\) and \(1\).  We apply
Lemma~\ref{lem:comparison-interval-Sobolev} with \(p=2\) to
\(\rho_{m_n}\) and \(\rho_{m_n}'\) on every
interval, sum and use Tonelli's theorem.  We obtain
\begin{equation}
 \E\|\rho_{m_n}\|_{C^1(I_n)}^2
 \leq
 Cn^{5/2-B^2/4}.
 \label{eq:full-line-low-band-C1-L2}
\end{equation}
We apply Markov's inequality to
\eqref{eq:full-line-low-band-C1-L2} and use \(B=20\) and
\(\varepsilon_n^{-2}=n^{76}\) to obtain
\[
 \Pp\!\left(
  \|\rho_{m_n}\|_{C^1(I_n)}
  \geq\varepsilon_n
 \right)
 \leq
 Cn^{157/2-B^2/4}
 =
 Cn^{-43/2}
 \leq
 Cn^{-21},
\]
because
\[
 \frac{157}{2}-\frac{B^2}{4}
 =
 \frac{157}{2}-100
 =
 -\frac{43}{2}.
\]
\end{proof}

\begin{proof}[Proof of Proposition~\ref{prop:full-line-low-band-localization}]
We apply Proposition~\ref{prop:logarithmic-edge-root-pairing} and
obtain
\begin{equation}
 \Pp\!\left(N_n(I_n)\neq N_\infty(I_n)\right)
 \leq Cn^{-16}.
 \label{eq:full-line-original-bulk-pairing}
\end{equation}
We compare \(P_{m_n}\) directly with \(P_\infty\).  We apply
Lemma~\ref{lem:comparison-near-critical-small-ball} with
\[
 I=I_n,
 \qquad
 p=36,
 \qquad
 M=\sqrt n,
 \qquad
 \delta=\delta_n,
 \qquad
 \varepsilon=\varepsilon_n.
\]
Since \(|I_n|\leq2\sqrt n\), the resulting probability bound for
simultaneous small values of \(q_\infty\) and \(q_\infty'\) is
\[
 Cn^{-35/2}+Cn^{-18}+Cn^{-38}
 \leq Cn^{-17}
\]
 and its endpoint estimate is \(Cn^{-38}\).  Let the good event be the
 intersection of \(\Omega_{\mathrm W}\) with the complements of the three bad
 events just described.  By \eqref{eq:full-line-low-band-C1}, the preceding
 simultaneous-small-value and endpoint estimates and
 \(\Pp(\Omega_{\mathrm W})=1\), its complement has probability at most
 \(Cn^{-17}\).  On the good event, the assumptions of
Lemma~\ref{lem:comparison-order-preserving-stability} hold on \(I_n\)
with
\[
 f=q_\infty,
 \qquad
 h=q_{m_n},
 \qquad
 \delta=\delta_n,
 \qquad
 \varepsilon=\varepsilon_n.
\]
We apply Lemma~\ref{lem:comparison-order-preserving-stability} and obtain that
the zero sets of \(q_\infty\) and \(q_{m_n}\) in \(I_n\) have the same
cardinality.  We use transversality and the \(C^1\) bound to conclude that
every zero of either function on \(I_n\) is simple.  Thus these cardinalities
equal the corresponding counts with multiplicity, and multiplication by
\(e^{-x^2/2}\) does not change zeros.  Consequently,
\[
 N_{P_{m_n}}(I_n)=N_\infty(I_n)
\]
on the good event.  We combine this conclusion with
\eqref{eq:full-line-original-bulk-pairing} and use a union bound to
prove \eqref{eq:full-line-low-band-localization}.
\end{proof}

We next compare \(N_n(I_n^{\mathrm{out}})\) with
\(N_{P_n-P_{m_n}}(I_n^{\mathrm{out}})\).  By
\eqref{eq:full-line-exterior-reciprocal-identity}, it suffices to compare
\(H_n\) and \(\widetilde H_n\) on \([0,1]\).  We begin with the following
\(C^1\) estimate.

\begin{lemma}
\label{lem:full-line-upper-band-C1}
For all sufficiently large \(n\),
\begin{equation}
 \Pp\!\left(
  \|H_n-\widetilde H_n\|_{C^1([0,1])}
  \geq\varepsilon_n
 \right)
 \leq Cn^{-19}.
 \label{eq:full-line-upper-band-C1}
\end{equation}
\end{lemma}

\begin{proof}
We combine Equation~\eqref{eq:full-line-reciprocal-Gaussian-upper}
with \(j(j-1)\geq j^2/2\), valid for \(j\geq2\), and obtain
\[
 \kappa_{n,j}^2\leq e^{-j^2/(4n)},
 \qquad 2\leq j\leq n.
\]
We have \(2\leq d_n\leq n\), and for \(d_n\leq j\leq n\),
\begin{equation}
 \kappa_{n,j}^2
 \leq
 e^{-j^2/(4n)}
 \leq
 n^{-B^2/4}
 =
 n^{-100}.
 \label{eq:full-line-upper-band-coefficient-tail}
\end{equation}
We use the falling-factorial convention, independence and Tonelli's
theorem together with \eqref{eq:full-line-upper-band-coefficient-tail}
to obtain, for
\(r\in\{0,1,2\}\),
\begin{equation}
 \E\int_0^1|(H_n-\widetilde H_n)^{(r)}(u)|^2\dd u
 =
 \sum_{j=d_n}^n(j)_r^2\kappa_{n,j}^2
 \int_0^1u^{2(j-r)}\dd u
 \leq
 n^{2r+1-B^2/4}.
 \label{eq:full-line-upper-band-derivative-L2}
\end{equation}
We apply Lemma~\ref{lem:comparison-interval-Sobolev} with \(p=2\)
to \(H_n-\widetilde H_n\) and
\((H_n-\widetilde H_n)'\) on \([0,1]\) and then use
\eqref{eq:full-line-upper-band-derivative-L2} to obtain
\[
 \E\|H_n-\widetilde H_n\|_{C^1([0,1])}^2
 \leq
 Cn^{5-B^2/4}.
\]
We apply Markov's inequality and use \(B=20\) and
\(\varepsilon_n^{-2}=n^{76}\) to obtain
\[
 \Pp\!\left(
  \|H_n-\widetilde H_n\|_{C^1([0,1])}
  \geq\varepsilon_n
 \right)
 \leq
 Cn^{5-B^2/4+76}
 =
 Cn^{-19},
\]
where
\[
 5-\frac{B^2}{4}+76=5-100+76=-19.
\]
\end{proof}

To turn this \(C^1\) approximation into equality of zero counts, we also
require uniform transversality and endpoint separation for \(H_n\).

\begin{proposition}
\label{prop:full-line-positive-transversality}
There is an event \(\mathcal R_n\) such that
\begin{equation}
 \Pp(\mathcal R_n^c)
 \leq Cn^{-17}
 \label{eq:full-line-positive-transversality-probability}
\end{equation}
and, on \(\mathcal R_n\),
\begin{equation}
 \inf_{u\in[0,1]}
 \max\{|H_n(u)|,|H_n'(u)|\}
 >
 2\delta_n
 \label{eq:full-line-positive-transversality}
\end{equation}
and
\begin{equation}
 \min\{|H_n(0)|,|H_n(1)|\}
 >
 2\varepsilon_n.
 \label{eq:full-line-positive-endpoints}
\end{equation}
\end{proposition}

\begin{proof}
For \(r\in\{0,1,2,3\}\), we apply
\eqref{eq:full-line-reciprocal-Gaussian-upper} and obtain
\begin{equation}
 \sup_{u\in[0,1]}
 \Var(H_n^{(r)}(u))
 \leq
  \sum_{j=r}^n
 (j)_r^2e^{-j(j-1)/(2n)}
 \leq
 C_rn^{r+1/2}.
 \label{eq:full-line-reciprocal-derivative-variance}
\end{equation}
Indeed, for \(j\geq\max\{2,r\}\), we bound the summand by
\(j^{2r}e^{-j^2/(4n)}\).  We compare with the corresponding integral
and substitute \(x=\sqrt n\,y\) to obtain \(O_r(n^{r+1/2})\).  Since
\(\kappa_{n,0}=\kappa_{n,1}=1\), the remaining indices \(j=0,1\),
when present, contribute at most \(2\).

We take \(p=72\).  We combine the Gaussian even-moment identity with
\eqref{eq:full-line-reciprocal-derivative-variance} and we then apply
Lemma~\ref{lem:comparison-interval-Sobolev} to
\(H_n'\) and \(H_n''\).  We obtain
\begin{equation}
 \E\left[
  \sup_{u\in[0,1]}
  \bigl(
   |H_n'(u)|+|H_n''(u)|
  \bigr)^p
 \right]
 \leq
 C_pn^{7p/4}.
 \label{eq:full-line-reciprocal-derivative-supremum-moment}
\end{equation}
We apply Markov's inequality at \(n^2\) to
\eqref{eq:full-line-reciprocal-derivative-supremum-moment} and obtain
\begin{equation}
 \Pp\!\left(
  \sup_{u\in[0,1]}
  \bigl(
   |H_n'(u)|+|H_n''(u)|
  \bigr)>n^2
 \right)
 \leq Cn^{-18}.
 \label{eq:full-line-reciprocal-derivative-good-event}
\end{equation}

For every \(u\in[0,1]\), the \(j=0,1\) terms contribute
\[
 A(u)=
 \begin{pmatrix}
  1+u^2&u\\
  u&1
 \end{pmatrix}
\]
to the covariance matrix of
\((H_n(u),H_n'(u))\).  Since
\[
 \det A(u)=1,
 \qquad
 \operatorname{tr}A(u)\leq3,
\]
its smallest eigenvalue is at least \(1/3\).  Every remaining
coefficient adds a positive semidefinite matrix.  Hence the joint
Gaussian density is at most \(3/(2\pi)\), uniformly in \(u,n\).

On the complement of the event in
\eqref{eq:full-line-reciprocal-derivative-good-event}, we divide
\([0,1]\) into
\(M_n=\lceil n^2/\delta_n\rceil\) equal subintervals and use their
endpoints as a grid.  If the
left-hand side of
\eqref{eq:full-line-positive-transversality} were at most
\(2\delta_n\), a nearest grid point \(v\) would satisfy
\[
 |H_n(v)|\leq3\delta_n,
 \qquad
 |H_n'(v)|\leq3\delta_n.
\]
We combine the uniform density bound with a union bound over at most
\(2+n^2/\delta_n\) grid points and obtain
\begin{equation}
 \begin{aligned}
 &\Pp\!\left(
  \inf_{u\in[0,1]}
  \max\{|H_n(u)|,|H_n'(u)|\}
  \leq2\delta_n
 \right)\\
 {}\leq{}&
 Cn^{-18}
 +
 C\left(2+\frac{n^2}{\delta_n}\right)\delta_n^2
 \leq
 Cn^{-17}.
 \end{aligned}
 \label{eq:full-line-reciprocal-near-critical-probability}
\end{equation}
Here the leading grid contribution is
\(n^2\delta_n=n^{-17}\).

At both endpoints, the variance of \(H_n\) is at least one because
of its constant term.  We apply the one-dimensional Gaussian density
bound and obtain
\begin{equation}
 \Pp\!\left(
  \min\{|H_n(0)|,|H_n(1)|\}
  \leq2\varepsilon_n
 \right)
 \leq C\varepsilon_n.
 \label{eq:full-line-reciprocal-endpoint-probability}
\end{equation}
We define \(\mathcal R_n\) as the complement of the two bad events in
\eqref{eq:full-line-reciprocal-near-critical-probability} and
\eqref{eq:full-line-reciprocal-endpoint-probability}.  We apply a union
bound to prove \eqref{eq:full-line-positive-transversality-probability}.
\end{proof}

\begin{proof}[Proof of Proposition~\ref{prop:full-line-upper-band-localization}]
We define \(\mathcal A_n\) as the intersection of
\(\mathcal R_n\) with
\[
 \left\{
  \|H_n-\widetilde H_n\|_{C^1([0,1])}
  <\varepsilon_n
 \right\}.
\]
On \(\mathcal A_n\), we apply
Lemma~\ref{lem:comparison-order-preserving-stability} on \([0,1]\)
with
\[
 f=H_n,
 \qquad
 h=\widetilde H_n,
 \qquad
 \delta=\delta_n,
 \qquad
 \varepsilon=\varepsilon_n.
\]
We use the endpoint estimate \eqref{eq:full-line-positive-endpoints}
to conclude that neither polynomial vanishes at \(0\) or \(1\).
 By transversality and the \(C^1\) bound, every zero of either polynomial on
 \([0,1]\) is simple.  We then apply
 Lemma~\ref{lem:comparison-order-preserving-stability} to obtain
\[
 N_{H_n}((0,1))
 =
 N_{\widetilde H_n}((0,1))
\]
on \(\mathcal A_n\).  We combine the probability bounds
\eqref{eq:full-line-upper-band-C1} and
\eqref{eq:full-line-positive-transversality-probability} and apply a
union bound to obtain
\begin{equation}
 \Pp\!\left(
  N_{H_n}((0,1))
  \neq
  N_{\widetilde H_n}((0,1))
 \right)
 \leq Cn^{-17}.
 \label{eq:full-line-positive-upper-band-localization}
\end{equation}

By \eqref{eq:full-line-pair-symmetry}, the bound in
\eqref{eq:full-line-positive-upper-band-localization} also holds with
\((0,1)\) replaced by \((-1,0)\).  Since
\(H_n(0)=\widetilde H_n(0)=\xi_n\neq0\) almost surely,
\[
 \begin{aligned}
 \{N_{H_n}((-1,1))\neq N_{\widetilde H_n}((-1,1))\}
 &\subset
 \{N_{H_n}((0,1))\neq N_{\widetilde H_n}((0,1))\}\\
 &\phantom{\subset{}}\cup
 \{N_{H_n}((-1,0))\neq N_{\widetilde H_n}((-1,0))\}.
 \end{aligned}
\]
 We apply a union bound and use
 \eqref{eq:full-line-exterior-reciprocal-identity} to obtain
 \eqref{eq:full-line-upper-band-localization}.

Almost surely, \(P_n-P_{m_n}\not\equiv0\).  It then has a zero of
multiplicity at least \(m_n+1\) at the origin and at most \(d_n-1\)
nonzero zeros.  Hence, almost surely,
\[
 0\leq N_n(I_n^{\mathrm{out}})\leq n,
 \qquad
 0\leq N_{P_n-P_{m_n}}(I_n^{\mathrm{out}})
 \leq d_n-1\leq n.
\]
Therefore
\[
 |N_n(I_n^{\mathrm{out}})
   -N_{P_n-P_{m_n}}(I_n^{\mathrm{out}})|
 \leq
 n\,
 \mathbf1_{\{
  N_n(I_n^{\mathrm{out}})
  \neq N_{P_n-P_{m_n}}(I_n^{\mathrm{out}})
 \}}.
\]
We apply \eqref{eq:full-line-upper-band-localization} and obtain
\[
 \|N_n(I_n^{\mathrm{out}})
   -N_{P_n-P_{m_n}}(I_n^{\mathrm{out}})\|_{L^2}
 \leq
 Cn(n^{-17})^{1/2}
 =
 Cn^{-15/2}.
\]

We finally use
\(\|Y-\E Y\|_{L^2}\leq2\|Y\|_{L^2}\),
Proposition~\ref{prop:full-line-exterior-second-moment} and
\eqref{eq:full-line-upper-band-coupling-L2}.  We apply the triangle
inequality and obtain
\[
 \begin{aligned}
 &\|N_{P_n-P_{m_n}}(I_n^{\mathrm{out}})
   -\E N_{P_n-P_{m_n}}(I_n^{\mathrm{out}})\|_{L^2}\\
 {}\leq{}&
 \|N_n(I_n^{\mathrm{out}})
   -\E N_n(I_n^{\mathrm{out}})\|_{L^2}\\
 &\mathord{+}\,
 \|(N_{P_n-P_{m_n}}(I_n^{\mathrm{out}})
    -N_n(I_n^{\mathrm{out}}))
   -\E(N_{P_n-P_{m_n}}(I_n^{\mathrm{out}})
    -N_n(I_n^{\mathrm{out}}))\|_{L^2}\\
 {}\leq{}&
 C\log n.
 \end{aligned}
\]
This proves all assertions.
\end{proof}

\subsection{Proofs of Lemmas~\ref{lem:full-line-independent-transfer}
and~\ref{lem:full-line-additive-transfer}}

\begin{proof}[Proof of Lemma~\ref{lem:full-line-independent-transfer}]
We condition on \(V\) and use the independence of \(U\) and \(V\).
For every \(z\in\R\), we obtain
\[
 \begin{aligned}
 &\left|
   \Pp\!\left(
    \frac{U+V-\E(U+V)}{\sqrt{\Var(U)+\Var(V)}}\leq z
   \right)
   -
   \E\Phi\!\left(
    \sqrt{1+\frac{\Var(V)}{\Var(U)}}\,z
    -\frac{V-\E V}{\sqrt{\Var(U)}}
   \right)
  \right|\\
 {}\leq{}&
  d_{\mathrm K}\!\left(
   \frac{U-\E U}{\sqrt{\Var(U)}},
   Z
  \right).
 \end{aligned}
\]
We combine Taylor's theorem about
\(\sqrt{1+\Var(V)/\Var(U)}\,z\) with the centering of \(V\) and
\[
 \sup_{x\in\R}|\Phi''(x)|
 =
 \frac{1}{\sqrt{2\pi e}}
\]
to obtain
\[
 \left|
  \E\Phi\!\left(
   \sqrt{1+\frac{\Var(V)}{\Var(U)}}\,z
   -\frac{V-\E V}{\sqrt{\Var(U)}}
  \right)
  -
  \Phi\!\left(
   \sqrt{1+\frac{\Var(V)}{\Var(U)}}\,z
  \right)
 \right|
 \leq
 \frac{1}{2\sqrt{2\pi e}}
 \frac{\Var(V)}{\Var(U)}.
\]
The vanishing first-order term is what replaces a
standard-deviation ratio by a variance ratio.  We combine the mean
value theorem with \(\sqrt{1+x}-1\leq x/2\) for \(x\geq0\) and obtain
\[
 \left|
  \Phi\!\left(
   \sqrt{1+\frac{\Var(V)}{\Var(U)}}\,z
  \right)
  -
  \Phi(z)
 \right|
 \leq
 \frac{\sqrt{1+\Var(V)/\Var(U)}-1}{\sqrt{2\pi e}}
 \leq
 \frac{1}{2\sqrt{2\pi e}}
 \frac{\Var(V)}{\Var(U)}.
\]
We add the three estimates and take the supremum over \(z\) to prove
\eqref{eq:full-line-independent-transfer}.
\end{proof}

\begin{proof}[Proof of Lemma~\ref{lem:full-line-additive-transfer}]
We combine the centered identity
\[
 Y-\E Y=(X-\E X)+(T-\E T)
\]
with the reverse triangle inequality in \(L^2\) to obtain
\begin{equation}
 \left|\sqrt{\Var(Y)}-\sqrt{\Var(X)}\right|
 \leq
 \|T-\E T\|_{L^2}.
 \label{eq:full-line-additive-sd-comparison}
\end{equation}
Thus
\[
\frac12\sqrt{\Var(X)}
\leq
\sqrt{\Var(Y)}
\leq
\frac32\sqrt{\Var(X)},
\]
so \(\Var(Y)>0\).  For every \(z\in\R\), we combine the bound
\(|T-\E T|\leq r\) on \(G\) with \(\Pp(G^c)\leq p\) and obtain
\begin{equation}
 \begin{aligned}
 &\Pp\!\left(
 \frac{X-\E X}{\sqrt{\Var(X)}}
 \leq
 \sqrt{\frac{\Var(Y)}{\Var(X)}}\,z
 -\frac{r}{\sqrt{\Var(X)}}
 \right)-p\\
 {}\leq{}&
 \Pp\!\left(
  \frac{Y-\E Y}{\sqrt{\Var(Y)}}
  \leq z
 \right)\\
 {}\leq{}&
 \Pp\!\left(
  \frac{X-\E X}{\sqrt{\Var(X)}}
  \leq
  \sqrt{\frac{\Var(Y)}{\Var(X)}}\,z
  +\frac{r}{\sqrt{\Var(X)}}
 \right)+p.
 \end{aligned}
 \label{eq:full-line-additive-probability-sandwich}
\end{equation}
We use \eqref{eq:full-line-additive-sd-comparison} to obtain
\[
 \sqrt{\frac{\Var(Y)}{\Var(X)}}\in[1/2,3/2],
 \qquad
 \left|
  \sqrt{\frac{\Var(Y)}{\Var(X)}}-1
 \right|
 \leq
 \frac{\|T-\E T\|_{L^2}}{\sqrt{\Var(X)}}.
\]
For either choice of sign, we apply the mean value theorem and use the
calculation in the proof of
Lemma~\ref{lem:finite-range-exact-coupling-transfer} to obtain
\[
 \begin{aligned}
 &\sup_{z\in\R}
  \left|
   \Phi\!\left(
    \sqrt{\frac{\Var(Y)}{\Var(X)}}\,z
    \mathbin{\pm}\frac{r}{\sqrt{\Var(X)}}
   \right)
   -
   \Phi(z)
  \right|\\
 {}\leq{}&
  \frac{\|T-\E T\|_{L^2}}{\sqrt{\Var(X)}}
  +
  \frac{r}{\sqrt{2\pi\Var(X)}}
  \leq
  \frac{\|T-\E T\|_{L^2}+r}{\sqrt{\Var(X)}}.
 \end{aligned}
\]
We apply the definition of \(d_{\mathrm K}\) to the two outer
probabilities in
\eqref{eq:full-line-additive-probability-sandwich} and take the
supremum over \(z\) to obtain \eqref{eq:full-line-additive-transfer}.
\end{proof}

This completes the proofs of the additional results used in
Section~\ref{sec:proof-full-line-main-theorem}.

\appendix

\section{External results used in the proofs}
\label{app:external-results}

We collect here the external results used in the preceding arguments,
grouping them by source in the order of their first use.  The statements
retain their original hypotheses, conclusions and notation, apart from
minor typographical adjustments.  All notation introduced within an
appended lemma is local to that statement.  Accordingly, the same symbol
may have a different meaning in the main text or in another appended
lemma without redefining any main-text notation.  Whenever a source
symbol is identified with an object used in this paper, we state that
identification explicitly.

The following result gives the variance asymptotics and asymptotic
Gaussianity of linear statistics of the real zeros of Gaussian Weyl
polynomials on growing windows.

\begin{lemma}
\label{lem:appendix-do-vu-thm5}
\textup{\cite[Theorem~5]{DoVu2020}}.
Let \(\mathcal Z_n\) be the multiset of real zeros of the Gaussian Weyl
polynomial \(P_n\).  There is a finite constant \(K>0\) such that the
following holds.  Let \(h:\R\to\R\) be bounded, nonzero and supported
on \([-1,1]\).  Suppose that \(h\) has finitely many discontinuities
and is H\"older continuous on each interval in the partition of
\([-1,1]\) determined by those discontinuities.  If
\(R_n\to\infty\) and
\[
 R_n\leq n^{1/2}+o(n^{1/4}),
\]
set
\[
 \mathcal N_n(h,R_n)
 =
 \sum_{x\in\mathcal Z_n}h(x/R_n).
\]
Then
\[
 \lim_{n\to\infty}
 \frac{\Var(\mathcal N_n(h,R_n))}
 {R_n\|h\|_{L^2(\R)}^2}
 =K,
\]
and
\[
 \frac{\mathcal N_n(h,R_n)-\E\mathcal N_n(h,R_n)}
 {\sqrt{\Var(\mathcal N_n(h,R_n))}}
 \xrightarrow{\mathrm d}\mathcal N(0,1).
\]
The constant \(K\) is the constant denoted by \(V_\infty\) in the
present paper.
\end{lemma}

The next result is the counterpart of
Lemma~\ref{lem:appendix-do-vu-thm5} for the full Weyl series.  It supplies
variance asymptotics,
asymptotic Gaussianity and moment bounds for linear statistics of the
zeros of the full Weyl series.

\begin{lemma}
\label{lem:appendix-do-vu-thm6}
\textup{\cite[Theorem~6]{DoVu2020}}.
Let \(\mathcal Z_\infty\) be the multiset of real zeros of the full
Gaussian Weyl series \(P_\infty\).  For a bounded, nonzero,
compactly supported function \(h:\R\to\R\) and \(R>0\), set
\[
 \mathcal N_\infty(R,h)
 =
 \sum_{x\in\mathcal Z_\infty}h(x/R).
\]
There is a finite constant \(K>0\) such that
\[
 \lim_{R\to\infty}
 \frac{\Var(\mathcal N_\infty(R,h))}
 {R\|h\|_{L^2(\R)}^2}
 =K
\]
and
\[
 \frac{\mathcal N_\infty(R,h)-\E\mathcal N_\infty(R,h)}
 {\sqrt{\Var(\mathcal N_\infty(R,h))}}
 \xrightarrow{\mathrm d}\mathcal N(0,1).
\]
Moreover, for every integer \(k\geq1\), there is a finite constant
\(C_{h,k}\) such that, for \(R\geq1\),
\[
 \E\bigl[\mathcal N_\infty(R,h)^k\bigr]
 \leq C_{h,k}R^k.
\]
The constant \(K\) is the same as in
Lemma~\ref{lem:appendix-do-vu-thm5}.
\end{lemma}

The next result is a version of Bulinskaya's lemma.  It rules out a
critical point at a prescribed level under a locally uniform density
bound.

\begin{lemma}
\label{lem:appendix-azais-wschebor-prop1-20}
\textup{\cite[Proposition~1.20]{AzaisWschebor2009}}.
Let \(X=\{X(t):t\in I\}\) be a stochastic process with \(C^1\)
sample paths on an interval \(I\subset\R\) and fix \(u\in\R\).
Suppose that, for every \(t\in I\), the random variable \(X(t)\) has
a density \(p_{X(t)}(x)\) that is bounded as \(t\) ranges over a
compact subset of \(I\) and \(x\) ranges over a neighborhood of
\(u\).  Then
\[
 \Pp\bigl(\{t\in I:X(t)=u,\ X'(t)=0\}\neq\varnothing\bigr)=0.
\]
\end{lemma}

The Gaussian Rice formula below expresses factorial moments of the
number of level crossings in terms of joint densities and conditional
derivative moments.

\begin{lemma}
\label{lem:appendix-azais-wschebor-thm3-2}
\textup{\cite[Theorem~3.2]{AzaisWschebor2009}}.
Let \(X=\{X(t):t\in I\}\) be a Gaussian process with \(C^1\)
sample paths on an interval \(I\subset\R\) and let \(k\geq1\) be an
integer.  Suppose that, for every \(k\) pairwise distinct points
\(t_1,\ldots,t_k\in I\), the Gaussian vector
\((X(t_1),\ldots,X(t_k))\) is nondegenerate.  For \(u\in\R\), let
\(N_u=N_u(X,I)\) be the number of points \(t\in I\) at which
\(X(t)=u\) and define
\[
 m^{[k]}=
 \begin{cases}
  m(m-1)\cdots(m-k+1),&m\geq k,\\
  0,&m<k.
 \end{cases}
\]
Then
\[
 \begin{aligned}
 \E\bigl[N_u^{[k]}\bigr]
 ={}&
 \int_{I^k}
 \E\!\left[
  \left|X'(t_1)\cdots X'(t_k)\right|
  \,\middle|\,
  X(t_1)=\cdots=X(t_k)=u
 \right]\\
 &\qquad\times
 p_{X(t_1),\ldots,X(t_k)}(u,\ldots,u)
 \dd t_1\cdots\dd t_k.
 \end{aligned}
\]
\end{lemma}

The next result gives a jet-nondegeneracy criterion ensuring finite
moments for zero counts of Gaussian fields.

\begin{lemma}
\label{lem:appendix-gass-stecconi-thm1-1}
\textup{\cite[Theorem~1.1]{GassStecconi2024}}.
Let \(p,d\) be positive integers and let \(\mathcal U\) be an open
subset of \(\R^d\), respectively of \(\mathbb C^d\).  Let \(F\) be
a Gaussian random field from \(\mathcal U\) to \(\R^d\), respectively
to \(\mathbb C^d\), with almost surely \(C^p\), respectively
holomorphic, sample paths.  Suppose that, for every
\(x\in\mathcal U\), the Gaussian vector
\[
 \bigl(\partial^\alpha F(x)\bigr)_{|\alpha|\leq p-1}
\]
is nondegenerate.  If
\(Z(F,K)=\{x\in K:F(x)=0\}\), then every compact set
\(K\subset\mathcal U\) satisfies
\[
 \E\bigl[\#Z(F,K)^p\bigr]<\infty.
\]
\end{lemma}

The accompanying statement shows that these zero-count moments pass
to limits under convergence in the relevant function-space topology.

\begin{lemma}
\label{lem:appendix-gass-stecconi-rem1-3}
\textup{\cite[Remark~1.3]{GassStecconi2024}}.
In the setting of Lemma~\ref{lem:appendix-gass-stecconi-thm1-1}, the
constant bounding the \(p\)th moment of \(\#Z(F,K)\) is a bounded
functional of the associated Gaussian field as a \(C^p\), respectively
holomorphic, random function.  In particular, if \((F_n)_{n\geq0}\)
is a sequence of Gaussian fields converging in distribution to \(F\)
in the \(C^p\) topology, respectively in the topology of uniform
convergence, on \(\mathcal U\), then
\[
 \lim_{n\to\infty}
 \E\bigl[\#Z(F_n,K)^p\bigr]
 =
 \E\bigl[\#Z(F,K)^p\bigr].
\]
The analogous assertion holds for the critical-point result in
\cite[Theorem~1.2]{GassStecconi2024}.
\end{lemma}

The final external result gives uniform and nonuniform bounds for Gaussian
approximation of sums of \(m\)-dependent random fields.

\begin{lemma}
\label{lem:appendix-chen-shao-thm2-6}
\textup{\cite[Theorem~2.6]{ChenShao2004}}.
Let \(d\geq1\), let \(J\) be a finite subset of
\(\mathbb Z_{>0}^d\) and equip \(\mathbb Z_{>0}^d\) with the distance
\[
 |i-j|_\infty=\max_{1\leq r\leq d}|i_r-j_r|.
\]
For subsets \(A,B\subset\mathbb Z_{>0}^d\), set
\[
 \rho(A,B)=\inf\{|i-j|_\infty:i\in A,\ j\in B\}.
\]
A random field \(\{X_i:i\in J\}\) is \(m\)-dependent if
\(\{X_i:i\in A\}\) and \(\{X_j:j\in B\}\) are independent whenever
\(A,B\subset J\) and \(\rho(A,B)>m\).

Suppose that \(\{X_i:i\in J\}\) is an \(m\)-dependent random field,
\(\E X_i=0\) and \(\E|X_i|^p<\infty\) for every \(i\in J\), where
\(2<p\leq3\).  Let
\[
 W=\sum_{i\in J}X_i,
 \qquad
 \Var(W)=1,
\]
let \(F_W\) be the distribution function of \(W\) and let \(\Phi\)
be the standard Gaussian distribution function.  Then
\[
 \sup_{z\in\R}|F_W(z)-\Phi(z)|
 \leq
 75(10m+1)^{(p-1)d}
 \sum_{i\in J}\E|X_i|^p.
\]
Moreover, for every \(z\in\R\),
\[
 |F_W(z)-\Phi(z)|
 \leq
 C(1+|z|)^{-p}19^{pd}(m+1)^{(p-1)d}
 \sum_{i\in J}\E|X_i|^p,
\]
where \(C\) is a finite absolute constant.
\end{lemma}

\raggedbottom

\end{document}